\documentclass[a4paper,12pt,reqno]{amsart}
\pdfoutput=1

\usepackage[foot]{amsaddr}
\usepackage[flushmargin]{footmisc} 
\usepackage[margin=1in]{geometry}
\usepackage{hyperref}
\usepackage{amsmath,mathtools}
\usepackage{enumerate}
\usepackage{tikz}
\usepackage[all]{xy}
\usepackage{verbatim}
\usepackage{setspace}
\usepackage{enumitem}
\usepackage{mathdots}
\usepackage{longtable}
\usepackage{amssymb}
\usetikzlibrary{calc}

\xyoption{rotate}

\DeclareMathOperator{\Hom}{Hom}

\DeclareMathOperator{\End}{End}

\DeclareMathOperator{\Ext}{Ext}

\DeclareMathOperator{\sgn}{sgn}

\DeclareMathOperator{\mod*}{mod}

\DeclareMathOperator{\Coker}{Coker}

\DeclareMathOperator{\dimvect}{\mathbf{dim}}

\DeclareMathOperator{\Frac}{Frac}

\newcommand{\dynA}{A}
\newcommand{\dynB}{B}
\newcommand{\dynC}{C}
\newcommand{\dynD}{D}
\newcommand{\dynE}{E}
\newcommand{\dynF}{F}
\newcommand{\dynG}{G}

\newcommand{\coxA}{\dynA}
\newcommand{\coxB}{\dynB}

\newcommand{\coxD}{\dynD}
\newcommand{\coxE}{\dynE}

\newcommand{\coxG}{\dynG}
\newcommand{\coxI}{I}
\newcommand{\coxH}{H}

\newcommand{\gend}{\Delta}	% General diagram

\newcommand{\cv}{c}

\newcommand{\gv}{g}
\newcommand{\rootsys}{\Phi}

\newcommand{\Isroots}[1]{\rootsys^{s}_{#1}}
\newcommand{\Ilroots}[1]{\rootsys^{l}_{#1}}
\newcommand{\Iproots}[1]{\rootsys^+_{#1}}

\newcommand{\Isproots}[1]{\rootsys^{s,+}_{#1}}
\newcommand{\Ilproots}[1]{\rootsys^{l,+}_{#1}}

\newcommand{\acat}{\mathcal{A}}
\newcommand{\der}{\mathcal{D}}
\newcommand{\bder}{\der^b}

\newcommand{\clus}{\mathcal{C}}
\newcommand{\gen}{\mathcal{G}}

\newcommand{\calM}{\mathcal{M}}
\newcommand{\sus}{\Sigma}
\newcommand{\dart}{\tau_{\bder_\gend}}

\newcommand{\vw}{\kappa} % vertex weight
\newcommand{\aw}{\xi} % arrow weight
\newcommand{\mw}{\varsigma} % modulated weight
\newcommand{\dimproj}{\delta}

\newcommand{\op}{^{\text{op}}}
\newcommand{\Iiin}{{\coxI_2(2n)}}
\newcommand{\gratio}{\varphi}
\newcommand{\croot}{\psi}

\newcommand{\imunit}{\mathrm{i}}

\newcommand{\field}{K}
\newcommand{\real}{\mathbb{R}}

\newcommand{\integer}{\mathbb{Z}}

\newcommand{\nnint}{\mathbb{Z}_{\geq 0}}
\newcommand{\rhom}[1]{\sigma^{#1}}
\newcommand{\hrhom}[1]{\widehat{\sigma}^{#1}}
\newcommand{\realhom}[1]{\sigma^{(#1)}}
\newcommand{\srhom}[1]{\rhom{#1}_+}
\newcommand{\chebr}[1]{\chi^{#1}}
\newcommand{\hchebr}[1]{\wh\chi^{#1}}
\newcommand{\achebr}[1]{\chebr{\coxA_{#1}}}
\newcommand{\dchebr}[1]{\chebr{\coxD_{#1}}}
\newcommand{\echebr}[1]{\chebr{\coxE_{#1}}}
\newcommand{\hachebr}[1]{\hchebr{\coxA_{#1}}}
\newcommand{\hdchebr}[1]{\hchebr{\coxD_{#1}}}
\newcommand{\hechebr}[1]{\hchebr{\coxE_{#1}}}
\newcommand{\chebsr}[1]{\chebr{#1}_+}
\newcommand{\hchebsr}[1]{\hchebr{#1}_+}
\newcommand{\achebsr}[1]{\chebsr{\coxA_{#1}}}
\newcommand{\dchebsr}[1]{\chebsr{\coxD_{#1}}}
\newcommand{\echebsr}[1]{\chebsr{\coxE_{#1}}}
\newcommand{\hachebsr}[1]{\hchebsr{\coxA_{#1}}}
\newcommand{\hdchebsr}[1]{\hchebsr{\coxD_{#1}}}
\newcommand{\hechebsr}[1]{\hchebsr{\coxE_{#1}}}
\newcommand{\ZUi}[1]{\chebr{(#1)}}
\newcommand{\regrep}[1]{\rho^{#1}}

\newcommand{\Dregrep}{\regrep{\coxD_{n+1}}}

\newcommand{\inv}{^{-1}}
\newcommand{\dbar}[1]{\bar{\bar{#1}}}

\newcommand{\wh}{\widehat}
\newcommand{\wt}{\widetilde}

\newcommand{\ps}[1]{^{(#1)}}

\newcommand{\tree}{\mathbb{T}}
\newcommand{\tikznode}[2]{
	\ifmmode
		\tikz[remember picture,baseline=(#1.base),inner sep=0pt] \node (#1) {$#2$};
	\else
		\tikz[remember picture,baseline=(#1.base),inner sep=0pt] \node (#1) {#2};
	\fi
}

\theoremstyle{plain}
\newtheorem{thm}{Theorem}[section]
\newtheorem*{thm*}{Theorem}
\newtheorem{lem}[thm]{Lemma}
\newtheorem{cor}[thm]{Corollary}
\newtheorem{prop}[thm]{Proposition}
\newtheorem*{prop*}{Proposition}

\theoremstyle{definition}
\newtheorem{defn}[thm]{Definition}

\newtheorem*{exam*}{Example}

\theoremstyle{remark}
\newtheorem{rem}[thm]{Remark}

\numberwithin{equation}{section}
\makeatletter
\renewcommand{\fnum@figure}{\textbf{Fig. \thefigure}}
\makeatother

\allowdisplaybreaks

\begin{document}
	\title[Categorifications of Non-Integer Quivers: Type $\Iiin$]{Categorifications of Non-Integer Quivers: \\ Type $\Iiin$}
	\author{Drew Damien Duffield and Pavel Tumarkin}
	\address{\noindent Durham University, Department of Mathematical and Computing Sciences, Durham, UK}
	\address{\noindent ddduffield2@gmail.com, pavel.tumarkin@durham.ac.uk}
	\address{\noindent ORCID: 0000-0002-1476-8147 and 0000-0002-9579-6077}
	\maketitle
	\begin{abstract}
		We use weighted unfoldings of quivers to provide a categorification of mutations of quivers of types $I_2(2n)$, thus extending the construction of categorifications of mutations of quivers to all finite types. 
	\end{abstract}
	
	{
	\noindent{\textbf{Corresponding author:} Drew Damien Duffield}
	
	~
	
	
	~
	
	
	\setcounter{tocdepth}{1}
	\tableofcontents
	
	\section*{Acknowledgements}
 We would like to thank Edmund Heng for helpful discussions. We would also like to thank the referee for conducting a thorough review of the paper, which included important corrections, suggesting numerous improvements to the exposition, and providing a detailed insight into the connections between fusion rings and our work. A part of the paper was written at the Isaac Newton Institute for Mathematical Sciences, Cambridge; we are grateful to the organisers of the programme “Cluster algebras and representation theory”, and to the Institute for support and hospitality during the programme; this work was supported by EPSRC grant no EP/R014604/1. In addition, research was supported by the Leverhulme Trust research grant RPG-2019-153.
 
 	\section*{Statements and Declarations}
	\subsection*{Ethical Approval} 
	Ethical approval is not applicable to this article as no human or animal studies were conducted during this work.
	
	\subsection*{Funding}
	Research was supported by funding from the Leverhulme Trust research grant RPG-2019-153. A part of this paper was written at the Isaac Newton Institute for Mathematical Sciences, Cambridge during the programme “Cluster algebras and representation theory”, which was supported by EPSRC grant no EP/R014604/1.
	
	\subsection*{Competing Interests} The author declares that there are no other financial or non-financial interests that are directly or indirectly related to the work submitted.
	}
	
	\subsection*{Data Availability} Data sharing not applicable to this article as no datasets were generated or analysed during the current study.

\section{Introduction and main results}\label{sec:Intro}
This is the second paper in a series started in~\cite{DTOdd}. In~\cite{DTOdd}, we constructed a categorification of mutations of non-integer quivers of finite types $H_4$, $H_3$ and $I_2(2n+1)$.  The main tool in the construction is a {\em weighted (un)folding} of quivers of types $E_8$, $D_6$ and $A_{2n}$, the application of which follows the projection of root systems developed in~\cite{Lusztig,WaveFrontsReflGroups,MoodyPatera}. The (un)foldings induce projections of dimension vectors of objects in module categories of integer quivers to the roots associated to folded quivers; they also induce semiring actions on categories associated to integer quivers. We also define the tropical seed patterns of folded quivers and show that our definition is consistent with both the folding and the categorical definition of $g$-vectors.

In this paper, we extend the results of~\cite{DTOdd} to the last remaining finite type of quivers, $I_2(2n)$, thus completing the construction of categorifications of mutations for quivers of all finite types. Unlike the settings of~\cite{DTOdd}, there is more than one possible folding that could be considered for every quiver of type $I_2(2n)$. More precisely, a quiver of type $I_2(2n)$ admits an unfolding to a quiver of type $A_{2n-1}$ and to a quiver of type $D_{n+1}$. Moreover, there are three exceptional foldings $E_6\to I_2({12})$, $E_7\to I_2({18})$, and $E_8\to I_2(30)$. We consider all of these different possible foldings to construct categorifications via semiring actions on the categories associated to different integer quivers, and we see that aside from the semiring action itself, the theory between these different foldings is very much the same. The key upshot of this is that the module, bounded derived, and cluster categories associated to any bipartite Dynkin quiver has a non-crystallographic interpretation via a semiring action and the associated non-integer quiver. This includes the ability to associate non-integer $g$-vectors and mutation to the categories of any bipartite Dynkin quiver.

Another difference to the setting of~\cite{DTOdd} is that for the foldings considered in this paper, it is more natural to consider the root system of $I_2(2n)$ in such a way that it contains roots of two different lengths. We therefore use non-trivial rescaling throughout the paper. This approach to the root system generalises how one classically considers the crystallographic root systems $\coxI_2(4)=\coxB_2$ and $\coxI_2(6)=\coxG_2$ to higher $n$. It is entirely possible to categorify the root system of $I_2(2n)$ via an unfolding with a trivial rescaling (and thus with all roots of $I_2(2n)$ of the same length) and a projection map from the categories associated to the unfolding. However, when it comes to considering the semiring action on the cluster category of an unfolded Dynkin quiver, the $g$-vectors one obtains are precisely those that are associated to the non-trivial rescaling, and thus we consider this approach to be the most fitting.

We recall all necessary definitions and details of the construction from~\cite{DTOdd} in Section~\ref{sec:Prelim}.

Our first main result is Theorem~\ref{thm:FoldingProjection} (see also Corollaries~\ref{cor:RowBijection} and~\ref{cor:RowWeights}), in which we prove that, given a weighted folding   $F\colon Q^{\gend} \rightarrow Q^{I_2(2n)}$, there is a weighting on the rows of the Auslander-Reiten quiver of $\mod* KQ^{\gend}$ and $\bder(\mod*KQ^\gend)$ consistent with $F$, with the projection of the dimension vectors of objects, and with the projection of the corresponding root systems. Objects in the rows of weight 1 in the Auslander-Reiten quiver map precisely to the roots of $\Iiin$, thus generalising Gabriel's Theorem for quivers of type $\Iiin$. In fact, Theorem~\ref{thm:FoldingProjection} says much more than this. The folding $F$ has a strong relationship with the structure of the Auslander-Reiten quiver. For example, objects that reside in the same column of the Auslander-Reiten quiver project onto a multiple of the same root, and these multiples are determined by the weights of the rows. Moreover, Auslander-Reiten translation acts by rotation with respect to the projection.

The projection map induced by the folding gives rise to a semiring action on the module and bounded derived categories of the integer quivers, where the semiring (we denote it by $R_+$) is defined separately for each folding (Section~\ref{sec:Prelim-Cheb}) using Chebyshev polynomials. The main results here (Theorem~\ref{thm:RPGenerators} and Corollary~\ref{cor:GenCorrespondence}) state that $\mod*KQ^\gend$ has an action of the corresponding semiring, and there exists a collection of indecomposable objects of $\mod*KQ^\gend$ that are in bijection with the positive roots of $I_2(2n)$, and these objects generate the whole category under the action of the semiring. These results naturally extend to the bounded derived category $\bder(\mod*KQ^\gend)$.

   The semiring action also applies to the cluster category of $Q^\gend$, and this allows us to extend the results of~\cite{BMRRT} to categorify mutations of folded quivers by considering distinguished objects with respect to the action, which we call $R_+$-{\em tilting objects}  (Section~\ref{sec:ClusterTilting}). We prove (Theorem~\ref{thm:RPTilting} and Corollary~\ref{cor:Complements}) that for a folding $F\colon Q^{\gend} \rightarrow Q^{I_2(2n)}$, the following hold for the cluster category $\clus_\gend$ of $Q^\gend$: every basic $R_+$-tilting object $T \in \clus_\gend$ (injectively) corresponds to a basic tilting object $\wh{T} \in \clus_\gend$; every basic $R_+$-tilting object has precisely $2$ indecomposable direct summands; every almost complete $R_+$-tilting object has exactly two complements, and changing the complement corresponds to a single mutation; if $T$ is a basic $R_+$-tilting object, and $A_T$ is the cluster-tilted algebra corresponding to $\wh{T}$, then there exists an $R_+$-action on $\mod*A_T$.  

     Finally, we define the tropical seed patterns of folded quivers. We define $\cv$-vectors and $C$-matrices as in the integer case, and we show that $\cv$-vectors of $Q^{\Iiin}$ are roots of $\Iiin$ and are thus sign-coherent (Corollary~\ref{cor:MiddleLeftSquare}). Going along the results of~\cite{NZ}, this allows us to define $G$-matrices as the inverse of transposed $C$-matrices, and we define $\gv$-vectors as the column vectors of $G$-matrices. We then prove (see Section~\ref{sec:Tropical} and Theorem~\ref{thm:Tesseract} for details) that these definitions are compatible with both rescaling and the projection of $C$-matrices and $G$-matrices with respect to the folding. We note that these results provide a categorical interpretation of $\gv$-vectors and $G$-matrices of $Q^{I_2(2n)}$ (Corollary~\ref{cor:GVectors}).

\medskip

      The paper is organised as follows. In Section~\ref{sec:Prelim} we recall all essential definitions and details from~\cite{DTOdd} about mutations, unfoldings, and semirings. In Section~\ref{sec:Projection}, we describe the projections of module and bounded derived categories induced by the (un)foldings of quivers. Section~\ref{sec:Action} is devoted to the description of the semiring action on the module and bounded derived categories. In Section~\ref{sec:ClusterTilting}, we extend the semiring action to the cluster categories of the unfolded quivers, thus providing a categorification of mutations of the folded quivers. Section~\ref{sec:Tropical} is devoted to the construction of the tropical seed pattern, and to the compatibility of the projections and mutations. Finally, Appendix~\ref{sec:Examples} provides two detailed worked examples (one for unfoldings of type $\coxA_{2n-1}$ and one for unfoldings of type $\coxD_{n+1}$) that showcases the entire theory of the paper for the benefit of the reader.

	\section{Preliminaries} \label{sec:Prelim}
	For the benefit of the reader, we will briefly recall some notation and definitions from \cite{DTOdd} that we will use throughout the paper. Further details can be found in the aforementioned reference.
	
	\subsection{General setup and notation} \label{sec:Prelim-Notation}
	Throughout the paper, $\field$ is an algebraically closed field and $Q^\gend$ is a quiver of Dynkin type $\gend$. The vertex set of $Q^\gend$ is denoted by $Q^\gend_0$ and the arrow set is denoted by $Q^\gend_1$. We denote by $\field Q^\gend$ the path algebra of $Q^\gend$ over the field $\field$. All $\field Q^\gend$-modules in this paper are right modules, and thus we read paths in the quiver from left to right. We denote by $\mod*\field Q^\gend$ the category of finitely generated right $\field Q^\gend$ modules and by $\bder(\field Q^\gend)$ the bounded derived category.
	
	Each vertex $i \in Q^\gend_0$ simultaneously corresponds to a simple, indecomposable projective, and indecomposable injective $\field Q^\gend$-module, which we will denote by $S(i)$, $P(i)$ and $I(i)$ respectively. The shift functor of $\bder(\field Q^\gend)$ shall be written as $\sus\colon\bder(\field Q^\gend) \rightarrow \bder(\field Q^\gend)$, and where appropriate, we will adopt the abuse of notation where for each object $M \in \mod*\field Q^\gend$, the object $M \in\bder(\field Q^\gend)$ is the corresponding object concentrated in degree 0. Additionally, we denote by $\tau_{\acat}$ the Auslander-Reiten translate in the category $\acat$, where $\acat$ is either a module, bounded derived, or cluster category of the appropriate type. Whenever the context is clear, we will omit this subscript and simply write $\tau$.

	\subsection{Exchange matrices over a ring, $R$-quivers, and the $\Iiin$ root system} \label{sec:Prelim-ExMats}
	Let $R$ be a totally ordered ring throughout. Eventually, $R$ will be an integral domain and torsion-free, though one need not make this assumption for the initial setup. By an \emph{exchange matrix} over $R$, we mean a square skew-symmetrisable matrix with entries in $R$ whose rows and columns are indexed by a set $Q^B_0$. The \emph{$R$-quiver} associated to a \emph{skew-symmetric} exchange matrix $B=(b_{ij})$ over $R$ is the $R$-arrow-weighted quiver $Q^B$ with vertex set $Q^B_0$, arrow set $Q^B_1$, and arrow weighting function $\aw\colon Q^B_1 \rightarrow R_{>0}$ determined such that $b_{ij}>0$ if and only if there exists an arrow $a\colon i \rightarrow j$ in $Q^B_1$ with weight $\aw(a)=b_{ij}$.
	
	This paper will be working with a particular family of $R$-quivers and exchange matrices over $R$. Namely, we will be working with exchange matrices of type $\Iiin$ over the ring $R=\integer[2 \cos \frac{\pi}{2n}]$ whose ordering is induced by $\mathbb{R}$. We will denote the corresponding $R$-quiver of $\Iiin$ by
	\begin{equation*}
		Q^{\Iiin}\colon \xymatrix@1{[0] \ar[rr]^-{2 \cos \frac{\pi}{2n}} && [1]}
	\end{equation*}
	where the weight of the unique arrow is given by its label.
	
	The quiver $Q^\Iiin$ is associated to a \emph{non-crystallographic root system} that is related to the Coxeter group $\Iiin$ in the sense of \cite{Deodhar}. The non-crystallographic root system of type $\Iiin$ has $2n$ positive roots, two of which are simple. Being non-crystallographic, the root system can appear in various forms which depends on the relative lengths of the simple roots. This paper will work with two versions of the root system.
	
	\textit{Version 1: The standard form.} In this version, the positive roots may be partitioned into two equally sized classes, called \emph{short} and \emph{long} positive roots. In particular, one of the simple positive roots will be longer than the other. The short roots of $\Iiin$ will be of length $\lambda$ and the long roots of $\Iiin$ will be of length $2 \lambda \cos \frac{\pi}{2n}$. For the quiver $Q^{\Iiin}$ above, we will choose the convention that the short positive roots correspond to the points $\lambda \mathrm{exp}(\frac{k\pi\imunit}{n})$ and the long positive roots correspond to the points $(2\lambda \cos \frac{\pi}{2n})\mathrm{exp}(\frac{(2k+1)\pi\imunit}{2n})$ for $0 \leq k \leq n-1$. For the opposite quiver $(Q^{\Iiin})\op$, this convention is reversed. See Figure~\ref{fig:I8RootSys} for an example. Throughout the paper, we will denote the set of short roots (resp. short positive roots) of $\Iiin$ by $\Isroots{2n}$ (resp. $\Isproots{2n}$) and we will denote the set of long roots (resp. long positive roots) of $\Iiin$ by $\Ilroots{2n}$ (resp. $\Ilproots{2n}$).
	
	We call this version of the root system of $\Iiin$ standard, as it agrees with the root systems $\coxB_2 = \coxI_2(4)$ and $\coxG_2 = \coxI_2(6)$, and is the one that is most compatible with the theory of semiring actions on module categories that we will develop.
	
	\begin{figure}
		\begin{tikzpicture}
	\foreach \i in {0,...,7} {
		\draw[->] (0,0) -- ($({cos(45*\i)},{sin(45*\i)})$);
	}
	\foreach \i in {0,...,7} {
		\draw[->] (0,0) -- ($({2*cos(22.5)*cos(2*22.5*\i+22.5)},{2*cos(22.5)*sin(2*22.5*\i+22.5)})$);
	}
	\draw[anchor=west] (1,0) node {$\alpha_0$};
	\draw[anchor=east] ($({2*cos(22.5)*cos(2*22.5*3+22.5)},{2*cos(22.5)*sin(2*22.5*3+22.5)})$) node {$\alpha_1$};
	
	\foreach \i in {0,...,15} {
		\draw[->] (4,0) -- ($({4+cos(22.5*\i)},{sin(22.5*\i)})$);
	}
	\draw[anchor=west] (5,0) node {$\alpha_0$};
	\draw[anchor=east] ($({4.1+cos(2*22.5*3+22.5)},{sin(2*22.5*3+22.5)})$) node {$\alpha_1$};
\end{tikzpicture}
		\caption{The root systems of type $\coxI_2(8)$ corresponding to the quiver $Q^{\Iiin}$, with simple positive roots labelled by $\alpha_0$ and $\alpha_1$. Left: The standard root system. Right: The rescaled root system.} \label{fig:I8RootSys}
	\end{figure}
	
	\textit{Version 2: The rescaled form.}
	In this version the positive roots all have the same length, which by convention, we choose to be $1$. The positive roots thus correspond to the points $\lambda \mathrm{exp}(\frac{k\pi\imunit}{2n})$ with $0 \leq k \leq 2n-1$. We call this the \emph{rescaled root system of $\Iiin$}. We only use the rescaled root system in Section~\ref{sec:Tropical}.
	
	To make precise which version of the root system we are using when considering the $R$-quiver $Q^{\Iiin}$, we will define a \emph{valuation} on the quiver.
	
	\begin{defn}
		Let $S$ be a totally ordered ring. An \emph{$S$-valuation} of an $R$-quiver is a positive $S$-vertex-weighting $\mw\colon Q^B_0 \rightarrow S_{>0}$.
	\end{defn}
	
	The standard form of the root system of type $\Iiin$ is associated to the $R$-valuation of $Q^\Iiin$ given by $\mw([0]) = \lambda$ and $\mw([1]) = 2\lambda \cos \frac{\pi}{2n}$. The rescaled version of $\Iiin$ is instead given by the $R$-valuation $\mw([0])=\mw([1]) = \lambda$. In all but one case in this paper, $\lambda = 1$ (where $\lambda =2$ in the other case).
	
	Henceforth, whenever we refer to the root system of $\Iiin$ and its roots, we mean the standard form (version 1) and its elements. We will always refer to version 2 of $\Iiin$ as the \emph{rescaled root system} and its elements as \emph{rescaled roots}.
	
	\subsection{Weighted (un)foldings} \label{sec:Prelim-Folding}
	Here we will recall the definition of a weighted (un)folding from \cite{DTOdd}. The definition is a generalisation of the classical definition of unfolding due to Zelevinsky (see \cite{FST3,FST2} for details) and makes use of rescaling introduced by Reading in \cite{Reading}. 
	
	\begin{defn}
	Let $B=(b_{ij})$ be an exchange matrix over $\integer$ and $B'=(b'_{[i][j]})$ be an exchange matrix over $R$. Suppose there exists a disjoint collection of index sets $\{E_{[i]} : [i] \in Q^{B'}_0\}$ such that $\bigcup_{[i] \in Q^{B'}_0} E_{[i]} = Q_0^B$. Then $B$ has the structure of a block matrix $(B_{[i][j]})$ with blocks indexed by $Q^{B'}_0$. In this case, we call a pair $(B,B')$ of exchange matrices an \emph{origami pair} if the following hold:
	\begin{enumerate}[label=(O\arabic*)]
		\item For each $[i],[j] \in Q_0^{B'}$, the sum of entries in each column of $B_{[i][j]}$ is $b'_{[i][j]}$.
		\item If $b'_{[i][j]}>0$ then the $B_{[i][j]}$ has all entries non-negative.
	\end{enumerate}
	
	Suppose $R$ is an integral domain and let $\Frac(R)$ be the field of fractions of $R$. An exchange matrix $B$ over $\integer$ is said to be a \emph{weighted unfolding} of an exchange matrix $B'$ over $R$ if the following hold.
	\begin{enumerate}[label=(U\arabic*)]
		\item There exist diagonal matrices $W=(w_i)$ and $P=(p_{[i]})$ with positive entries in $R \subseteq \Frac(R)$ such that $(W B W\inv, P B' P\inv)$ is an origami pair with blocks indexed by the collection $\{E_{[i]} : [i] \in Q^{B'}_0\}$. Here, multiplication occurs over $\Frac(R)$.
		\item For any sequence of iterated mutations $\mu_{[k_1]}\ldots\mu_{[k_l]}$ of $B'$, the pair
		\begin{equation*}
			(W \wh\mu_{[k_1]}\ldots\wh\mu_{[k_l]}(B) W\inv, P\mu_{[k_1]}\ldots\mu_{[k_l]}(B') P\inv)
		\end{equation*}
		is origami, where each $\wh\mu_{[j]}$ is the \emph{composite mutation} of $B$ given by
		\begin{equation*}
		\wh\mu_{[j]} = \prod_{i \in E_{[j]}} \mu_i.
		\end{equation*}
	\end{enumerate}
	\end{defn}
	
	The composite mutations of $B$ in the above definition are well-defined because the mutations indexed by each block are pairwise commutative. We call the matrix $W$ the \emph{weight matrix} of the unfolding (whose entries $w_i$ we call \emph{weights}), and we call $P$ the \emph{rescaling matrix} of the unfolding. We also call the exchange matrix $B'$ the \emph{folded exchange matrix} of $B$, and we call $B$ the \emph{unfolded exchange matrix} of $B'$.
	
	A weighted unfolding associated to an origami pair $(W B W\inv, P B' P\inv)$ equips both the associated $\integer$-quiver $Q^B$ and the $R$-quiver $Q^{B'}$ with an $R$-valuation. For the $\integer$-quiver, we denote the valuation with the function $\vw\colon Q^B_0 \rightarrow R_{>0}$ defined by $\vw(i)=w_i$ for each $i \in Q^B_0$. For the $R$-quiver $Q^{B'}$, we denote the valuation with the function $\mw\colon Q^{B'}_0 \rightarrow R_{>0}$ defined by $\mw([i])=p_{[i]}$. To each weighted unfolding $B$ of $B'$, we define a \emph{weighted folding} of quivers $F\colon Q^B \rightarrow Q^{B'}$. That is, $F$ is a surjective morphism of quivers such that $F(i)=[j]$ whenever $i \in E_{[j]}$. Given such a folding $F$, we call $Q^B$ the \emph{unfolded quiver} and $Q^{B'}$ the \emph{folded quiver}. In any given folding, the quivers $Q^B$ and $Q^{B'}$ are implicitly assumed to have the aforementioned structure of $R$-valued quivers (with weight function $\vw$ for $Q^B$ and $\mw$ for $Q^{B'}$).
	
	This paper is concerned with three families of foldings onto $R$-quivers of type $\Iiin$. We call these families foldings of type $\coxA$, $\coxD$ and $\coxE$, respectively. The foldings of type $\coxA$ in particular are consistent with similar constructions in \cite{WaveFrontsReflGroups,Muhlherr}. 
	
	\begin{rem} \label{rem:CoxeterPlane}
		Each of the foldings $F^\gend$ (with $\gend \in \{\coxA_{2n-1},\coxD_{n+1},\coxE_6,\coxE_7,\coxE_8\}$) described in the remainder of this section induces a projection of the roots of $\gend$ onto the roots of $\Iiin$. In fact, this projection is precisely the projection onto the Coxeter plane (studied in detail in \cite{SteinbergRefl}). As such, the Coxeter number of $\gend$ agrees with the Coxeter number of $\Iiin$ for each folding $F^\gend$.
	\end{rem}
	
	\subsubsection{Foldings of type $\coxA$} The first family to consider is the folding from a bipartite quiver of type $\coxA_{2n-1}$, which we denote by $Q^{\coxA_{2n-1}}$. Specifically, we have $F^{\coxA_{2n-1}}\colon Q^{\coxA_{2n-1}} \rightarrow Q^{\Iiin}$, where $Q^{\coxA_{2n-1}}$ is given by
	\begin{equation*}
		\xymatrix{0 \ar[r] & 1 & \ar[l] 2 \ar[r] & \cdots & \ar[l] 2n-4 \ar[r] & 2n-3 & \ar[l] 2n-2}		
%		\xymatrix@R=0ex@C=0.5ex{
%			0 \ar[drr] &&							\\
%			&& 1 &&								\\
%			2 \ar[urr] \ar[drr] &&					\\
%			&& \\
%			\vdots &  & \vdots & \rightarrow & [0] \ar[rrrr]^-{2\cos\frac{\pi}{2n}} &&&&	[1]\\
%			& &&									\\
%			2n-2 \ar[urr] \ar[drr] &&				\\
%			&& 2n-1&&							\\
%			2n \ar[urr] &&
%		}
	\end{equation*}
	with vertex weights such that $\vw(i) = U_i(\cos\frac{\pi}{2n})$, where $U_i$ is the $i$-th Chebyshev polynomial of the second kind. Furthermore, each vertex $i$ is such that $F^{\coxA_{2n-1}}(i)=[0]$ if $i$ is even and $F^{\coxA_{2n-1}}(i)=[1]$ if $i$ is odd. The valuation of $Q^{\Iiin}$ is such that $\mw([0])=1$ and $\mw([1])=2\cos\frac{\pi}{2n}$. It is not difficult to verify that on the level of unfoldings, this class of unfoldings factors through an unfolding of an exchange matrix of Dynkin type $\dynC_{n}$. The above quiver has $\integer_2$-symmetry via the action that maps $i \in Q^{\coxA_{2n-1}}_0$ to $2n-2-i \in Q^{\coxA_{2n-1}}_0$. This action also respects the vertex-weights in the sense that $U_i(\cos\frac{\pi}{2n})=U_{2n-2-i}(\cos\frac{\pi}{2n})$. In light of this group action, it is sometimes convenient to work with the alternative labelling
	\begin{equation*}
		\xymatrix@C=3.5ex{0^+ \ar[r] & 1^+ & \ar[l] \cdots \ar@{-}[r] & (n-2)^+ & \ar@{-}[l] n-1 \ar@{-}[r] & (n-2)^- & \ar@{-}[l] \cdots \ar[r] & 1^- & \ar[l] 0^-}	.	
	\end{equation*}
	
	\subsubsection{Foldings of type $\coxD$} The second family of foldings we consider are from bipartite quivers of type $\coxD_{n+1}$, which we denote by $Q^{\coxD_{n+1}}$. Specifically, we have $F^{\coxD_{n+1}}\colon Q^{\coxD_{n+1}} \rightarrow Q^{\Iiin}$, where $Q^{\coxD_{n+1}}$ is given by
	\begin{equation*}
		\xymatrix{
			&&&&&(n-1)^+ \\
			0 \ar[r] & 1 & \ar[l] 2 \ar[r] & \cdots & \ar@{-}[l] n-2 \ar@{-}[ur] \ar@{-}[dr] \\
			&&&&&(n-1)^-
		}
%		\xymatrix{
%		(n-1)^+ \\
%			& \ar@{-}[ul] \ar@{-}[dl] n-2 \ar@{-}[r] & \cdots \ar[l] 2 \ar[r] & 1 & \ar[l] 0 & \rightarrow & [0] \ar[rr]^-{2\cos\frac{\pi}{2n}} &&	[1] \\
%		(n-1)^-
%		}
	\end{equation*}
	with vertex weights such that $\vw(i) = 2 U_i(\cos\frac{\pi}{2n})$ for $0 \leq i \leq n-2$ and $\vw((n-1)^\pm) =U_{n-1}(\cos\frac{\pi}{2n})$. Furthermore, each vertex $i$ is such that $F^{\coxD_{n+1}}(i)=[0]$ if $i$ is even and $F^{\coxD_{n+1}}(i)=[1]$ if $i$ is odd (the same rule applies to the vertices $(n-1)^\pm$ with the integer $n-1$). The valuation of $Q^{\Iiin}$ is such that $\mw([0])=2$ and $\mw([1])=4\cos\frac{\pi}{2n}$. It is not difficult to verify that on the level of unfoldings, this class of unfoldings factors through an unfolding of an exchange matrix of Dynkin type $\dynB_{n}$. This quiver also has $\integer_2$-symmetry via the action that fixes each $i \in Q^{\coxD_{n+1}}_0$ and maps $(n-1)^\pm \in Q^{\coxD_{n+1}}_0$ to $(n-1)^\mp \in Q^{\coxD_{n+1}}_0$.
	
	\subsubsection{Foldings of type $\coxE$} There are three exceptional foldings of type $E$. Namely we have foldings 
	\begin{align*}
		F^{\coxE_6} \colon Q^{\coxE_6} &\rightarrow Q^{\coxI_2(12)}, \\
		F^{\coxE_7} \colon Q^{\coxE_7} &\rightarrow Q^{\coxI_2(18)}, \\
		F^{\coxE_8} \colon Q^{\coxE_8} &\rightarrow Q^{\coxI_2(30)}, 
	\end{align*}
	where $Q^{\coxE_6}$, $Q^{\coxE_7}$ and $Q^{\coxE_8}$ are bipartite $R$-valued quivers of type $\coxE_6$, $\coxE_7$ and $\coxE_8$, respectively. Throughout, the valuation of $Q^{\Iiin}$ is such that $\mw([0])=1$ and $\mw([1])=2\cos\frac{\pi}{2n}$.
	
	The quiver $Q^{\coxE_6}$ is given by
	\begin{equation*}
		\xymatrix{
			&&v_6 \\
			0^+ \ar[r] & 1^+ & \ar[l] 2 \ar[u] \ar[r] & 1^- & \ar[l] 0^-
		}
	\end{equation*}
	with weighting such that $\vw(i^\pm)=U_i(\cos\frac{\pi}{12})$, $\vw(2)=U_2(\cos\frac{\pi}{12})$ and
	\begin{equation*}
		\vw(v_6)=U_3(\cos\tfrac{\pi}{12}) - U_1(\cos\tfrac{\pi}{12})=U_3(\cos\tfrac{\pi}{12}) - U_5(\cos\tfrac{\pi}{12})+U_1(\cos\tfrac{\pi}{12})=\sqrt{2}. 
	\end{equation*}
	Moreover, we have $F^{\coxE_6}(0^\pm)=F^{\coxE_6}(2)=[0]$ and $F^{\coxE_6}(1^\pm)=F^{\coxE_6}(v_6)=[1]$. This unfolding can be shown to factor through an unfolding of an exchange matrix of Dynkin type $\dynF_{4}$.
	
	The quiver $Q^{\coxE_7}$ is given by
	\begin{equation*}
		\xymatrix{
			&&&\wt{1}\ar[d] \\
			0 \ar[r] & 1 & \ar[l] 2 \ar[r] & 3  & \ar[l] \wt{2} \ar[r] & v_7
		}
	\end{equation*}
	with weighting such that $\vw(i)=U_i(\cos\frac{\pi}{18})$, $\vw(\wt{i})=U_i(\cos\frac{\pi}{9})$ and
	\begin{equation*}
		\vw(v_7)= U_5(\cos\tfrac{\pi}{18}) - U_3(\cos\tfrac{\pi}{18}) = 2 \cos \tfrac{5\pi}{18}. 
	\end{equation*}
	It worth noting that $U_0(\cos\frac{\pi}{9}) = U_0(\cos\frac{\pi}{18})$, $U_1(\cos\frac{\pi}{9}) = U_4(\frac{\pi}{18})-U_6(\frac{\pi}{18})+U_2(\frac{\pi}{18})$, $U_2(\cos\frac{\pi}{9}) = U_6(\cos\frac{\pi}{18})-U_2(\frac{\pi}{18})$, and that $U_3(\cos\frac{\pi}{9}) = U_2(\cos\frac{\pi}{18})$. It is therefore convenient to occasionally adopt the labelling $\wt{0}=0$ and $\wt{3}=2$. Moreover, $F^{\coxE_7}$ maps the source vertices to $[0]$ and sink vertices to $[1]$.
	
	The quiver $Q^{\coxE_8}$ is given by
	\begin{equation*}
		\xymatrix{
			&&&&v_8 \\
			0 \ar[r] & 1 & \ar[l] 2 \ar[r] & 3 & \ar[l] 4 \ar[u] \ar[r] & \phi_1 & \ar[l] \phi_0
		}
	\end{equation*}
	with weighting such that $\vw(i)=U_i(\cos\frac{\pi}{30})$, $\vw(\phi_{i})=\gratio U_i(\cos\frac{\pi}{30})$ and $\vw(v_8)=\gratio\inv U_3(\cos\frac{\pi}{30})$, where $\gratio=2\cos\frac{\pi}{5}$ is the golden ratio and $\gratio\inv= \gratio - 1=2\cos\frac{2 \pi}{5}$ is its inverse. It is also worth noting that $\gratio = U_6(\cos\frac{\pi}{30}) - U_4(\cos\frac{\pi}{30})$. Similar to above, $F^{\coxE_8}$ maps the source vertices to $[0]$ and sink vertices to $[1]$.
	
	\subsection{Chebyshev polynomials and associated (semi)rings} \label{sec:Prelim-Cheb}
	As previously mentioned, we denote throughout the paper the $i$-th Chebyshev polynomial of the second kind by $U_i$. That is, $U_i$ satisfies the identity
	\begin{equation*}
		U_i(\cos \theta) \sin \theta = \sin ((i+1)\theta).
	\end{equation*}
	Chebyshev polynomials of the second kind satisfy many identities that are useful and which appear within the theory of this paper. Many of these useful identities are highlighted in \cite[Lemma 4.1]{DTOdd}, which we restate (and adapt to our specific context) here for the benefit of the reader.
	\begin{lem}[\cite{DTOdd}, Lemma 4.1] \label{lem:ChebPolyProps}
		Define a sequence $(\theta_i)_{i \in \nnint}$ by $\theta_i = U_i(\cos\frac{\pi}{2n})$. Then
		\begin{enumerate}[label=(\alph*)]
			\item $\theta_1 = 2 \cos\frac{\pi}{2n}$,
			\item $\theta_{2n-1} = 0$,
			\item $\theta_{n-1-i} = \theta_{n-1+i}$
			\item $\theta_{2n-1+i} = -\theta_{2n-1-i}$ for $i \leq 2n-1$
			\item $\theta_i\theta_j = \sum_{k=0}^j \theta_{i-j+2k}$ for $j \leq i$,
			\item $\theta_i > 1$ for $0 < i < 2n-2$.
		\end{enumerate}
	\end{lem}
	
	In \cite{DTOdd}, we defined a family of rings related to Chebyshev polynomials of the second kind that would later be used to define a semiring action on the module, derived and cluster categories of unfolded quivers. We will do the same here, except the (semi)rings we consider in this paper will be slightly different.

	Throughout the paper, denote $\ZUi{2n}=\integer[2 \cos \tfrac{\pi}{2n}]$ for each $n \geq 2$. Now consider the ring
	\begin{equation*}
		Z_{2n} = \integer[x] / (U_n(\tfrac{x}{2}) - U_{n-2}(\tfrac{x}{2}))
	\end{equation*}
	Since Chebyshev polynomials of the second kind satisfy the product rule given in Lemma~\ref{lem:ChebPolyProps}(e), it is not hard to verify that $U_{2n-1}(\tfrac{x}{2}) = 0$, $U_{n-1-i}(\tfrac{x}{2}) = U_{n-1+i}(\tfrac{x}{2})$ and $U_{2n-1+i}(\tfrac{x}{2}) = -U_{2n-1-i}(\tfrac{x}{2})$ in the ring $Z_{2n}$. In particular, this can be obtained inductively from the relation $U_n(\tfrac{x}{2}) - U_{n-2}(\tfrac{x}{2})$ by multiplying by $x=U_1(\frac{x}{2})$ (c.f. \cite[Remark 4.3]{DTOdd}). It therefore follows from Lemma~\ref{lem:ChebPolyProps} that there exists a ring epimorphism $\zeta_{2n}\colon Z_{2n} \rightarrow \ZUi{2n}$ defined by $\zeta_{2n}(x) = 2 \cos \tfrac{\pi}{2n}$. In fact, it follows from the results of \cite{WatkinsZeitlin} that $\zeta_{2n}$ is actually an isomorphism whenever $n$ is even. On the other hand, when $n$ is odd, we have $\ZUi{2n} \cong \integer[x] / (p_{2n} - q_{2n})$, where $x p_{2n} = U_n(\frac{x}{2})$ and $xq_{2n} = U_{n-2}(\frac{x}{2})$. Most of the (semi)rings that we consider are given by modifying the ring $Z_{2n}$ in some way.

	\subsubsection{(Semi)rings of type $\coxA$}
	The following (semi)rings will be used in foldings of type $\coxA$, which are known to be Verlinde fusion rings (c.f. \cite[Example 4.10.6]{EGNO}).
	
	\begin{defn} \label{defn:ASemirings}
		For each $n\geq 2$, define the families of commutative rings
		\begin{align*}
			\achebr{2n-1} &= \integer[\croot_2,\croot_4,\ldots,\croot_{2n-2}] & &\subset & \hachebr{2n-1}&=\integer[\croot_1,\croot_2,\croot_3,\ldots,\croot_{2n-2}],
		\end{align*}
		subject to the following product rule for any $1 \leq j \leq i \leq 2n-2$:
		\begin{equation*}
			\croot_i\croot_j = \croot_j \croot_i = \sum_{k=0}^j \croot_{i-j+2k},
		\end{equation*}
		where any element $\croot_k$ with $k \not\in\{1,\ldots,2n-2\}$ resulting from the above product is such that $\croot_0 = 1$, $\croot_{2n-1} = 0$ and $\croot_{2n+k}=-\croot_{2n-2-k}$. It is easy to check that, after cancellation, the above product will always produce a sum of elements $\croot_k$ with positive coefficients. Thus, it makes sense to define the corresponding semirings:
		\begin{align*}
			\achebsr{2n-1} &= \integer_{\geq 0}[\croot_2,\croot_4,\ldots,\croot_{2n-2}] & &\subset & \hachebsr{2n-1}&=\integer_{\geq 0}[\croot_1,\croot_2,\croot_3,\ldots,\croot_{2n-2}],
		\end{align*}
		which are subsemirings of $\achebr{2n-1}$ and $\hachebr{2n-1}$ respectively.
	\end{defn}
	
	The ring $\hachebr{2n-1}$ is related to the ring $Z_{2n}$ by forgetting the relations $U_{n-1+i}(\frac{x}{2}) - U_{n-1-i}(\frac{x}{2})$ for $i \leq n-1$ and preserving the relation $U_{2n-1}(\frac{x}{2})$. In particular, we have
	\begin{equation*}
		\hachebr{2n-1} \cong \integer[x] / (U_{2n-1}(\tfrac{x}{2})),
	\end{equation*}
	where each element $\croot_i$ corresponds to $U_{i}(\tfrac{x}{2})$ (and $1=\croot_0$ corresponds to $1 = U_0(\frac{x}{2})$). From this isomorphism, we can deduce the following.
	
	\begin{lem} \label{lem:ABasis}
		The set $\mathcal{B}^{\coxA_{2n-1}} =\{\croot_0,\croot_1,\ldots, \croot_{2n-2}\}$ is a $\integer$-basis of $\hachebr{2n-1}$.
	\end{lem}
	\begin{proof}
		Since $\hachebr{2n-1} \cong \integer[x] / (U_{2n-1}(\tfrac{x}{2}))$, one can deduce by iterative products by $x$ that the elements of the set $\{U_{i}(\tfrac{x}{2}) : i\geq 2n-1\}$ are $\integer$-linearly dependent to the elements in the set $\{U_{i}(\tfrac{x}{2}) : 0 \leq i \leq 2n-2\}$. Specifically, $U_{(2n-1)i +j} = -U_{(2n-1)i -j}$ for any $i,j \geq 0$. Since Chebyshev polynomials of the second kind are $\integer$-linearly independent within the ring $\integer[x]$, it is clear that the elements of $\{U_{i}(\tfrac{x}{2}) : 0 \leq i \leq 2n-2\}$ are $\integer$-linearly independent within the ring $\integer[x] / (U_{2n-1}(\tfrac{x}{2}))$. The result then follows from the isomorphism.
	\end{proof}
	
	The element $\croot_{2n-2} \in \hachebr{2n-1}$ (or equivalently, $U_{2n-2}(\tfrac{x}{2}) \in \integer[x] / (U_{2n-1}(\tfrac{x}{2}))$) is representative of the group symmetry of $\coxA_{2n-1}$. In particular, $\croot_{2n-2}\croot_i = \croot_{2n-2-i}$. We can thus also write
	\begin{equation*}
		\hachebr{2n-1} \cong \integer[x,y] / (U_n(\tfrac{x}{2}) - yU_{n-2}(\tfrac{x}{2}), yU_{n-1}(\tfrac{x}{2})-U_{n-1}(\tfrac{x}{2}), y^2 -1),
	\end{equation*}
	where each $\croot_i$ corresponds to $U_{i}(\tfrac{x}{2})$ and $y$ corresponds to $\croot_{2n-2}$. This second isomorphism is useful for the next family of (semi)rings.
	
	\subsubsection{(Semi)rings of type $\coxD$}
	In the next family of (semi)rings, we use a signed notation on the elements. Here, elements $\croot_i^+$ and $\croot_i^-$ are distinct, but in some statements we collectively refer to them as $\croot_i^\pm$. We also have statements dependent on elements which either have the same sign or opposite signs. For example, by $\croot_i^\pm \croot_j^\pm$, we mean either $\croot_i^+ \croot_j^+$ or $\croot_i^- \croot_j^-$. On the other hand, by $\croot_i^\pm \croot_j^\mp$, we mean $\croot_i^+ \croot_j^-$ or $\croot_i^- \croot_j^+$.
	
	\begin{defn} \label{defn:DSemirings}
		Define commutative (semi)rings of type $\coxD_4$ by
		\begin{align*}
			\dchebr{4} &= \integer[g] & &\subset & \hdchebr{4} &= \integer[g,\croot_1], \\
			\dchebsr{4} &= \nnint[g] & &\subset & \hdchebsr{4} &= \nnint[g,\croot_1],
		\end{align*}
		satisfying the relations $g^3 = 1$, $g\croot_1=\croot_1$ and $\croot_1^2 = 1 + g + g^2$. For $n> 3$, define further families of commutative (semi)rings of type $\coxD_{n+1}$ by
		\begin{align*}
			\dchebr{n+1}&=\integer\left[\croot_0^-,\croot^\pm_2,\ldots,\croot^\pm_{2\left\lfloor\frac{n-1}{2}\right\rfloor}\right] & &\subset & \hdchebr{n+1}&=\integer[\croot_0^-,\croot^\pm_1,\croot^\pm_2,\croot^\pm_3,\ldots,\croot^\pm_{n-1}], \\
			\dchebsr{n+1}&=\integer_{\geq 0}\left[\croot^-_0,\croot^\pm_2,\ldots,\croot^\pm_{2\left\lfloor\frac{n-1}{2}\right\rfloor}\right] & &\subset & \hdchebsr{n+1}&=\integer_{\geq 0}[\croot_0^-,\croot^\pm_1,\croot^\pm_2,\croot^\pm_3,\ldots,\croot^\pm_{n-1}],
		\end{align*}
		subject to the following product rules for any $1 \leq j \leq i \leq n-1$:
		\begin{align*}
			\croot^\pm_i\croot^\pm_j &= \sum_{k=0}^j \omega_{i-j+2k}, & 
			\croot_0^- \croot_j^\pm &= \croot_j^\mp,  &
			(\croot_0^-)^2 &= \croot_0^+ = 1,
		\end{align*}
		where each element $\omega_k$ resulting from the above product is such that
		\begin{equation*}
			\omega_k =
			\begin{cases}
				\croot_k^+			& \text{if }k <n-1, \\
				\croot_{n-1}^+		& \text{if }k =n-1\text{ and }i+j-n+1 \equiv 0 \text{ (mod 4)}, \\
				\croot_{n-1}^-		& \text{if }k =n-1\text{ and }i+j-n+1 \equiv 2 \text{ (mod 4)}, \\
				\croot_{2n-2-k}^-	& \text{if }k >n-1.
			\end{cases}
		\end{equation*}
	\end{defn}
	
	The ring $\hdchebr{4}$ is an exceptional case that exploits the $\integer_3$ group symmetry of $\coxD_4$, along with the relation $U_2(\cos\frac{\pi}{6}) = 2$. Here the element $\croot_1$ is represents $U_1(\cos\frac{\pi}{6}) = \sqrt{3}$, and $1$, $g$ and $g^2$ represent $U_0(\cos\frac{\pi}{6}) = 1$. Consequently, $\croot_1^2$ represents $U_0(\cos\frac{\pi}{6}) + U_2(\cos\frac{\pi}{6}) = 3$, which is given by the element $1 + g + g^2$. One can show that this is actually the fusion ring of Tambara-Yamagami type associated to the cyclic group of order 3 (c.f. \cite[Example 4.10.5]{EGNO}). It is not difficult to see from this that we have the following.
	
	\begin{rem} \label{lem:D4Basis}
		The set $\mathcal{B}^{\coxD_4} = \{\croot_0,\croot_1,g,g^2\}$ is a $\integer$-basis of $\hdchebr{4}$, where for later convenience, we have adopted the notation $\croot_0 = 1$.
	\end{rem}
	
	For $n>3$, the ring $\hdchebr{n+1}$ is related to the ring $Z_{2n}$ by adding a generator that represents the $\integer_2$ group symmetry of $\coxD_{n+1}$, and by modifying the relations with respect to this group action. Specifically, we have
	\begin{equation*}
		\hdchebr{n+1} \cong \integer[x,y]/(U_{n}(\tfrac{x}{2}) - yU_{n-2}(\tfrac{x}{2}), y^2 -1),
	\end{equation*}
	where $\croot^+_i$ corresponds to $U_{i}(\tfrac{x}{2})$ and $\croot^-_i$ corresponds to $yU_{i}(\tfrac{x}{2})$. One can also see that $\hdchebr{n+1}$ is related to the second isomorphism of the ring $\hachebr{2n-1}$ by forgetting the relation $yU_{n-1}(\tfrac{x}{2})-U_{n-1}(\tfrac{x}{2})$. An interesting consequence of this is that $\dchebr{n+1} \cong \achebr{2n-1}$ whenever $n$ is even. We also have the following.
	
	\begin{lem} \label{lem:DBasis}
		For $n>3$, the set $\mathcal{B}^{\coxD_{n+1}}=\{\croot_0^\pm,\croot_1^\pm,\ldots,\croot_{n-1}^\pm\}$ is a $\integer$-basis for $\hdchebr{n+1}$.
	\end{lem}
	\begin{proof}
		Let $n >3$ be fixed. We will consider the ring $\integer[x,y]/(U_{n}(\tfrac{x}{2}) - yU_{n-2}(\tfrac{x}{2}), y^2 -1)$. It is not difficult to show that, for each $0 \leq l<\frac{1}{2}(n-1)$, we can inductively obtain relations 
		\begin{equation*}
			U_{n+2l}(\tfrac{x}{2}) = yU_{n-2-2l}(\tfrac{x}{2})
		\end{equation*}
		from the relation $U_{n}(\tfrac{x}{2}) = yU_{n-2}(\tfrac{x}{2})$ by iteratively multiplying by $U_2(\frac{x}{2})$. By multiplying each of these relations by $U_1(\tfrac{x}{2})$, we obtain further relations
		\begin{equation} \tag{$\dagger$}\label{eq:Drel}
			U_{n+2l-1}(\tfrac{x}{2})+U_{n+2l+1}(\tfrac{x}{2}) = yU_{n-3-2l}(\tfrac{x}{2})+yU_{n-1-2l}(\tfrac{x}{2}).
		\end{equation}
		If $n$ is even, then one can further deduce that
		\begin{equation*}
			yU_{2l+1}(\tfrac{x}{2}) = U_{2n-2l-3}(\tfrac{x}{2}) + (-1)^l U_{2n-1}(\tfrac{x}{2}).
		\end{equation*}
		On the other hand, if $n$ is odd, then one can deduce that $U_{2n-1}(\tfrac{x}{2})=0$, which implies that the set $\{U_i(\tfrac{x}{2}):i\geq 2n-1\}$ is $\integer$-linearly dependent to the set $\{U_i(\frac{x}{2}) : 0 \leq i \leq 2n-2\}$. One can also deduce that $y$ is $\integer$-linearly independent to the set $\{U_i(\frac{x}{2}) : 0 \leq i \leq 2n-2\}$ in this case. In addition, for each $l>0$, one can deduce the following relations when $n$ is odd.
		\begin{equation*}
			yU_{2l}(\tfrac{x}{2}) = U_{2(n-l-1)}(\tfrac{x}{2}) + (-1)^l (U_{2(n-l)}(\tfrac{x}{2})-y).
		\end{equation*}
		Consequently, in both the odd and even cases, the elements of the set $\{U_i(\frac{x}{2}),yU_i(\frac{x}{2}) : 0 \leq i \leq n-1\}$ are $\integer$-linearly independent. These correspond to the elements of the set $\mathcal{B}^{\coxD_{n+1}}$ under the previously stated isomorphism, and thus the elements of $\mathcal{B}^{\coxD_{n+1}}$ are $\integer$-linearly independent in $\hdchebr{n+1}$. Moreover, the product of any two elements in $\hdchebr{n+1}$ is a $\integer$-linear combination of elements in $\mathcal{B}^{\coxD_{n+1}}$. Thus, $\mathcal{B}^{\coxD_{n+1}}$ is a $\integer$-basis of $\hdchebr{n+1}$, as required.
	\end{proof}
	
	It is also worth noting how the alternating behaviour (of the sign superscript) of the term $\omega_{n-1}$ in the product rule for $\hdchebr{n+1}$ arises from the relations. If $i+j - n + 1 \equiv 0$ (mod 4) in a product $\croot_i^\pm\croot_j^\pm$ which contains a term $\omega_{n-1}$, then there are an even number of terms $\omega_k$ with $k>n-1$. In this case, we can use the relations (\ref{eq:Drel}) with $l>0$ to obtain $\omega_{n-1} = \croot_{n-1}^+$ in the product. On the other hand, if $i+j - n + 1 \equiv 2$ (mod 4) in a product $\croot_i^\pm\croot_j^\pm$ which contains a term $\omega_{n-1}$, then there are an odd number of terms $\omega_k$ with $k>n-1$. In this case, we must use the relation (\ref{eq:Drel}) with $l=0$ along with the other relations (\ref{eq:Drel}) with $l>0$ to obtain $\omega_{n-1} = \croot_{n-1}^-$ in the product.
	 
	\subsubsection{(Semi)rings of type $\coxE$}
	The exceptional (semi)rings of types $\coxE_6$, $\coxE_7$ and $\coxE_8$ largely result from exploiting special relations that arise from Chebyshev polynomials of the second kind when they are evaluated at $\cos \frac{\pi}{12}$, $\cos \frac{\pi}{18}$ and $\cos \frac{\pi}{30}$ respectively. We will begin by defining these (semi)rings, and then provide some additional explanation.

	\begin{defn} \label{defn:ESemirings}
		Define the following exceptional rings of type $\coxE$.
		\begin{align*}
			\echebr{6} &=\integer[\croot_0^-, \croot_2] / (S_6) & &\subseteq & \hechebr{6} &= \integer[\croot_0^-, \croot^+_1,\croot^-_1, \croot_2, \croot_{v_6}] / (\wh S_6),\\
			\echebr{7} &=\integer[\croot_2, \wt\croot_2] / (S_7) & &\subseteq & \hechebr{7} &= \integer[\croot_1, \wt\croot_1, \croot_2, \wt\croot_2, \croot_3, \croot_{v_7}] / (\wh S_7), \\
			\echebr{8} &=\integer[\gratio, \croot_2] / (S_8) & &\subseteq & \hechebr{8} &= \integer[\croot_1, \croot_2, \croot_{v_8},\gratio] / (\wh S_8),
		\end{align*}
		where each $(\wh S_i)$ is the ideal generated by a set of relations $\wh{S}_i = S_i \cup S'_i$ with
		\begin{align*}
			S_6 &= \{(\croot_0^-)^2-1, (\croot_2)^2 - 1 - 2\croot_2 - \croot_0^-, \croot_0^- \croot_2 - \croot_2\}, \\
			S'_6 &= \{\croot_0^-\croot_1^+ - \croot_1^-, (\croot_1^+)^2 - 1 - \croot_2, \croot^+_1\croot_2 - \croot^+_1 - \croot^-_1 - \croot_{v_6}, \croot^+_1 \croot_{v_6} - \croot_2\}, \\
			S_7 &= \{\croot_2^2 - 1 - \croot_2\wt\croot_2, \wt\croot_2^2 - 1 - \wt\croot_2 - \croot_2\}, \\
			S'_7 &= \{\croot_1^2 - 1 - \croot_2, \wt\croot_1 - \croot_2\wt\croot_2 +\croot_2 +\wt\croot_2, \croot_1\croot_2 - \croot_3 -\croot_1, \croot_1\wt\croot_1-\croot_3, \croot_1\wt\croot_2 - \croot_3 - \croot_{v_7}\}, \\
			S_8 &= \{\gratio^2 - \gratio - 1, \croot_2^2 - \gratio\croot_2 - \croot_2 - 1\}, \\
			S'_8 &= \{\croot^2_1 - 1 - \croot_2, \croot_1\croot_2 - \croot_1 - \gratio\croot_{v_8}, \croot_1\croot_{v_8}-\gratio\croot_2\}.
		\end{align*}
		Since the product of any two elements with only positive coefficients in the rings $\echebr{6}$, $\echebr{7}$ and $\echebr{8}$ will produce an element which again has only positive coefficients, it is natural to define semirings $\echebsr{6}$, $\echebsr{7}$ and $\echebsr{8}$ by restricting the coefficient ring from $\integer$ to $\nnint$. Later in the paper, it will be convenient to denote the multiplicative identity in these (semi)rings by $\croot_0^+ =1 \in \hechebr{6}$ and $\croot_0 =1 \in \hechebr{7},\hechebr{8}$.
	\end{defn}
	
	\begin{rem}
		Note that we could also define semirings $\hechebr{6}_+$ and $\hechebr{8}_+$ similarly, but we do not need this in our construction. Moreover, the ring $\hechebr{7}$ is somewhat unusual in that the product of any two elements with only positive coefficients in this ring will not necessarily produce an element which has only positive coefficients (see Table~\ref{tab:E7} for details). It is thus not possible to define a semiring $\hechebr{7}_+$, and hence it is preferable to avoid working with the notion of semirings of the form $\hechebr{i}_+$ in this paper. This may be related to the lack of a fusion ring of type $\coxE_7$, whereas $\hechebr{6}$ and $\hechebr{8}$ can be shown to be the $\coxE_6$ and $\coxE_8$ fusion rings respectively (see \cite{Kirillov} for details).
	\end{rem}

	\begin{table}[h] 
		\begin{tabular}{l | c c c c}
			& $1$ & $\croot_1$ & $\croot_2$ & $\croot_3$ \\ \hline
			$1$ & $1$ & $\croot_1$ & $\croot_2$ & $\croot_3$ \\
			$\croot_1$ & $\croot_1$ & $1 + \croot_2$ & $\croot_1 + \croot_3$ & $\croot_2+\wt\croot_1+\wt\croot_2$ \\
			$\croot_2$ & $\croot_2$ & $\croot_1 + \croot_3$ & $1 + \croot_2 + \wt\croot_1 + \wt\croot_2$ & $\croot_1 + 2 \croot_3 + \croot_{v_7}$ \\
			$\croot_3$ & $\croot_3$ & $\croot_2+\wt\croot_1+\wt\croot_2$ & $\croot_1 + 2 \croot_3 + \croot_{v_7}$ & $1 + 2\croot_2 + \wt\croot_1 + 2\wt\croot_2$ \\
			$\wt\croot_1$ & $\wt\croot_1$ & $\croot_3$ & $\croot_2+\wt\croot_2$ & $\croot_1 + \croot_3 + \croot_{v_7}$ \\
			$\wt\croot_2$ & $\wt\croot_2$ & $\croot_3 + \croot_{v_7}$ & $\croot_2 + \wt\croot_1 +\wt\croot_2$ & $\croot_1 + 2\croot_3$ \\
			$\croot_{v_7}$ & $\croot_{v_7}$ & $\wt\croot_2$ & $\croot_3$ & $\croot_2 + \wt\croot_1$
		\end{tabular}
		
		~
		
		\begin{tabular}{l | c c c}
			& $\wt\croot_1$ & $\wt\croot_2$ & $\croot_{v_7}$ \\ \hline \\[-2ex]
			$\wt\croot_1$ & $1 + \wt\croot_2$ & $\wt\croot_1 + \croot_2$ & $\croot_3 - \croot_{v_7}$ \\
			$\wt\croot_2$ & $\wt\croot_1+\croot_2$ & $1 + \croot_2 + \wt\croot_2$ & $\croot_1 + \croot_{v_7}$ \\
			$\croot_{v_7}$ & $\croot_3 - \croot_{v_7}$ & $\croot_1 + \croot_{v_7}$ & $1 - \wt\croot_1 + \wt\croot_2$
		\end{tabular}
		\caption{The products between the generators of the ring $\hechebr{7}$. All products follow from the relations in the set $S_7 \cup S'_7$.} \label{tab:E7}
	\end{table}
	
	We can equivalently write the rings $\hechebr{6}$, $\hechebr{7}$ and $\hechebr{8}$ as the following quotients of polynomial rings, which makes the connection with Chebyshev polynomials of the second kind clearer. We leave the reader to check these isomorphisms.
	\begin{align*}
		\hechebr{6} &\cong \integer[x,y,z] / (S''_6), \\
		\hechebr{7} &\cong \integer[x,y,z] / (S''_7), \\
		\hechebr{8} &\cong \integer[x,y,z] / (S''_8),
	\end{align*}
	where
	\begin{align*}	
		S''_6 &= \{xz - U_2(\tfrac{x}{2}), yz - z, y^2 - 1, z^2 - 1 - y\} \\
		S''_7 &= \{U_4(\tfrac{x}{2}) - y-U_2(\tfrac{y}{2}), U_5(\tfrac{x}{2}) - U_3(\tfrac{x}{2}) - z, U_6(\tfrac{x}{2}) - U_2(\tfrac{x}{2}) - U_2(\tfrac{y}{2}), \\ & \qquad U_7(\tfrac{x}{2})-x-U_3(\tfrac{x}{2}),U_8(\tfrac{x}{2})-1-y-U_2(\tfrac{x}{2})\} \\
		S''_8 &= \{U_3(\tfrac{x}{2})-yz,U_4(\tfrac{x}{2})-yU_2(\tfrac{x}{2}),xz-yU_2(\tfrac{x}{2}),y^2-y-1\}.
	\end{align*}
	Specifically, each $U_i(\frac{x}{2})$ corresponds to $\croot_i$, and each $z$ corresponds to $\croot_{v_i}$ (where $i\in\{6,7,8\}$ as appropriate). For $\hechebr{6}$, $y$ corresponds to $\croot_0^-$. For $\hechebr{7}$, each $U_i(\frac{y}{2})$ corresponds to $\wt\croot_i$. For $\hechebr{8}$, $y$ corresponds to $\gratio$.
	
	One can also check that we have relations $U_6(\frac{x}{2}) = y U_4(\frac{x}{2})$ and $U_5(\frac{x}{2}) = y U_5(\frac{x}{2})$ in the ring $\integer[x,y,z] / (S''_6)$. We also have the relations $U_9(\frac{x}{2}) = U_7(\frac{x}{2})$ and $U_3(\frac{y}{2}) = U_4(\frac{y}{2})$ in the ring $\integer[x,y,z] / (S''_7)$, and the relation $U_{15}(\frac{x}{2}) = U_{13}(\frac{x}{2})$ in the ring $\integer[x,y,z] / (S''_8)$. The rings $\ZUi{12}$, $\ZUi{18}$ and $\ZUi{30}$ can therefore be considered as quotients of the rings $\hechebr{6}$, $\hechebr{7}$ and $\hechebr{8}$ via the map which takes $x$ to $2\cos\frac{\pi}{2n}$ (with $n=6,9,15$ for $\coxE_6,\coxE_7,\coxE_8$ respectively), $y$ to $1$ (for $\coxE_6$) or $2\cos\frac{\pi}{9}$ (for $\coxE_7$) or $\gratio = 2\cos\frac{\pi}{5}$ (for $\coxE_8$), and $z$ to $\vw(v_i)$.
	
	\begin{lem} \label{lem:EBasis}
		The sets
		\begin{align*}
			\mathcal{B}^{\coxE_6}&=\{\croot_0^+,\croot_0^-, \croot_1^+, \croot_1^-, \croot_2,\croot_{v_6}\}, \\
			\mathcal{B}^{\coxE_7}&=\{\croot_0,\croot_1,\wt\croot_1,\croot_2,\wt\croot_2,\croot_3, \croot_{v_7}\}, \\
			\mathcal{B}^{\coxE_8}&=\{\croot_0,\gratio,\croot_1,\gratio\croot_1,\croot_2,\gratio\croot_2,\croot_{v_8},\gratio\croot_{v_8}\}
		\end{align*}
		are $\integer$-bases of the rings $\hechebr{6}$, $\hechebr{7}$ and $\hechebr{8}$ respectively.
	\end{lem}
	\begin{proof}
		It is not difficult to verify that the product of any two elements in $\hechebr{i}$ is a $\integer$-linear combination of elements in $\mathcal{B}^{\coxE_i}$ --- we leave this for the reader to check. It thus remains to show that the elements in $\mathcal{B}^{\coxE_i}$ are $\integer$-linearly independent. This largely follows from the linear independence Chebyshev polynomials of the second kind. Specifically for $\hechebr{6}$, this follows from the fact that
		\begin{align*}
			y &= U_4(\tfrac{x}{2}) - U_2(\tfrac{x}{2}), \\
			xy &= U_5(\tfrac{x}{2}) - U_1(\tfrac{x}{2}), \\
			z &= U_3(\tfrac{x}{2}) - U_5(\tfrac{x}{2})+U_1(\tfrac{x}{2})
		\end{align*}
		in the ring $\integer[x,y,z] / (S''_6)$. For $\hechebr{7}$, this follows from the fact that 
		\begin{align*}
			y &= U_4(\tfrac{x}{2})-U_6(\tfrac{x}{2})+U_2(\tfrac{x}{2}), \\
			U_2(\tfrac{y}{2}) &= U_6(\tfrac{x}{2})-U_2(\tfrac{x}{2}), \\
			z &= U_5(\tfrac{x}{2})-U_3(\tfrac{x}{2})
		\end{align*}
		in the ring $\integer[x,y,z] / (S''_7)$. For $\hechebr{8}$, this follows from the fact that
		\begin{align*}
			y&=U_6(\tfrac{x}{2}) - U_4(\tfrac{x}{2}), \\
			z&=U_5(\tfrac{x}{2})-U_7(\tfrac{x}{2})+U_3(\tfrac{x}{2}), \\
			xy &= U_7(\tfrac{x}{2})-U_3(\tfrac{x}{2}), \\
			U_2(\tfrac{x}{2})y &= U_4(\tfrac{x}{2}),  \\
			yz &= U_3(\tfrac{x}{2}).
		\end{align*}
		in the ring $\integer[x,y,z] / (S''_8)$. All of these relations can be obtained by computing Chebyshev polynomials of the second kind via the identity $U_{i+1}(\frac{x}{2}) = x U_i(\frac{x}{2}) - U_{i-1}(\frac{x}{2})$ and then by considering the relations in $S''_6$, $S''_7$ or $S''_8$.
	\end{proof}
	
	\subsubsection{Partial orderings on the (semi)rings} The relationship between the above (semi)rings and Chebyshev polynomials of the second kind is further highlighted by the existence of ring homomorphisms
	\begin{align*}
		\rhom{\gend}\colon \chebr{\gend} &\rightarrow \ZUi{2n} \\
		\hrhom{\gend}\colon \hchebr{\gend} &\rightarrow \ZUi{2n} \\
		\croot^\bullet_i &\mapsto U_i\left(\cos\frac{\pi}{2n}\right) 	&	
		(\gend &\in\{\coxA_{2n-1},\coxD_{n+1},\coxE_{6},\coxE_{7},\coxE_{8}\}), \\
		g &\mapsto 1 & (\gend &= \coxD_4), \\
		\wt\croot_i &\mapsto U_i\left(\cos\frac{\pi}{9}\right)	& (\gend &=\coxE_{7}), \\	
		\gratio &\mapsto 2\cos\frac{\pi}{5} & (\gend &=\coxE_{8}), \\
		\croot_{v_i} &\mapsto \vw(v_i) & (\gend &\in\{\coxE_{6},\coxE_{7},\coxE_{8}\}),
	\end{align*}
	where $n=6,9,15$ for $\gend=\coxE_6,\coxE_7,\coxE_8$ respectively. In particular, the homomorphisms $\hrhom{\gend}$ are epimorphisms (since, in each case, $\ZUi{2n}$ is a quotient of $\chebr{\gend}$), and $\rhom{\gend}$ is given by pre-composing $\hrhom{\gend}$ by the inclusion $\chebr{\gend}\rightarrow \hchebr{\gend}$. These homomorphisms along with the following will be used at various points in the paper.
	
	\begin{defn} \label{defn:SemiringPO}
		Let $\gend \in \{\coxA_{2n-1},\coxD_{n+1},\coxE_{6},\coxE_{7},\coxE_{8}\}$ and let $\realhom{2n}\colon \ZUi{2n} \rightarrow \real$ be the canonical embedding. Then $\chebr{\gend}$ has a partial ordering given by
		\begin{equation*}
			r \leq s \Leftrightarrow r=s \text{ or } \realhom{2n}\rhom{\gend}(r) < \realhom{2n}\rhom{\gend}(s).
		\end{equation*}
		This endows $\chebsr{\gend}$ with a partial ordering in the natural way via the embedding $\srhom{\gend}\colon \chebsr{\gend} \rightarrow \chebr{\gend}$.
	\end{defn}

	\section{The folding projections of the module and bounded derived categories}\label{sec:Projection}
	In \cite{DTOdd}, it is shown that a folding onto a quiver of type $\coxH_3$, $\coxH_4$ or $\coxI_2(2n+1)$ induces a map that projects the dimension vectors of objects in the module and derived categories of the unfolded quiver to certain multiples of the roots of the folded quiver. Additionally, the folding induces a weighting on the rows of the Auslander-Reiten quiver of the module and derived categories associated to the unfolded quiver such that the dimension vectors of the objects in the rows with weight 1 are projected to the roots of the folded quiver --- providing an analogue of Gabriel's Theorem (c.f. \cite{Gabriel,DlabRingelFinite}) for $R$-quivers. On the other hand, the dimension vectors of the objects in the rows with weight $w$ are mapped to the $w$-multiple of the projected dimension vector of an object in a corresponding row with weight 1. We will show that the same is true for all possible foldings onto quivers of type $\Iiin$. Throughout this section, we denote by $\Frac(R)$ to be the field of fractions of an integral domain $R$. First let us make precise the notion of the weighting of the rows of the Auslander-Reiten quiver.
	
	\begin{defn}\label{defn:RowWeights}
		Throughout, let $F\colon Q^{\gend} \rightarrow Q^{\Iiin}$ be a folding of quivers from Section~\ref{sec:Prelim-Folding} (so $\gend \in \{\coxA_{2n-1},\coxD_{n+1},\coxE_6,\coxE_7,\coxE_8\}$) and let $\acat$ be either the module or bounded derived category of the path algebra $KQ^\gend$. For each $i \in Q^\gend$, we denote by $\mathcal{P}^\acat_i$ the row of the Auslander-Reiten quiver of $\acat$ containing the module $P(i)$. That is,
	\begin{equation*}
		\mathcal{P}^\acat_i = \{[\tau^m P(i)] : m \in \integer\}
	\end{equation*}
	where $[M]$ denotes the iso-class of an object $M \in \acat$. We define the \emph{weight of the row $\mathcal{P}^\acat_i$} to be the value $\varepsilon_i=\frac{\vw(i)}{\mw(F(i))} \in \Frac(\ZUi{2n})$.
	\end{defn}
	
	\begin{rem}\label{defn:InjRow}
		One may note that since $KQ^\gend$ is representation-finite, we could equivalently define rows of the Auslander-Reiten quiver by sets
		\begin{equation*}
			\mathcal{I}^\acat_i = \{[\tau^m I(i)] : m \in \integer\}.
		\end{equation*}
		For the foldings with $\gend \in \{\coxD_{2m},\coxE_7,\coxE_8\}$, we have $\mathcal{P}^\acat_i = \mathcal{I}^\acat_i$. For foldings with $\gend =  \coxA_{2n-1}$, we note that $\mathcal{P}^\acat_i = \mathcal{I}^\acat_{2n-2-i}$, which mirrors the relation $U_{i}(\cos\frac{\pi}{2n}) = U_{2n-2-i}(\cos\frac{\pi}{2n})$. We also note in this case that $F(i)=F(2n-2-i)$. For the foldings with $\gend= \coxD_{n+1}$ with $n$ even, we have $\mathcal{P}^\acat_i = \mathcal{I}^\acat_{i}$ for all $i \neq (n-1)^\pm$ and $\mathcal{P}^\acat_{(n-1)^\pm} = \mathcal{I}^\acat_{(n-1)^\mp}$. For foldings with $\gend =  \coxE_6$, we note that $\mathcal{P}^\acat_{i^\pm} = \mathcal{I}^\acat_{i^\mp}$ for $i=0,1$, that $\mathcal{P}^\acat_{i} = \mathcal{I}^\acat_{i}$ for $i=2,v_6$, and that $F(i^\pm)=F(i^\mp)$. It is then easy to see that in all cases, we could equivalently define the weight of a row containing both $P(i)$ and $I(j)$ as $\varepsilon_i=\varepsilon_j$.
	\end{rem}
	
	Now we will recall the definition from \cite{DTOdd} of the projection map that will be central to this paper, which we modify slightly to account for the non-trivial $R$-valuation of $Q^{\Iiin}$.
	\begin{defn} \label{defn:ProjMap}
		Given an arbitrary weighted folding $F\colon Q \rightarrow Q^{\Iiin}$, define a map $d_F\colon \integer^{|Q_0|} \rightarrow (\Frac(\ZUi{2n}))^2$ by the weighted sum
		\begin{equation*}
			(u_i)_{i \in Q_0} \mapsto \left(\sum_{F(i)=[0]} \frac{\vw(i)}{\mw([0])} u_i, \sum_{F(i)=[1]} \frac{\vw(i)}{\mw([1])} u_i\right).
		\end{equation*}
		For each object $M \in \mod*KQ$, we define
		\begin{equation*}
			\dimproj_{F}^{\mod*KQ}(M) = d_F\dimvect(M).
		\end{equation*}
		In addition, for each object $X=\bigoplus_{j \in \integer} \sus^j M_j \in \bder(\mod*KQ)$, we define
		\begin{equation*}
			\dimproj_{F}^{\bder(\mod*KQ)}(X) = \sum_{j \text{ even}} \dimproj_{F}^{\mod*KQ}(M_j) - \sum_{j \text{ odd}} \dimproj_{F}^{\mod*KQ}(M_j).
		\end{equation*}
		Given any object $X \in \acat$, where $\acat$ is either $\mod*KQ$ or $\bder(\mod*KQ)$, we call the vector $\dimproj_{F}^{\acat}(X)$ the \emph{$F$-projected (dimension) vector of $X$}.
	\end{defn}
	
	\begin{rem} \label{rem:dimprojSums}
		It is easy to see that $\dimproj_F^\acat(X \oplus X') = \dimproj_F^\acat(X) + \dimproj_F^\acat(X')$ for any object $X,X' \in \acat$. This fact will be useful in several proofs.
	\end{rem}
	
	We will now state the main theorem of this section.
	\begin{thm} \label{thm:FoldingProjection}
		Let $F\colon Q^\gend \rightarrow Q^{\Iiin}$ be a weighted folding and let $\acat$ be either $\mod*KQ^\gend$ or $\bder(\mod*KQ^\gend)$. Then the following hold:
		\begin{enumerate}[label=(\alph*)]
			\item For any $i \in Q^\gend_0$ such that $F(i)=[0]$, we have
			\begin{equation*}
				\dimproj^\acat_F(\tau^m I(i)) = \varepsilon_i \alpha_m,
			\end{equation*}
			where $\alpha_m$ is the short root corresponding to the point $\mw([0])\mathrm{exp}(\frac{m \pi\imunit}{n})$.
			\item For any $i \in Q^\gend_0$ such that $F(i)=[1]$, we have
			\begin{equation*}
				\dimproj^\acat_F(\tau^m I(i)) = \varepsilon_i \beta_m,
			\end{equation*}
			where $\beta_m$ is the long root corresponding to the point $\mw([1])\mathrm{exp}(\frac{(2m+1) \pi\imunit}{2n})$.
		\end{enumerate}
	\end{thm}
	
	Essentially, the theorem above says that the map $\dimproj^\acat_F$ maps an indecomposable object in the category $\acat$ to a specific multiple of a root of the standard root system of $\Iiin$. In fact, it says much more than this: objects in the same row map to vectors of the same length/multiple and that the Auslander-Reiten translation of an object corresponds to a rotation of the corresponding vector.
	
	\begin{rem}
		As previously mentioned in Remark~\ref{rem:CoxeterPlane}, the projection of roots induced by the folding $F\colon Q^\gend \rightarrow Q^{\Iiin}$ is precisely the projection onto the Coxeter plane considered by Steinberg in \cite{SteinbergRefl}. It follows from the work of Steinberg, that a bipartite Coxeter element is a rotation on the Coxeter plane of order given by the Coxeter number. In our setting, the Coxeter number is $2n$ and the Auslander-Reiten translation is a bipartite Coxeter functor (\cite{Bernstein,BrennerButler}). This offers an explanation as to why the Auslander-Reiten translation of an object $X \in \acat$ corresponds to a rotation of the vector $\dimproj_F^{\acat}(X)$. An example of this is illustrated in Figure~\ref{fig:A7-I8}. 
	\end{rem}
	
	Before we begin the proof of Theorem~\ref{thm:FoldingProjection}, we will highlight some immediate corollaries for the benefit of the reader. The first Corollary may be viewed as a generalisation of Gabriel's Theorem to quivers of type $\Iiin$.
	
	\begin{cor} \label{cor:RowBijection}
		The map $\dimproj^\acat_F$ induces a bijection from each row of the Auslander-Reiten quiver of $\acat$ to a subset of roots of $\Iiin$. In particular,
		\begin{equation*}
			\dimproj^\acat_F(\mathcal{P}_i^\acat) = 
			\begin{cases}
				\Isproots{2n} & \text{if } \varepsilon_i = 1 \text{, } F(i)=[0] \text{ and } \acat=\mod*KQ^\gend, \\
				\Isroots{2n} & \text{if }\varepsilon_i = 1 \text{, } F(i)=[0]  \text{ and } \acat=\bder(\mod*KQ^\gend), \\
				\Ilproots{2n} & \text{if } \varepsilon_i = 1 \text{, } F(i)=[1]  \text{ and } \acat=\mod*KQ^\gend, \\
				\Ilroots{2n} & \text{if } \varepsilon_i = 1 \text{, } F(i)=[1]  \text{ and } \acat=\bder(\mod*KQ^\gend).
			\end{cases}
		\end{equation*}
	\end{cor}
	
	\begin{cor} \label{cor:RowWeights}
		For any $i \in Q_0^\gend$ and $m \in \integer$, the length of the vector $\dimproj^\acat_F(\tau^m P(i))$ with respect to the standard basis of $\real^2$ is
		\begin{equation*}
			\ell(\dimproj^\acat_F(\tau^m P(i))) = \vw(i).
		\end{equation*}
	\end{cor}
	
	The proof of these results follows the same reasoning as that used for the foldings onto $\coxI_2(2n+1)$ featuring in \cite{DTOdd}. The proof works along these lines: We first show that the the relative length of the $F$-projected vectors of any two indecomposable projective (resp. injective) modules that correspond to the same folded vertex is determined by the $R$-weights of the unfolded vertices. Next, we show that the Auslander-Reiten translate of an object corresponds to an anticlockwise rotation of its $F$-projected vector by an angle of $\frac{\pi}{n}$ about the origin. The proof is then finalised with the calculation of the $F$-projected vectors of two objects (one corresponding to a short root and the other a long root).
	
	\begin{lem}\label{lem:ProjWeights}
		For any $i,j\in Q^\gend_0$ such that $F(i)=F(j)$, we have
		\begin{align*}
			\vw(i) \dimproj^{\mod*KQ^\gend}_F(P(j)) &= \vw(j) \dimproj^{\mod*KQ^\gend}_F(P(i)) \\
		\vw(i) \dimproj^{\mod*KQ^\gend}_F(I(j)) &= \vw(j) \dimproj^{\mod*KQ^\gend}_F(I(i)).
		\end{align*}
	\end{lem}
	\begin{proof}
		One can see from the multiplication rule for Chebyshev polynomials of the second kind that
		\begin{align*}
			\dimproj^{\mod*KQ^\gend}_F(P(i)) &=
			\begin{cases}
				(0, \varepsilon_i)			& \text{if } P(i) \text{ is simple,} \\
				(\varepsilon_i,\varepsilon_i)		& \text{if }P(i) \text{ is non-simple.}
			\end{cases} \\
			\dimproj^{\mod*KQ^\gend}_F(I(i)) &=
			\begin{cases}
				(\varepsilon_i, 0)									& \text{if } I(i) \text{ is simple,} \\
				(4\varepsilon_i\cos^2\tfrac{\pi}{2n},\varepsilon_i)	& \text{if } I(i) \text{ is non-simple.}
			\end{cases}
		\end{align*}
		Thus, the result follows from the commutativity of $\ZUi{2n}$.
	\end{proof}
	
	\begin{rem} \label{rem:ProjLengths}
		Recall that $\ell((1,0)) = \mw([0]) = \lambda$ and $\ell((0,1)) = \mw([1]) = 2 \lambda \cos\tfrac{\pi}{2n}$, where $\lambda = 1$ if $\gend\neq\coxD_{n+1}$ and $\lambda = 2$ if $\gend=\coxD_{n+1}$. Thus by the law of cosines, the length of the vector $(1,1)$ is $\lambda$. From this, the proof of the above lemma, and the definition of the foldings, we can conclude that $\ell(\dimproj^\acat_F(P(i)))=\vw(i)$. Likewise, we have $\ell(\dimproj^\acat_F(I(i)))=\vw(i)$.
	\end{rem}
	
	\begin{lem}\label{lem:tauRotation}
		Let $M \in \acat$. Then $\dimproj^\acat_F(\tau M)$ is obtained from $\dimproj^\acat_F(M)$ by an anticlockwise rotation of $\frac{\pi}{n}$ about the origin.
	\end{lem}
	\begin{proof}
		Suppose $M \in \mod*KQ^\gend$ is a non-projective indecomposable. Let
		\begin{align*}
			0 \rightarrow P_1 \rightarrow &P_0 \rightarrow M \rightarrow 0, \\
			0 \rightarrow \tau M \rightarrow I_1 \rightarrow &I_0 \rightarrow 0
		\end{align*}
		be the minimal projective resolution of $M$ and the corresponding injective resolution of $\tau M$ given by applying the Nakayama functor $\nu$. Since $KQ^\gend$ is radical square zero, $P_1$ and $I_0$ are semisimple modules. Consequently, we have
		\begin{align*}
			\dimproj_F^\acat(P_1) &= (0,r'), & \dimproj_F^\acat(P_0) &= (r,r), \\
			\dimproj_F^\acat(I_1) &= (4s'\cos^2\tfrac{\pi}{2n},s'), & \dimproj_F^\acat(I_0) &= (s,0).
		\end{align*}
		for some $r,r',s,s' \in \Frac(\ZUi{2n})$. But since $\nu P(i)=I(i)$ for each $i \in Q^\gend_0$, we in fact have $r=s$ and $r'=s'$. This, along with Remark~\ref{rem:ProjLengths}, implies that
		\begin{equation*}
		\ell(\dimproj_F^\acat(P_0))=\ell(\dimproj_F^\acat(I_0)) \qquad \text{and} \qquad \ell(\dimproj_F^\acat(P_1))=\ell(\dimproj_F^\acat(I_1)).
		\end{equation*}
		Now note that the vector $-\dimproj_F^\acat(I_0)$ is obtained from $\dimproj_F^\acat(P_0)$ by an anticlockwise rotation of $\frac{\pi}{n}$ about the origin. Likewise,  $\dimproj_F^\acat(I_1)$ is obtained from $-\dimproj_F^\acat(P_1)$ by an anticlockwise rotation of $\frac{\pi}{n}$ about the origin. But such a rotation is a linear transformation, and we have
		\begin{align*}
			\dimproj_F^\acat(M) &= \dimproj_F^\acat(P_0) - \dimproj_F^\acat(P_1), \\
			\dimproj_F^\acat(\tau M) &= \dimproj_F^\acat(I_1) - \dimproj_F^\acat(I_0)
		\end{align*}
		by dimension counting the above projective/injective resolutions. Thus, $\dimproj_F^\acat(\tau M)$ is an anticlockwise rotation of the vector $\dimproj_F^\acat(M)$ by an angle of $\frac{\pi}{n}$ about the origin.
		
		Now suppose instead that $M \in \bder(\mod*KQ^\gend)$ corresponds to an indecomposable projective module concentrated in some degree $k$. Then $\tau M \cong \sus^{k+1} I$ for some indecomposable injective module $I$ whose corresponding vertex has the same weight as with $M$. So $\dimproj_F^\acat(\tau M)=-\dimproj_F^\acat(\sus^k I)$, which we have already shown to be the appropriate rotated vector. By Remark~\ref{rem:dimprojSums}, we therefore conclude that for any object $M \in \acat$, the vector $\dimproj^\acat_F(\tau M)$ is obtained from $\dimproj^\acat_F(M)$ by an anticlockwise rotation of $\frac{\pi}{n}$ about the origin.
	\end{proof}
	
	We can now prove Theorem~\ref{thm:FoldingProjection}.
	\begin{proof}[Proof of Theorem~\ref{thm:FoldingProjection}]
		(a) If $F(i) = [0]$ then $I(i)$ is simple. So $\dimproj_F^\acat(I(i))=(\varepsilon_i,0)$. The result then follows from Lemma~\ref{lem:tauRotation}.
		
		(b) If $F(i)=[1]$ then $I(i)$ is non-simple. As shown in the proof of Lemma~\ref{lem:tauRotation}, $\dimproj^\acat_F(I(i))$ is given by an anticlockwise rotation of the vector $-\dimproj^\acat_F(P(i))=(0,-\varepsilon_i)$ by $\frac{\pi}{n}$ about the origin. Thus, $\dimproj^\acat_F(I(i))=\varepsilon_i\beta_0$, where $\beta_0$ is the long root corresponding to the point $\mw([1])\mathrm{exp}(\frac{\pi\imunit}{2n})$. The result then follows from Lemma~\ref{lem:tauRotation}.
	\end{proof}
	
	\begin{figure} 
		\centering
		\input{Diagrams/A7-I8.tex}
		\caption{The folding $F^{\coxA_7}\colon Q^{\coxA_7} \rightarrow Q^{\coxI_2(8)}$. Top: The Auslander-Reiten quiver of $\mod*KQ^{\coxA_7}$. The category has $\integer_2$-symmetry determined by reflection in the dashed line. Bottom: The non-crystallographic projection of the Auslander-Reiten quiver of $\acat=\bder(\mod*KQ^{\coxA_7})$ under the map $\dimproj_F^\acat$, with irreducible morphisms superimposed. Objects in the same $\integer_2$-orbit map to the same point. Objects concentrated in odd (resp. even) degree map to the points labelled with (resp. without) $\Sigma$. One of the sectors of morphisms is dashed to indicate that they are not morphisms between objects of degrees $k$ and $k-1$, but rather between objects of degrees $k$ and $k+1$. Rays in the Auslander-Reiten quiver map to octagonal arcs in the projection (eg. the blue arrows).} \label{fig:A7-I8}
	\end{figure}
	
	\begin{rem}~\label{rem:HengProjection}
		The results of this section can be viewed as an `unfolded categorification' of the root system of type $\Iiin$. Equivalently, this can be approached from the perspective of a `folded categorification', as presented in \cite{Heng}. In particular, Theorem 4.11 of \cite{Heng} in the $\Iiin$ case is related to our Theorem~\ref{thm:FoldingProjection} with $\gend= \coxA_{2n-1}$ and $\acat = \mod*KQ^{\coxA_{2n-1}}$. Note, however, that Heng uses a two-copy version of the quiver $Q^{\coxA_{2n-1}}$. Essentially, this means that Heng's construction considers the rescaled root system of $\Iiin$ rather than the standard root system that we consider in our theorem.
	\end{rem}
	
	\section{The semiring actions on the module and bounded derived categories} \label{sec:Action}
	Each folding $F\colon Q^\gend \rightarrow Q^\Iiin$ and the non-crystallographic projection it induces gives rise to a semiring action of $R_+=\chebsr{\gend}$ on $\mod*KQ^\gend$ and $\bder(\mod*KQ^\gend)$ in the sense of \cite[Definition 6.1]{DTOdd}, where $\chebsr{\gend}$ is as defined in Section~\ref{sec:Prelim-Cheb}. As in \cite[Definition 6.1]{DTOdd}, we call the appropriate category an \emph{$R_+$-coefficient category} when it is equipped with a semiring action of $R_+$. Whilst the semiring actions are slightly different for each folding, the general principle underlying each action is the same. We will first define the semiring actions for $\acat=\mod* KQ^\gend$, as the semiring action on $\bder(\mod*KQ^\gend)$ is a straightforward extension of this action.
	
	\subsection{The action on iso-classes of objects} \label{sec:Action-isoclasses}
	For each positive root $\alpha$ of $\Iiin$, define a set of iso-classes
	\begin{equation*}
		\mathbf{M}_\alpha=\{[M] : M \in \acat\text{ indecomposable and } r' \dimproj_F^\acat(M)=r\alpha \text{ for some } r,r'\in\ZUi{2n}\}.
	\end{equation*}
	Since $Q^\gend$ is bipartite, it follows from Theorem~\ref{thm:FoldingProjection} that the iso-classes in $\mathbf{M}_\alpha$ bijectively correspond to objects in the same column of the Auslander-Reiten quiver. That is, $\mathbf{M}_\alpha=\{[\tau^m I(i)]: F(i)=j\}$, where $j=[0]$ if $\alpha\in\Isproots{2n}$ and corresponds to the point $\mw([0])\mathrm{exp}(\frac{m \pi\imunit}{n})$, and $j=[1]$ if $\alpha\in\Ilproots{2n}$ and corresponds to the point $\mw([1])\mathrm{exp}(\frac{(2m+1) \pi\imunit}{2n})$. In both cases, define $M_{i,\alpha}=\tau^m I(i)$. Hence $[M_{i,\alpha}] \in \mathcal{I}_i^\acat$ and $\mw(F(i))\dimproj_F^\acat(M_{i,\alpha})=\vw(i)\alpha$. Consequently, we have
	\begin{equation*}
		\mathbf{M}_\alpha=
		\begin{cases}
			\{[M_{i,\alpha}] : F(i)=[0]\} & \text{if }\alpha\in\Isproots{2n} \\
			\{[M_{i,\alpha}] : F(i)=[1]\} & \text{if }\alpha\in\Ilproots{2n}.
		\end{cases}
	\end{equation*}
	%By $\add \mathbf{M}_\alpha$, we denote the additive subcategory of $\mod*KQ^\gend$ consisting of all objects isomorphic to direct sums of objects whose iso-classes are in $\mathbf{M}_\alpha$.
	
	The semiring $R_+$ will act on $\mod*KQ^\gend$ in such a way that for any positive root $\alpha$, for any $[M] \in \mathbf{M}_\alpha$ and for any $r \in R_+$, the object $rM$ will be isomorphic to an object whose indecomposable direct summands belong to iso-classes in $\mathbf{M}_\alpha$. We will describe this precisely for each folding.
	
	\subsubsection{Isomorphism conditions of type $\coxA_{2n-1}$}
	We will adopt the unsigned labelling of the vertices for $Q^{\coxA_{2n-1}}$ (the first quiver in Section~\ref{sec:Prelim-Folding}). For any positive root $\alpha$, we have a function $\omega_\alpha\colon \mathbf{M}_\alpha \rightarrow \hachebsr{2n-1}$ defined by $\omega_\alpha([M_{i,\alpha}])=\croot_i$ for each $[M_{i,\alpha}] \in \mathbf{M}_\alpha$. Now recall the product rule for the semiring $\hachebsr{2n-1}$, which gives
	\begin{equation*}
		\croot_j \omega_\alpha([M_{i,\alpha}]) =\sum_{k \in V_{ji}} \croot_k
	\end{equation*}
	for each $\croot_j \in \achebsr{2n-1}$ and for some index set $V_{ji} \subset \{0,1,\ldots,2n-2\}$. Since the index $j$ is even, note that the set $V_{ji}$ consists of indices that are either all odd if $i$ is odd, or all even if $i$ is even. In particular, this implies that there exists $[M_{k,\alpha}] \in \mathbf{M}_\alpha$ for each $k\in V_{ji}$. We will thus define the semiring action of $\achebsr{2n-1}$ on $\mod*KQ^{\coxA_{2n-1}}$ such that the following isomorphism holds for each positive root $\alpha$, each $[M_{i,\alpha}] \in\mathbf{M}_\alpha$ and each $\croot_j \in \achebsr{2n-1}$.
	\begin{equation}\tag{A1} \label{eq:AIso}
		 \croot_j M_{i,\alpha} \cong \bigoplus_{k \in V_{ji}} M_{k,\alpha}.
	\end{equation}
	
	An example demonstrating the action of $\achebsr{2n-1}$ on $\mod*KQ^{\coxA_{2n-1}}$ is given in Section~\ref{sec:AExample}.
	
	\subsubsection{Isomorphism conditions of type $\coxD_4$}
	The semiring action of $\dchebsr{4}$ on $\mod*KQ^{\coxD_4}$ is a minor extension of the classical $\integer_3$-group action on $\mod*KQ^{\coxD_4}$. That is we define for any positive root $\alpha$
	\begin{align*}
		g M_{0,\alpha} &\cong M_{2^+,\alpha}, & g M_{2^+,\alpha} &\cong M_{2^-,\alpha}, & g M_{2-,\alpha} &\cong M_{0,\alpha}, \\
		g M_{1,\alpha} &\cong M_{1,\alpha}.
	\end{align*}
	The action of $\dchebsr{4}$ is only different from the action of $\integer_3$ in that we have $(r+r') M \cong rM \oplus r'M$ for any $r,r' \in \dchebsr{4}$ and $M \in \mod*KQ^{\coxD_4}$.
	
	\subsubsection{Isomorphism conditions of type $\coxD_{n+1}$}
	We assume $n \geq 4$, since $\coxD_3 = \coxA_3$ and $\coxD_4$ is covered separately. For any positive root $\alpha$, we define a map $\omega_\alpha\colon\mathbf{M}_\alpha \rightarrow \hdchebsr{n+1}$ by
	\begin{equation*}
		\omega_\alpha([M_{i,\alpha}])=
		\begin{cases}
			\croot_i^+ + \croot_i^- 	& \text{if } 0 \leq i < n-1 \\
			\croot_{n-1}^\pm		& \text{if } i = (n-1)^\pm.
		\end{cases}
	\end{equation*}
	
	For each $[M_{i,\alpha}] \in \mathbf{M}_\alpha$ with $i \neq (n-1)^\pm$, for each $0 \leq j \leq n-1$, and for each $0 \leq k \leq n-1$, there exist constants $q_{ijk} \in \{0,1,2\}$ such that
	\begin{align*}
		\croot^\pm_j \omega_\alpha([M_{i,\alpha}]) &=\sum_{k=0}^{n-1} q_{ijk}(\croot^+_k + \croot_k^-) \\
		&=\sum_{k=0}^{n-2} q_{ijk}\omega_\alpha([M_{k,\alpha}])+ q_{ij(n-1)}(\omega_\alpha([M_{(n-1)^+,\alpha}])+\omega_\alpha([M_{(n-1)^-,\alpha}])).
	\end{align*}
	In particular, $q_{ijk}=0$ for all even $k$ if $i$ is odd and $q_{ijk}=0$ for all odd $k$ if $i$ is even. We will thus define the semiring action of $\dchebsr{n+1}$ on $\mod*KQ^{\coxD_{n+1}}$ to be such that the following isomorphism holds for each positive root $\alpha$, each $[M_{i,\alpha}] \in\mathbf{M}_\alpha$ with $i \neq (n-1)^\pm$ and each $\croot^\pm_j \in \dchebsr{n+1}$.
	\begin{equation} \tag{D1} \label{eq:DIso1}
		\croot^\pm_j M_{i,\alpha} \cong \bigoplus_{k=0}^{n-2} M_{k,\alpha}^{\oplus q_{ijk}} \oplus (M_{(n-1)^+,\alpha} \oplus M_{(n-1)^-,\alpha})^{\oplus q_{ij(n-1)}},
	\end{equation}
	where $X^{\oplus m} = \bigoplus_{l=1}^m X$ and $X^{\oplus 0} = 0$ for any $X \in \acat$. This is a slight abuse of notation, as $M_{k,\alpha}$ might not exist if $q_{ijk}=0$. However, this is not a problem, as we define $M_{k,\alpha}^{\oplus 0} = 0$.
	
	On the other hand, if $i=(n-1)^\pm$ then 
	\begin{align*}
		\croot^\pm_j \omega_\alpha([M_{(n-1)^\pm,\alpha}]) &= 
		\begin{cases}
			\sum_{k \in V_{ji}} (\croot^+_k + \croot_k^-) + \croot_{n-1}^+
			& \text{if } \frac{j}{2} \text{ even}, \\
			\sum_{k \in V_{ji}} (\croot^+_k + \croot_k^-) + \croot_{n-1}^-
			& \text{if } \frac{j}{2} \text{ odd}, \\
		\end{cases} \\
		\croot^\mp_j \omega_\alpha([M_{(n-1)^\pm,\alpha}]) &= 
		\begin{cases}
			\sum_{k \in V_{ji}} (\croot^+_k + \croot_k^-) + \croot_{n-1}^-
			& \text{if } \frac{j}{2} \text{ even}, \\
			\sum_{k \in V_{ji}} (\croot^+_k + \croot_k^-) + \croot_{n-1}^+
			& \text{if } \frac{j}{2} \text{ odd}. \\
		\end{cases} \\
	\end{align*}
	for some index set $V_{ji} \subset \{0,\ldots,n-1\}$ whose elements are odd if $n-1$ is odd, or even if $n-1$ is even. Thus, we will define the action of $\dchebsr{n+1}$ on $\mod*KQ^{\coxD_{n+1}}$ to be such that
	\begin{equation} \tag{D2} \label{eq:DIso2}
		\croot^\pm_j M_{i,\alpha} \cong
		\begin{cases}
			\bigoplus_{k \in V_{ji}} M_{k,\alpha} \oplus M_{(n-1)^+,\alpha} 
			& \text{if } i=(n-1)^\pm \text{ and } \frac{j}{2} \text{ even} \\
			& \text{or if } i=(n-1)^\mp \text{ and } \frac{j}{2} \text{ odd,}\\
			\bigoplus_{k \in V_{ji}} M_{k,\alpha} \oplus M_{(n-1)^-,\alpha}
			& \text{if } i=(n-1)^\mp \text{ and } \frac{j}{2} \text{ even} \\
			& \text{or if } i=(n-1)^\pm \text{ and } \frac{j}{2} \text{ odd.}
		\end{cases}
	\end{equation}
	
	An example demonstrating the action of $\dchebsr{n+1}$ on $\mod*KQ^{\coxD_{n+1}}$ is given in Section~\ref{sec:DExample}.
	
	\subsubsection{Isomorphism conditions of type $\coxE_6$}
	Let $\alpha$ be a positive root of $\coxI_2(12)$ and let $\mathbf{M}_\alpha$ be as before. Define a map $\omega_\alpha\colon\mathbf{M}_\alpha \rightarrow \hechebsr{6}$ by
	\begin{equation*}
		\omega_\alpha([M_{i,\alpha}])=
		\begin{cases}
			\croot_i^\pm			& \text{if } i \in \{0^\pm,1^\pm\}, \\
			\croot_i				& \text{if } i \in \{2,v_6\}.
		\end{cases}	
	\end{equation*}
	Then for each $i \in \{0^\pm, 2\}$, the products $\croot_0^-\omega_\alpha([M_{i,\alpha}])$ and $\croot_2\omega_\alpha([M_{i,\alpha}])$ will each be linear combinations of elements from the set 
	\begin{equation*}
		\{\croot_0^+, \croot_0^-, \croot_2\} = \{\omega_\alpha([M_{i,\alpha}]): F(i)=[0]\}.
	\end{equation*}
	On the other hand, for each $i \in \{1^\pm, v_6\}$, the products $\croot_0^-\omega_\alpha([M_{i,\alpha}])$ and $\croot_2\omega_\alpha([M_{i,\alpha}])$ will each be linear combinations of elements from the set 
	\begin{equation*}
		\{\croot_1^+, \croot_1^-, \croot_{v_6}\} = \{\omega_\alpha([M_{i,\alpha}]): F(i)=[1]\}.
	\end{equation*}
	We will thus define the action of $R_+$ on $\mod* KQ^{\coxE_6}$ to be such that the following isomorphisms hold.
	\begin{align}
		\croot_0^- M_{i,\alpha} &\cong 
		\begin{cases}
			M_{j^\mp,\alpha} & \text{if } i=j^\pm \text{ with } j \in \{0,1\}, \\
			M_{i,\alpha} & \text{otherwise,}
		\end{cases} \tag{E6.1}\label{eq:E6Iso1}\\
		\croot_2 M_{i,\alpha} &\cong 
		\begin{cases}
			M_{2,\alpha} & \text{if } i=0^\pm, \\
			M_{1^+,\alpha} \oplus M_{v_6,\alpha} \oplus M_{1^-,\alpha} & \text{if } i=1^\pm, \\
			M_{0^+,\alpha} \oplus M_{2,\alpha} \oplus M_{2,\alpha} \oplus M_{0^-,\alpha} & \text{if } i=2,\\
			M_{1^+,\alpha} \oplus M_{1^-,\alpha} & \text{if } i=v_6,
		\end{cases} \tag{E6.2}\label{eq:E6Iso2}
	\end{align}
	In particular, this corresponds to the following products in $\hechebr{6}$.
	\begin{align*}
		\croot_0^-\croot_0^\pm &= \croot_0^\mp, &  
		\croot_0^-\croot_1^\pm &= \croot_1^\mp, & 
		\croot_0^-\croot_2 &= \croot_2, & 
		\croot_0^-\croot_{v_6} &= \croot_{v_6}, \\
		\croot_2\croot_0^\pm &= \croot_2, &
		\croot_2\croot_1^\pm &= \croot_1^+ + \croot_{v_6} + \croot_1^-, &
		\croot_2\croot_2 &= \croot_0^+ +\croot_2 + \croot_2 + \croot_0^-, &
		\croot_2\croot_{v_6} &= \croot_1^+ + \croot_1^-.
	\end{align*}
	
	\subsubsection{Isomorphism conditions of type $\coxE_7$}
	For each positive root $\alpha$ of $\coxI_2(18)$, define $\omega_\alpha\colon\mathbf{M}_\alpha \rightarrow \hechebsr{7}$ by
	\begin{equation*}
		\omega_\alpha([M_{i,\alpha}])=
		\begin{cases}
			\croot_i				& \text{if } i \in \{0,1,2,3,v_7\}, \\
			\wt\croot_i			& \text{if } i \in \{\wt{1},\wt{2}\}.
		\end{cases}	
	\end{equation*}
	We can see from Table~\ref{tab:E7} that we again have that, for each $i$ such that $F(i)=[j]$, the products $\croot_2\omega_\alpha([M_{i,\alpha}])$ and $\wt\croot_2\omega_\alpha([M_{i,\alpha}])$ are both given by linear combinations of elements from the set $\{\omega_\alpha([M_{i,\alpha}]): F(i)=[j]\}$. Thus, the isomorphism conditions for the folding $F^{\coxE_7}$ are such that for any positive root $\alpha$, we have the following.
	\begin{align}
		\croot_2 M_{i,\alpha} &\cong 
		\begin{cases}
			M_{2,\alpha} & \text{if } i=0, \\
			M_{1,\alpha} \oplus M_{3,\alpha} & \text{if } i=1, \\
			M_{0,\alpha} \oplus M_{\widetilde{1},\alpha} \oplus M_{\widetilde{2},\alpha} \oplus M_{2,\alpha} & \text{if } i=2, \\
			M_{1,\alpha} \oplus M_{3,\alpha} \oplus M_{3,\alpha} \oplus M_{v_7,\alpha} & \text{if } i=3, \\
			M_{\widetilde{2},\alpha} \oplus M_{2,\alpha} & \text{if } i=\widetilde{1}, \\
			M_{\widetilde{1},\alpha} \oplus M_{\widetilde{2},\alpha} \oplus M_{2,\alpha} & \text{if } i=\widetilde{2}, \\
			M_{3,\alpha} & \text{if } i=v_7,
		\end{cases} \tag{E7.1}\label{eq:E7Iso1}\\
		\widetilde{\croot}_2 M_{i,\alpha} &\cong 
		\begin{cases}
			M_{\widetilde{2},\alpha} & \text{if } i=0, \\
			M_{3,\alpha} \oplus M_{v_7,\alpha} & \text{if } i=1, \\
			M_{\widetilde{1},\alpha} \oplus M_{\widetilde{2},\alpha} \oplus M_{2,\alpha} & \text{if } i=2, \\
			M_{1,\alpha} \oplus M_{3,\alpha} \oplus M_{3,\alpha} & \text{if } i=3, \\
			M_{\widetilde{1},\alpha} \oplus M_{2,\alpha} & \text{if } i=\widetilde{1}, \\
			M_{0,\alpha} \oplus M_{\widetilde{2},\alpha} \oplus M_{2,\alpha} & \text{if } i=\widetilde{2}, \\
			M_{1,\alpha} \oplus M_{v_7,\alpha} & \text{if } i=v_7,
		\end{cases} \tag{E7.2}\label{eq:E7Iso2}
	\end{align}
	It is easy to verify that the above is compatible with the relations $\wt\croot_2^2 = 1 + \wt\croot_2 + \croot_2$ and $\croot_2^2 = 1 + \croot_2\wt\croot_2$. It is also straightforward to check that these actions correspond to the products $\croot_2\omega_\alpha([M_{i,\alpha}])$ and $\wt\croot_2\omega_\alpha([M_{i,\alpha}])$ in Table~\ref{tab:E7}.
	
	\subsubsection{Isomorphism conditions of type $\coxE_8$}
	For any positive root $\alpha$ of $\coxI_2(30)$, we define a map $\omega_\alpha\colon\mathbf{M}_\alpha \rightarrow \hechebsr{8}$ by
	\begin{equation*}
		\omega_\alpha([M_{i,\alpha}])=
		\begin{cases}
			\croot_i 				& \text{if } i \in \{0,1,2,v_8\}, \\
			\gratio\croot_{v_8}	& \text{if } i =3, \\
			\gratio\croot_2		& \text{if } i =4, \\
			\gratio\croot_j		& \text{if } i =\phi_j.
		\end{cases}	
	\end{equation*}
	Once again, for each $i$ such that $F(i)=[j]$, the products $\croot_2\omega_\alpha([M_{i,\alpha}])$ and $\gratio\omega_\alpha([M_{i,\alpha}])$ are given by linear combinations of elements in the set $\{\omega_\alpha([M_{i,\alpha}]) : F(i)=[j]\}$. Thus, the isomorphism conditions for this folding are such that for any positive root $\alpha$, we have the following.
	\begin{align}
		\croot_2 M_{i,\alpha} &\cong 
		\begin{cases}
			M_{2,\alpha} & \text{if } i=0, \\
			M_{1,\alpha} \oplus M_{3,\alpha} & \text{if } i=1, \\
			M_{0,\alpha} \oplus M_{2,\alpha} \oplus M_{4,\alpha} & \text{if } i=2, \\
			M_{1,\alpha} \oplus M_{3,\alpha} \oplus M_{v_8,\alpha} \oplus M_{\phi_1,\alpha} & \text{if } i=3, \\
			M_{2,\alpha} \oplus M_{4,\alpha} \oplus M_{4,\alpha} \oplus M_{\phi_0,\alpha} & \text{if } i=4, \\
			M_{3,\alpha} \oplus M_{\phi_1,\alpha} & \text{if } i=v_8, \\
			M_{3,\alpha} \oplus M_{v_8,\alpha} \oplus M_{\phi_1,\alpha} & \text{if } i=\phi_1, \\
			M_{4,\alpha} & \text{if } i=\phi_0,
		\end{cases} \tag{E8.1}\label{eq:E8Iso1}\\
		\varphi M_{i,\alpha} &\cong 
		\begin{cases}
			M_{\phi_0,\alpha} & \text{if } i=0, \\
			M_{\phi_1,\alpha} & \text{if } i=1, \\
			M_{4,\alpha} & \text{if } i=2, \\
			M_{3,\alpha} \oplus M_{v_8,\alpha} & \text{if } i=3, \\
			M_{2,\alpha} \oplus M_{4,\alpha} & \text{if } i=4, \\
			M_{3,\alpha} & \text{if } i=v_8, \\
			M_{1,\alpha} \oplus M_{\phi_1,\alpha} & \text{if } i=\phi_1, \\
			M_{0,\alpha} \oplus M_{\phi_0,\alpha} & \text{if } i=\phi_0,
		\end{cases} \tag{E8.2}\label{eq:E8Iso2}
	\end{align}
	We leave it as an exercise to the reader to check that these actions are consistent with the products in $\hechebr{8}$.

	\subsection{The action on morphisms} \label{sec:Action-Morphisms}
	The action of $R_+$ on morphisms in $\mod*KQ^\gend$ is best described by adopting a particular basis for the modules. Under this basis, the semiring can be seen to act diagonally on morphisms. This requires us to establish additional notation for an explicit explanation. As in \cite[Section 6]{DTOdd}, let $\{\eta_i : i \in Q^\gend_0\}$ be a complete set of primitive orthogonal idempotents of $KQ^\gend$ and let $N$ be a $KQ^\gend$-module. Let $N_{i,j}\colon N_i \rightarrow N_j$ be the restriction of the linear action of the unique arrow $a_{i,j}\colon i \rightarrow j \in Q^\gend_1$ on $N$ to the vector subspace $N_i=N\eta_i$, composed with the canonical surjection onto $N_j=N\eta_j$. The data given by all vector subspaces $N_i$ and all linear maps $N_{{i,j}}$ entirely determine the structure of $N$ --- this is equivalently the module $N$ expressed as a $K$-representation of $Q^\gend$, which in this setting, is more convenient to work with. For any $L\in {\calM_F}$ and morphism $f \in \Hom_{\calM_F}(N,L)$, we then recall by Schur's lemma that $f$ can be written diagonally as $(f_i)_{i \in Q^\gend_0}$ with $f_i=f|_{N_i} \colon N_i \rightarrow L_i$.
	
	To define the appropriate basis of $\croot_j N$, first let $[M_{i,\alpha}] \in \mathbf{M}_\alpha$ for some positive root $\alpha$ of $\Iiin$ and write
	\begin{equation*}
		\croot_j M_{i,\alpha} \cong \bigoplus_{k \in Q_0^\gend} \bigoplus_{l=1}^{q_{ijk}} M_{k,\alpha}
	\end{equation*}
	for some non-negative integers $q_{ijk}$ determined by the relevant isomorphism conditions. Then we define the vector subspace $(\croot_j N)_i$ of $\croot_j N$ to be the vector space
	\begin{equation*}
		(\croot_j N)_i = \bigoplus_{k \in Q_0^\gend} \bigoplus_{l=1}^{q_{ijk}} N\ps{l}_{k},
	\end{equation*}
	where each $N\ps{l}_{k}$ is a copy of the vector subspace $N_{k} \subseteq N$.
	For each arrow $a\colon i \rightarrow i'$ in $Q^\gend$, the linear map $(\croot_j N)_{i,i'}\colon (\croot_j N)_{i}\rightarrow (\croot_j N)_{i'}$ is defined blockwise by linear maps
	\begin{equation*}
		N\ps{l}_k \rightarrow N\ps{l'}_{k'}\colon v \mapsto \zeta_{a,k,k'}\ps{l,l'} N_{k,k'}(v)
	\end{equation*}
	for some constants $\zeta_{a,k,k'}\ps{l,l'} \in K$ determined such that the isomorphism conditions hold. Under such a basis, one can define the action of $\croot_j$ on a morphism $f=(f_i)_{i \in Q_0^\gend}\colon N\rightarrow L$ by the morphism $\croot_j f \colon \croot_j N\rightarrow \croot_j L$ given by 
	\begin{equation*}
		(\croot_j f)_i = \bigoplus_{k \in Q_0^\gend} \bigoplus_{l=1}^{q_{ijk}} f_l.
	\end{equation*}
	Finally, we define
	\begin{equation*}
		(r+r')f\colon (r+r') N \rightarrow (r+r') L = rf \oplus r'f\colon rN \oplus r'N \rightarrow rL \oplus r'L
	\end{equation*}
	for any $r,r' \in R_+$, and we define $\croot_0 = 1 \in R_+$ to act by the identity functor.
	
	\begin{rem} \label{rem:HengAction}
		For the case where $\gend = \coxA_{2n-1}$, the semiring action on the category $\mod*KQ^\gend$ can also be described via the equivalence given in Theorem 3.1 (and Section 6.3) of \cite{Heng}.
	\end{rem}
	
	\subsection{Basic properties and actions on the bounded derived category} \label{sec:Action-Properties}
	Let $F\colon Q^{\gend} \rightarrow Q^{\coxI_2(2n)}$ be a weighted folding. Since the action of $R_+$ on $\mod* KQ^{\gend}$ has been defined such that certain isomorphism conditions are met and that the actions on morphisms are diagonalisable with respect to a given basis, it is easy to see that $\calM_\gend = \mod* KQ^{\gend}$ has the structure of an $R_+$-coefficient category. In particular, the action of any $r \in R_+$ is exact, faithful, and induces an injective map $\Ext^1_{\calM_\gend}(M,N) \rightarrow \Ext^1_{\calM_\gend}(rM,rN)$. As is the case in \cite{DTOdd}, it is straightforward to extend this action to the bounded derived category $\bder_\gend = \bder(\calM_\gend)$. Since $r$ acts by an exact functor on $\calM_\gend$, we can define $r \sus M =\sus rM$ for each object $M$. One then notes that the faithful action of $R_+$ on $\Ext^1_{\calM_\gend}$-spaces ensures that the $R_+$-action on $\bder_\gend$ is faithful. One can then see that $\bder_\gend$ is also an $R_+$-coefficient category.
	
	Throughout, we let $\mathcal{A}$ be either of the $R_+$-coefficient categories $\calM_\gend$ or $\bder_\gend$.  We have the following basic properties.
	
	\begin{rem} \label{rem:ActionOnARQ}
		Let $M \in \mathcal{A}$ be indecomposable and $r \in R_+$. Write $rM \cong M_1 \oplus \ldots \oplus M_m$, where each $M_i$ is indecomposable.
		\begin{enumerate}[label=(\alph*)]
			\item $M_1, \ldots, M_m$ reside in the same column of the Auslander-Reiten quiver. In particular, $r$ commutes with the Auslander-Reiten translate $\tau$.
			\item The Auslander-Reiten quiver starting (resp. ending) in $M$ is mapped under $r$ to the direct sum of each Auslander-Reiten sequence starting (resp. ending) in $M_i$.
		\end{enumerate}
	\end{rem}
	
	In \cite{DTOdd}, the notion of basic and minimal $R_+$-generators was defined. We will briefly recall those here.
	\begin{defn} \label{defn:RPGeneration}
		Suppose $\acat$ is an abelian $R_+$-coefficient category and let $R$ be the ring obtained from $R_+$ by applying the Grothendieck group construction on the additive commutative monoid structure of $R_+$. We say that an object $N \in \mathcal{A}$ is \emph{$R_+$-generated by $M \in \acat$} if for some $r,r' \in R_+$, there exists a split exact sequence
		\begin{equation*}
			0\rightarrow r'M \rightarrow rM \rightarrow N \rightarrow 0.
		\end{equation*} 
		Similarly for a triangulated $R_+$-coefficient category $\acat$, we say an object $N \in \acat$ is \emph{$R_+$-generated by $M \in \acat$} if for some $r,r' \in R_+$, there exists a split triangle
		\begin{equation*}
			r'M \rightarrow rM \rightarrow N \rightarrow \sus r'M
		\end{equation*}
		with monic first morphism. In both the abelian and triangulated settings, we can equivalently say that $N$ is $R_+$-generated by $M \in \acat$ if $rM \cong r'M \oplus N$. We call the pair $(r,r')$ the \emph{$R_+$-generating pair of $N$ by $M$} and call the value $r-r' \in R$ the \emph{$R$-index of $N$ with respect to $M$}. We denote the class of all objects $R_+$-generated by $M$ in a category $\acat$ by $\gen_M$.
		
		Given a set $\Gamma$ of objects of $\acat$, we denote by $\acat(\Gamma)$ the full subcategory of $\acat$ whose objects are $R_+$-generated by the objects in $\Gamma$. The objects of $\acat(\Gamma)$ may be endowed with a partial ordering by defining $N_1 \leq N_2$ if and only if $N_1 \cong N_2$ or there exists $M \in \Gamma$ such that $N_1$ is $R_+$-generated by $M$ with the pair $(r_1,r'_1)$ and $N_2$ is $R_+$-generated by $M$ with the pair $(r_2,r'_2)$ and $r_1+r'_2 < r_2 + r'_1$.
		
		We say $\Gamma$ is a set of \emph{$R_+$-generators for $\acat$} if $\acat(\Gamma) \simeq \acat$. We say $\Gamma$ is \emph{basic} if the elements of $\Gamma$ are pairwise non-isomorphic indecomposable objects. We say a basic set of $R_+$-generators $\Gamma$ is \emph{minimal} if for any other basic set of $R_+$-generators $\Gamma'$, there exists an injective map of sets $\theta\colon \Gamma \rightarrow \Gamma'$ and inclusions $\iota\colon \Gamma \rightarrow \Gamma \cup \Gamma'$ and $\iota'\colon \Gamma' \rightarrow \Gamma \cup \Gamma'$ such that $\iota(M) \leq \iota' \theta(M)$ for each $M \in \Gamma$. We say $\Gamma$ is \emph{$\tau$-closed} if for any (non-projective if $\acat$ is abelian) $M \in \Gamma$, we have $\tau M \in \Gamma$.
	\end{defn}
	
	For the foldings outlined in this paper, there is a close relationship between the projection map $\dimproj^\acat_F$ and $R_+$-generated objects. The proof of the following remarks is identical to \cite[Section 6]{DTOdd}.
	
	\begin{rem}\label{rem:ActionProps}
		Recall that for each folding $F^\gend$ we have $R_+ = \chebsr{\gend}$. Let $\srhom{\gend}\colon \chebsr{\gend} \rightarrow \chebr{\gend}$ be the canonical embedding and let $\rhom{\gend}\colon \chebr{\gend} \rightarrow \ZUi{2n}$ be the homomorphism from Section~\ref{sec:Prelim-Cheb}. Let $\ell(v)$ denote the length of the vector $v \in \real^2$ with respect to the standard basis of $\real^2$
		\begin{enumerate}[label=(\alph*)]
			\item For any $M \in \acat$ and any $r \in \chebsr{\gend}$, we have $\ell(\dimproj^\acat_F(rM)) = \rhom{\gend}\srhom{\gend}(r) \ell(\dimproj^\acat_F(M))$ and $\dimproj^\acat_F(rM)$ is collinear to $\dimproj^\acat_F(M)$.
			\item Suppose $\acat$ has a set of $R_+$-generators and that $N_1,N_2 \in \acat$ are indecomposable. Suppose further that $\acat = \calM_\gend$. Then $N_1$ and $N_2$ have a common $\chebsr{\gend}$-generator if and only if $\dimproj^\acat_F(N_1)$ and $\dimproj^\acat_F(N_2)$ are collinear. Furthermore, $N_1 < N_2$ if and only if $\dimproj^\acat_F(N_1)$ and $\dimproj^\acat_F(N_2)$ are collinear, and $\ell(\dimproj^\acat_F(N_1)) < \ell(\dimproj^\acat_F(N_2))$. 
			
			On the other hand, if $\acat = \bder_\gend$, then $N_1$ and $N_2$ have a common $\chebsr{\gend}$-generator if and only if $N_1$ and $N_2$ are concentrated in the same degree and $\dimproj^\acat_F(N_1)$ and $\dimproj^\acat_F(N_2)$ are collinear. Furthermore, $N_1 < N_2$ if and only if $N_1$ and $N_2$ are concentrated in the same degree, $\dimproj^\acat_F(N_1)$ and $\dimproj^\acat_F(N_2)$ are collinear, and $\ell(\dimproj^\acat_F(N_1)) < \ell(\dimproj^\acat_F(N_2))$.
		\end{enumerate}
	\end{rem}
	
	One of the main results of \cite{DTOdd} is that the $R_+$-coefficient categories that arise from foldings onto quivers of type $\coxH_4$, $\coxH_3$ and $\coxI_2(2n+1)$ all have basic, minimal, $\tau$-closed sets of $R_+$-generators that are unique up to isomorphic elements. The same is not always true for $R_+$-coefficient categories that arise from foldings onto quivers of type $\coxI_2(2n)$ (except for the foldings $F^{\coxE_7}$ and $F^{\coxE_8}$). One always has sets of basic, $\tau$-closed $R_+$-generators, but the condition that is lost is the minimality condition. The problem in this setting is, in effect, the category is too big. One has non-isomorphic indecomposable objects that map to the same (multiple of a) root, which leads to some objects being incomparable under the partial ordering. This happens precisely when there is a $\integer_2$-action on the category and when an object is not stable under this action. So to fix this, we need a weaker notion of minimality that takes into account the group action.
	
	\begin{lem}\label{lem:RPGenGroupAction}
		Let $G \subseteq U(R_+)$ be a subgroup of the group $U(R_+)$ of multiplicative units of $R_+$. Suppose that there exists a set $\Gamma$ of $R_+$-generators for $\acat$. Then for each $g \in G$ and $M \in \Gamma$, the set
		\begin{equation*}
			\Gamma^{gM} =(\Gamma \setminus \{M\}) \cup \{gM\}
		\end{equation*}
		is a set of $R_+$-generators for $\acat$
	\end{lem}
	\begin{proof}
		Since $G$ is a multiplicative group embedded in the semiring $R_+$, it follows that $M$ is $R_+$-generated by $gM$ by the pair $(g\inv,0)$. This recovers the original set $\Gamma$. Hence, $\acat(\Gamma^{gM}) \simeq \acat(\Gamma) \simeq \acat$, as required.
	\end{proof}
	
	The above lemma implies that we have an equivalence relation $\sim_G$ on sets of $R_+$-generators whenever there is a group embedded in $R_+$. That is, we say that $\Gamma \sim_G \Gamma'$ if and only if for each $M \in \Gamma$, there exists $g_M \in G$ such that $\Gamma' = \{g_M M : M \in \Gamma\}$.
	
	\begin{defn}
		Let $G \subseteq U(R_+)$ be a subgroup of the group $U(R_+)$ of multiplicative units of $R_+$. We say that a basic set of $R_+$-generators $\Gamma$ is \emph{$G$-minimal} if for any basic set of $R_+$-generators $\Gamma' \not\sim_G \Gamma$, there exists an injective map of sets $\theta\colon \Gamma \rightarrow \Gamma'$ and inclusions $\iota\colon \Gamma \rightarrow \Gamma \cup \Gamma'$ and $\iota'\colon \Gamma' \rightarrow \Gamma \cup \Gamma'$ such that $\iota(M) \leq \iota' \theta(M)$ for each $M \in \Gamma$. 
	\end{defn}
	
	From the above definition, we could equivalently define a minimal set of $R_+$-generators $\Gamma$ to be one that is $\integer_1$-minimal, where $\integer_1$ is the trivial subgroup of $U(R_+)$. For each semiring $\chebsr{\gend}$, we will fix a subgroup $G_{\gend} \subseteq U(\chebsr{\gend})$. In particular, we will define
	\begin{equation*}
		G_\gend =
		\begin{cases}
		\{1,\croot_{2n-2}\} \cong \integer_2	&\text{if } \gend = \coxA_{2n-1},\\
		\{1,g, g^2\} \cong \integer_3			&\text{if } \gend = \coxD_{4},\\
		\{1,\croot_0^-\} \cong \integer_2		&\text{if } \gend \in \{\coxD_{n+1},\coxE_6\},\\
		\{1\} \cong \integer_1					&\text{if } \gend \in \{\coxE_6,\coxE_7\}.
		\end{cases}
	\end{equation*}
	It is not difficult to see from Lemma~\ref{lem:ChebPolyProps}(f) and the homomorphism $\rhom{\gend}\srhom{\gend}$ that we actually have $G_\gend = U(\chebsr{\gend})$ in each case. Moreover, there exists no non-trivial proper subgroup $G' \subset G_\gend$ in each case..
	
	\begin{thm} \label{thm:RPGenerators}
		Let $F^\gend\colon Q^\gend \rightarrow Q^{\coxI_2(2n)}$ be the weighted folding of type $\gend$. Then there exists a basic, $\tau$-closed set of $R_+$-generators of $\acat$, where $R_+=\chebsr{\gend}$. In particular, we have the following.
		\begin{enumerate}[label=(\alph*)]
			\item If $\gend=\coxA_{2n-1}$, then there exist four distinct (up to isomorphic elements) such sets of $\chebsr{\coxA_{2n-1}}$-generators that are $\integer_2$-minimal. Namely,
			\begin{equation*}
				\Gamma_{i,j} = \{M \in \acat : [M] \in \mathcal{P}_{i}^\acat \cup \mathcal{P}_{j}^\acat\}
			\end{equation*}
			with $i \in \{0,2n-2\}$ and $j \in \{1,2n-3\}$. These are pairwise equivalent under $\sim_{\integer_2}$.
			\item If $\gend=\coxD_{4}$, then there exist three distinct (up to isomorphic elements) such sets of $\dchebsr{4}$-generators that are $\integer_3$-minimal. Namely,
			\begin{equation*}
				\Gamma_i = \{M \in \acat : [M] \in \mathcal{P}_{i}^\acat \cup \mathcal{P}_{1}^\acat\},
			\end{equation*}
			with $i \neq 1$. In addition, $\Gamma_i \sim_{\integer_3} \Gamma_j$ for any distinct $i,j\neq 1$.
			\item If $\gend=\coxD_{n+1}$, then there exist two distinct (up to isomorphic elements) such sets of $\chebsr{\coxD_{n+1}}$-generators that are $\integer_2$-minimal. Namely,
			\begin{equation*}
				\Gamma^\pm = \{M \in \acat : [M] \in \mathcal{P}_{i}^\acat \cup \mathcal{P}_{(n-1)^\pm}^\acat\},
			\end{equation*}
			where $i = 0$ if $n$ is even and $i=1$ if $n$ is odd. In addition, $\Gamma^+ \sim_{\integer_2} \Gamma^-$.
			\item If $\gend=\coxE_{6}$, then there exist four distinct (up to isomorphic elements) such sets of $\chebsr{\coxE_{6}}$-generators that are $\integer_2$-minimal. Namely,
			\begin{align*}
				\Gamma\ps{\pm,\pm} &= \{M \in \acat : [M] \in \mathcal{P}_{0^\pm}^\acat \cup \mathcal{P}_{1^\pm}^\acat\}, \\
				\Gamma\ps{\pm,\mp} &= \{M \in \acat : [M] \in \mathcal{P}_{0^\pm}^\acat \cup \mathcal{P}_{1^\mp}^\acat\}.
			\end{align*}
			These are pairwise equivalent under $\sim_{\integer_2}$.
			\item If $\gend=\coxE_{7}$, then there exists a unique minimal such set of $\chebsr{\coxE_{7}}$-generators (up to isomorphic elements). Namely,
			\begin{equation*}
				\Gamma = \{M \in \acat : [M] \in \mathcal{P}_{0}^\acat \cup \mathcal{P}_{v_7}^\acat\}.
			\end{equation*}
			\item If $\gend=\coxE_{8}$, then there exists a unique minimal such set of $\chebsr{\coxE_{8}}$-generators (up to isomorphic elements). Namely,
			\begin{equation*}
				\Gamma = \{M \in \acat : [M] \in \mathcal{P}_{0}^\acat \cup \mathcal{P}_{1}^\acat\}.
			\end{equation*}
		\end{enumerate}
	\end{thm}
	\begin{proof}
		(a) We will first show that $\Gamma_{0,1}$ is a set of $R_+$-generators. We have
		\begin{align*}
			\mathbf{M}_{\alpha} &= \{[M_{0,\alpha}], [M_{2,\alpha}],\ldots,[M_{2n-2,\alpha}]\}	& (\alpha &\in \Isproots{2n}), \\
\mathbf{M}_{\beta} &= \{[M_{1,\beta}], [M_{3,\beta}],\ldots,[M_{2n-3,\beta}]\}	& (\beta &\in \Ilproots{2n}).
		\end{align*}
		It is easy to see from (\ref{eq:AIso}) that $\croot_{2i} M_{0,\alpha} \cong M_{2i,\alpha}$. Thus the iso-classes of $\mathbf{M}_{\alpha}$ are $R_+$-generated by $M_{0,\alpha}$. To see that the iso-classes of $\mathbf{M}_{\beta}$ are $R_+$-generated by $M_{1,\beta}$, first define $r_1 = \croot_2$ and $r'_1 = 1$. It then follows from (\ref{eq:AIso}) that $M_{2i+1,\beta}$ is $R_+$-generated by $(r_{2i+1},r'_{2i+1})$, where $r_{2i+1} = \croot_{2i} + r'_{2i-1}$ and $r'_{2i+1} = r_{2i-1}$. Thus by Theorem~\ref{thm:FoldingProjection} and Remark~\ref{rem:ActionOnARQ}, $\Gamma_{0,1}$ is a set of $R_+$-generators for $\acat$, which is clearly also basic and $\tau$-closed.
	
		It remains to show that $\Gamma_{0,1}$ is $\integer_2$-minimal. By (\ref{eq:AIso}), we have $\integer_2 \cong \{1, \croot_{2n-2}\} =G_{\coxA_{2n-1}}$ as $\croot_{2n-2} M_{i, \gamma} \cong M_{2n-1-i,\gamma}$ for any $i$ and positive root $\gamma$. From this and Lemma~\ref{lem:RPGenGroupAction}, we can see that $\Gamma_{0,1} \sim_{\integer_2} \Gamma_{2n-2,1} \sim_{\integer_2} \Gamma_{0,2n-3} \sim_{\integer_2} \Gamma_{2n-2,2n-3}$. Moreover, $\ell(\dimproj^\acat_{F^\gend}(M_{0,\alpha})) < \ell(\dimproj^\acat_{F^\gend}(M_{2i,\alpha}))$ for all $i \neq n-1$ and $\ell(\dimproj^\acat_{F^\gend}(M_{1,\beta})) < \ell(\dimproj^\acat_{F^\gend}(M_{2i+1,\beta}))$ for all $i \neq n-2$. So each $\Gamma_{i,j}$ is $\integer_2$-minimal by Remark~\ref{rem:ActionProps}, as required. 
		
		(b) Here, we have $\integer_3 \cong \{1,g,g^2\} = G_{\coxD_4}$. This is thus a consequence of the sole generating element $g \in \dchebsr{4}$ being the $\integer_3$-group action on the category, along with the fact that $\ell(\dimproj_{F^\gend}^\acat(M_{0,\alpha}) ) = \ell(\dimproj_{F^\gend}^\acat(M_{2^+,\alpha}) ) = \ell(\dimproj_{F^\gend}^\acat(M_{2^-,\alpha}) )$ for each short positive root $\alpha$.
		
		(c) We will first show that in both the odd and even cases, the sets $\Gamma^\pm$ are both basic, $\tau$-closed sets of $R^+$-generators that are equivalent under the group action $\integer_2 \cong \{1,\croot^-_0\} = G_{\coxD_{n+1}}$. Suppose $n$ is even. In this case we have
		\begin{align*}
			\mathbf{M}_{\alpha} &= \{[M_{0,\alpha}], [M_{2,\alpha}],\ldots,[M_{n-2,\alpha}]\}	& (\alpha &\in \Isproots{2n}), \\
\mathbf{M}_{\beta} &= \{[M_{1,\beta}], [M_{3,\beta}],\ldots,[M_{(n-1)^+,\beta}],[M_{(n-1)^-,\beta}]\}	& (\beta &\in \Ilproots{2n}).
		\end{align*}
		It is straightforward to see from (\ref{eq:DIso1}) that $M_{0,\alpha}$ $R_+$-generates the iso-classes in $\mathbf{M}_{\alpha}$ -- the argument is identical to that used in type $\coxA$. We can also see that $M_{(n-1)^\pm,\beta}$ both $R_+$-generate the iso-classes in $\mathbf{M}_{\beta}$. This follows from (\ref{eq:DIso2}), where we can see that each $M_{n-1-2i, \beta}$ is $R_+$-generated from $M_{(n-1)^\pm,\beta}$ by the pair $(\croot^+_{2i},\croot^-_{2(i-1)})$. Thus $\Gamma^\pm$ are both sets of $R_+$-generators of $\acat$.
		
		Now suppose $n$ is odd. In this case we have
		\begin{align*}
			\mathbf{M}_{\alpha} &= \{[M_{0,\alpha}], [M_{2,\alpha}],\ldots,[M_{(n-1)^+,\alpha}],[M_{(n-1)^-,\alpha}]\}	& (\alpha &\in \Isproots{2n}), \\
\mathbf{M}_{\beta} &= \{[M_{1,\beta}], [M_{3,\beta}],\ldots,[M_{n-2,\beta}]\}	& (\beta &\in \Ilproots{2n}).
		\end{align*}
		By (\ref{eq:DIso2}), $M_{(n-1)^\pm,\alpha}$ $R_+$-generates the iso-classes in $\mathbf{M}_{\alpha}$ in the same way as $M_{(n-1)^\pm,\beta}$ does for $\mathbf{M}_{\beta}$ in the even case. By (\ref{eq:DIso1}), $M_{1,\beta}$ $R_+$-generates the iso-classes in $\mathbf{M}_{\beta}$ in a similar way to how $M_{1,\beta}$ does for $\mathbf{M}_{\beta}$ in type $\coxA$ --- the pairs are the same except with $\pm$ superscripts. Thus, $\Gamma^\pm$ are again both sets of $R_+$-generators of $\acat$.
		
		It now remains to show that these sets of $R_+$-generators satisfy the required properties. That they are basic and $\tau$-closed follows trivially by construction. The equivalence under $\sim_{\integer_2}$ is also obvious, as $\croot_0^- M_{i,\gamma} = M_{i,\gamma}$ for any $\gamma$ and any $i \neq (n-1)^\pm$, and $\croot_0^- M_{(n-1)^\pm,\gamma} = M_{(n-1)^\mp,\gamma}$. Thus, it remains to show that these sets are $\integer_2$-minimal. Here, it is important to note the following: if $[M_{(n-1)^\pm,\gamma}] \in \mathbf{M}_{\gamma}$ for some $\gamma$, then every $M_{i,\gamma}$ (with $i\neq (n-1)^\pm$) does not $R_+$-generate $M_{(n-1)^\pm,\gamma}$. This follows from (\ref{eq:DIso1}), where we can see that for any object $N$ that is $R_+$-generated by $M_{i,\gamma}$, $M_{(n-1)^+,\gamma}$ is isomorphic to direct summand of $N$ if and only if $M_{(n-1)^-,\gamma}$ is isomorphic to direct summand of $N$. That is, it may be possible to $R_+$-generate $M_{(n-1)^+,\gamma} \oplus M_{(n-1)^-,\gamma}$ from $M_{i,\gamma}$, but not one of these summands individually. So even though we may have $M_{i,\gamma} < M_{(n-1)^\pm,\gamma}$, this is irrelevant, as any set that does not contain one of $M_{(n-1)^\pm,\gamma}$ is not a set of $R_+$-generators. As for the set of iso-classes such that $[M_{(n-1)^\pm,\gamma}] \not\in \mathbf{M}_{\gamma}$, the $\integer_2$-minimality of $M_{0,\gamma}$ or $M_{1,\gamma}$ is obvious from Remark~\ref{rem:ActionProps}. Hence, $\Gamma^\pm$ are both $\integer_2$-minimal.
		
		(d) Let $\alpha$ be a short positive root and $\beta$ be a long positive root of $\coxI_2(12)$. Then we can see $\Gamma^{(i,j)}$ (with $i,j \in \{\pm,\mp\}$) is a set of $R_+$-generators by the following table of $R_+$-generating pairs (where columns are $R_+$-generated by rows with pairs in the given cell).
		\begin{center}
			\begin{tabular}{ c | c | c | c | c | c | c }
				& $M_{0^\pm,\alpha}$ & $M_{0^\mp,\alpha}$ & $M_{2,\alpha}$ & $M_{1^\pm,\beta}$  & $M_{1^\mp,\beta}$ & $M_{v_6,\beta}$ \\ \hline
				$M_{0^\pm,\alpha}$ & $(1,0)$ & $(\croot^-_0,0)$ & $(\croot_2,0)$ & & & \\ \hline
				$M_{1^\pm,\beta}$ & & & & $(1,0)$ & $(\croot^-_0,0)$ & $(\croot_2,1 + \croot^-_0)$
			\end{tabular}
		\end{center}
		By Theorem~\ref{thm:FoldingProjection} and Remark~\ref{rem:ActionOnARQ}, this is basic and $\tau$-closed, as required. The group is $\integer_2\cong \{1, \croot_0^-\} = G_{\coxE_6}$, and thus it is clear from the table that these sets of $R_+$-generators are pairwise equivalent under $\sim_{\integer_2}$. To see that these are $\integer_2$-minimal, we note that
		\begin{align*}
			\ell(\dimproj_{F^\gend}^\acat(M_{0^+,\alpha}) ) &= \ell(\dimproj_{F^\gend}^\acat(M_{0^-,\alpha}) ) < \ell(\dimproj_{F^\gend}^\acat(M_{2,\alpha}) ), \\
			\ell(\dimproj_{F^\gend}^\acat(M_{v_6,\beta}) ) < \ell(\dimproj_{F^\gend}^\acat(M_{1^+,\beta}) ) &= \ell(\dimproj_{F^\gend}^\acat(M_{1^-,\beta}) ),
		\end{align*}
		but also that $M_{v_6,\beta}$ does not $R_+$-generate either $M_{1^\pm,\beta}$ (the best we can do is $R_+$-generate $M_{1^+,\beta} \oplus M_{1^-,\beta}$, which is insufficient). Hence, each $\Gamma^{(i,j)}$ is $\integer_2$-minimal by Remark~\ref{rem:ActionProps}, as required.
		
		(e) Let $\alpha$ be a short positive root and $\beta$ be a long positive root of $\coxI_2(18)$. Then we can see $\Gamma$ is a set of $R_+$-generators by the following table of $R_+$-generating pairs.
		\begin{center}
			\renewcommand{\arraystretch}{1.5}
			\begin{tabular}{ c | c | c | c | c | c | c | c }
				& $M_{0,\alpha}$ & $M_{\wt{1},\alpha}$ & $M_{\wt{2},\alpha}$ & $M_{2,\alpha}$ & $M_{v_7,\beta}$  & $M_{1,\beta}$ & $M_{3,\beta}$ \\ \hline
				$M_{0,\alpha}$ & $(1,0)$ & $(\croot_2\wt\croot_2,\croot_2 + \wt\croot_2)$ & $(\wt\croot_2,0)$ & $(\croot_2,0)$ & & \\ \hline
				$M_{v_7,\beta}$ & & & & & $(1,0)$ & $(\wt\croot_2,1)$ & $(\croot_2,0)$
			\end{tabular}
			\renewcommand{\arraystretch}{1}
		\end{center}
		By Theorem~\ref{thm:FoldingProjection} and Remark~\ref{rem:ActionOnARQ}, $\Gamma$ is basic and $\tau$-closed, as required. It also follows from Remark~\ref{rem:ActionProps} that $\Gamma$ is minimal, since
		\begin{align*}
			\ell(\dimproj_{F^\gend}^\acat(M_{0,\alpha}) ) &< \ell(\dimproj_{F^\gend}^\acat(M_{\wt{1},\alpha}) ) < \ell(\dimproj_{F^\gend}^\acat(M_{\wt{2},\alpha}) ) < \ell(\dimproj_{F^\gend}^\acat(M_{2,\alpha}) ),  \\
			\ell(\dimproj_{F^\gend}^\acat(M_{v_7,\beta}) ) &< \ell(\dimproj_{F^\gend}^\acat(M_{1,\beta}) ) < \ell(\dimproj_{F^\gend}^\acat(M_{3,\beta}) ),
		\end{align*}
		as required.
		
		(f) Let $\alpha$ be a short positive root and $\beta$ be a long positive root of $\coxI_2(30)$. Then we can see $\Gamma$ is a set of $R_+$-generators by the following table of $R_+$-generating pairs.
		\begin{center}
			\begin{tabular}{ c | c | c | c | c | c | c | c | c }
				& $M_{0,\alpha}$ & $M_{2,\alpha}$ & $M_{4,\alpha}$ & $M_{\phi_0,\alpha}$ & $M_{1,\beta}$  & $M_{3,\beta}$ & $M_{\phi_1,\beta}$ & $M_{v_8,\beta}$ \\ \hline
				$M_{0,\alpha}$ & $(1,0)$ & $(\croot_2,0)$ & $(\gratio\croot_2,0)$ & $(\gratio,0)$ & & & \\ \hline
				$M_{1,\beta}$ & & & & & $(1,0)$ & $(\croot_2,1)$ & $(\gratio,0)$ & $(\gratio\croot_2+1,\croot_2+\gratio)$
			\end{tabular}
		\end{center}
		By Theorem~\ref{thm:FoldingProjection} and Remark~\ref{rem:ActionOnARQ}, $\Gamma$ is basic and $\tau$-closed, as required. It also follows from Remark~\ref{rem:ActionProps} that $\Gamma$ is minimal, since
		\begin{align*}
			\ell(\dimproj_{F^\gend}^\acat(M_{0,\alpha}) ) &< \ell(\dimproj_{F^\gend}^\acat(M_{\phi_0,\alpha}) ) < \ell(\dimproj_{F^\gend}^\acat(M_{2,\alpha}) ) < \ell(\dimproj_{F^\gend}^\acat(M_{4,\alpha}) ),  \\
			\ell(\dimproj_{F^\gend}^\acat(M_{1,\beta}) ) &< \ell(\dimproj_{F^\gend}^\acat(M_{v_8,\beta}) ) < \ell(\dimproj_{F^\gend}^\acat(M_{\phi_1,\beta}) ) < \ell(\dimproj_{F^\gend}^\acat(M_{3,\beta}) ),
		\end{align*}
		as required.
	\end{proof}
	
	\begin{cor}\label{cor:GenCorrespondence}
		Let $F^\gend$ be a folding onto $\Iiin$ and let $\Gamma$ be a basic, $\tau$-closed set of minimal (or $G_\gend$-minimal) set of $R_+$-generators of $\calM_\gend$. Then the elements of $\Gamma$ are in bijective correspondence with the positive roots of $\Iiin$.
	\end{cor}
	\begin{proof}
		We have by construction (and Theorem~\ref{thm:FoldingProjection}) that the sets $\mathbf{M}_\alpha$ (with $\alpha \in \Iproots{2n}$) partition the iso-classes of indecomposable objects of $\calM_\gend$. By Theorem~\ref{thm:RPGenerators}, $\Gamma = \{N_\alpha : \alpha \in \Iproots{2n}\}$, where each $[N_\alpha] \in \mathbf{M}_\alpha$. Thus, the elements of $\Gamma$ are in bijective correspondence with the roots in $\Iproots{2n}$, as required.
	\end{proof}
	
	An interesting consequence of the above Theorem is the following.
	\begin{rem} \label{rem:GrothendieckModule}
		Let $R_+ \subset R$, where $R$ is the ring obtained by performing the Grothendieck group construction on the additive commutative monoid structure of $R_+$. Then the Grothendieck group $K_0(\calM_\gend)$ of the category $\calM_\gend$ has the structure of an $R$-module. For each iso-class $[M] \in K_0(\calM_\gend)$, mulitplication by $r - r' \in R$ (where $r,r' \in R_+$) is given by $(r-r')[M]=[rM] - [r'M]$. Note, however, that unlike the situation in \cite{DTOdd}, the $R$-module $K_0(\calM_\gend)$ is not free.
	\end{rem}
	
	\section{Cluster-$R_+$-tilting theory for foldings onto $\Iiin$}\label{sec:ClusterTilting}
	Cluster algebras and mutations of Dynkin $\integer$-quivers have previously been categorified via the cluster category (constructed in \cite{BMRRT}, see \cite{ClusterTiltingSurvey} for further details). The semiring action on the category $\bder_\gend$ naturally extends to the cluster category $\clus_\gend = \bder_\gend / \langle\dart \sus\inv\rangle$ of $Q^\gend$. In particular, we have a canonical triangulated quotient functor $E\colon \bder_\gend \rightarrow \clus_\gend$ (due to \cite{KellerOrbitCats}) and we define the action of $R_+$ on $\clus_\gend$ by $rE = Er$ for each $r \in R_+$. Given any set of $R_+$-generators $\Gamma_{\bder_{\gend}}$ of $\bder_{\gend}$, we obtain a set of $R_+$-generators
	\begin{equation*}
		\Gamma_{\clus_{\gend}} = \{E(M) : M \in \Gamma_{\bder_{\gend}}\}
	\end{equation*}
	of $\clus_\gend$. Moreover, $\Gamma_{\clus_{\gend}}$ is basic if $\Gamma_{\bder_{\gend}}$ is basic and $\Gamma_{\clus_{\gend}}$ is $G$-minimal if $\Gamma_{\bder_{\gend}}$ is $G$-minimal. For worked examples of the theory that follows, see Section~\ref{sec:Examples}.
	
	\begin{defn}
		Let $H$ be a finite-dimensional hereditary algebra such that $\mod*H$ is has the structure of an $R_+$-coefficient category. Let $\clus_H$ be the cluster category of $H$, and suppose that $\Gamma$ is a set of $R_+$-generators of $\clus_H$. We say an object $T \in \clus_H$ is \emph{(cluster)-$R_+$-tilting} with respect to $\Gamma$ if the following hold.
		\begin{enumerate}[label=(T\arabic*)]
			\item $T \cong \bigoplus_{i=1}^{|T|} T_i$ with each $T_i \in \Gamma$.
			\item $T$ is $R_+$-rigid: $\Ext_{\clus_\gend}^1(T', T'')=0$ for any $T',T'' \in \gen_T$.
			\item $T$ is maximal: If there exists $X \in \Gamma$ such that $T \oplus X$ is $R_+$-rigid, then $X$ is isomorphic to a direct summand of $T$.
		\end{enumerate}
		In addition, we say an $R_+$-tilting object $T \in \clus_H$ is \emph{basic} if the direct summands of $T$ are pairwise non-isomorphic and $T$ is with respect to a basic, $G$-minimal set of $R_+$-generators, where $G$ is some (possibly trivial) group embedded in $U(R_+)$.
	\end{defn}
	
	Henceforth, we will assume that $\Gamma$ is a basic, $\tau$-closed, $G_\gend$-minimal set of $\chebsr{\gend}$-generators of $\clus_\gend$ that corresponds to a basic, $\tau$-closed, $G_\gend$-minimal set of $\chebsr{\gend}$-generators $\Gamma'$ of $\bder_\gend$ under the quotient functor $E\colon \bder_\gend \rightarrow \clus_\gend$.
	
	\begin{thm} \label{thm:RPTilting}
		Let $T \cong X_1 \oplus \ldots \oplus X_m \in \clus_\gend$ with each $X_i \in \Gamma$. For each $i$, denote by $\mathbf{I}_{X_i}$ the set of all iso-classes of indecomposable objects in $\clus_\gend$ that are $\chebsr{\gend}$-generated by $X_i$. Then the following are equivalent.
		\begin{enumerate}[label=(\alph*)]
			\item $T$ is basic $\chebsr{\gend}$-tilting with respect to $\Gamma$.
			\item $T \cong X_1 \oplus X_2$, where $X_1,X_2 \in \Gamma$ reside in (distinct) adjacent columns of the Auslander-Reiten quiver of $\clus_\gend$.
			\item The object
			\begin{equation*}
				\wh{T} \cong \bigoplus_{i=1}^m\bigoplus_{[Y_i] \in \mathbf{I}_{X_i}} Y_i \in \clus_\gend
			\end{equation*}
			is basic tilting.
		\end{enumerate}
	\end{thm}
	\begin{proof}
		(c) $\Rightarrow$ (b): Since $\wh{T}$ is basic, the direct summands of $\wh{T}$ are pairwise non-isomorphic. By construction, this is true only if we have $X_i \not\cong X_j$ for all $i \neq j$. By Remark~\ref{rem:ActionOnARQ}, the iso-classes in $\mathbf{I}_{X_i}$ are precisely the iso-classes of indecomposable objects that reside in the same column as $X_i$ in the Auslander-Reiten quiver of $\clus_\gend$. Thus we conclude that $\wh{T}$ is basic only if each $X_i$ resides in a distinct column of the Auslander-Reiten quiver of $Q_0^\gend$. To see that $m = 2$, we note that if $m \neq 2$, then $|\wh{T}| \neq |Q_0^\gend|$, and thus cannot be tilting by the results of \cite{BMRRT}. Thus, $T \cong X_1 \oplus X_2$.
		
		It remains to show that $X_1$ and $X_2$ reside in adjacent columns. Suppose for a contrapositive argument that this is not the case. Then at least one of $X_1$ and $X_2$ is represented by an object $M_1 \in \bder_\gend$ concentrated in degree $0$ --- for convenience, we will say (without loss of generality) that this is a representative of $X_1$. Then by assumption, $X_2$ is represented either by another object $M_2$ concentrated in degree $0$, or by an object $\sus P(i) \in \bder_\gend$ for some $i \in Q_0^\gend$.
		
		In the latter case, if $P(i)$ is simple (or equivalently, $i$ is a sink vertex), then $M_1 \not\cong S(j)$ for any source vertex $j \in Q_0^\gend$ (otherwise, $X_1$ and $X_2$ are adjacent). Thus we must have $S(k) \subseteq M_1$ for some sink vertex $k \in Q_0^\gend$ and hence $\Ext_{\bder_\gend}^1(\sus S(k), M_1) = \Hom_{\bder_\gend}(\sus S(k), \sus M_1) \neq 0$. But each $\sus S(k)$ ($k$ a source vertex) is a representative of some object $Y_2 \in \clus_\gend$ in the same column of the Auslander-Reiten quiver as $X_2$ (that is, $[Y_2] \in \mathbf{I}_{X_2}$). So $\Ext^1_{\bder_\gend}(\wh{T}, \wh{T}) \neq 0$, which implies $\Ext^1_{\clus_\gend}(\wh{T}, \wh{T}) \neq 0$, and hence $\wh{T}$ is not tilting, as required. On the other hand, if $P(i)$ is not simple (or equivalently, $i$ is a source vertex), then $M_1 \not\cong S(j)$ for any sink vertex $j \in Q_0^\gend$ (otherwise, $X_1$ and $X_2$ are again adjacent). In this subcase, we trivially have $\Ext_{\bder_\gend}^1(\sus P, M_1) = \Hom_{\bder_\gend}(\sus P, \sus M_1) \neq 0$ for some indecomposable object $P$ in the same column as $P(i)$, as this occurs as a component of a projective cover of $M_1$ in the category $\calM_\gend$. Thus we again have $\Ext^1_{\clus_\gend}(\wh{T}, \wh{T}) \neq 0$, which implies $\wh{T}$ is not tilting.
		
		Now we consider the former case where we have representatives $M_1$ and $M_2$ both concentrated in degree $0$. Suppose without loss of generality that $M_2$ resides in a column strictly to the right of $M_1$ (that is, in the direction of the irreducible morphisms in the Auslander-Reiten quiver). Then $M_1$ does not correspond to a injective object in $\calM$ and $M_2$ does not correspond to a projective object in $\calM$ (otherwise we are forced to have $X_1$ and $X_2$ in adjacent columns). Thus by the Auslander-Reiten formula,
		\begin{equation*}
			\dim_K\Ext_{\bder_\gend}^1(M'_2,M'_1) = \dim_K\Hom_{\bder_\gend}(M'_2,\sus M'_1) = \dim_K\Hom_{\calM_\gend}^1(M'_1,\tau M'_2),
		\end{equation*}
		where each $M'_i$ is some object in the same column of the Auslander-Reiten quiver as $M_i$. This quantity must be non-zero for some choice of $M'_1$ and $M'_2$, since $\tau M'_2$ is not injective and thus has a simple projective subobject $P$ such that the inclusion of $P$ into $\tau M'_2$ factors through $M'_1$. This implies that we must have $\Ext_{\clus_\gend}^1(\wh{T},\wh{T}) \neq 0$ in this case, which implies $\wh{T}$ is not tilting, as required. This was the final case to consider, and hence we conclude that $X_1$ and $X_2$ must reside in (distinct) adjacent columns of the Aulsander-Reiten quiver.
		
		(b) $\Rightarrow$ (a): Let $Y_1 \in \gen_{X_1}$ and $Y_2 \in \gen_{X_2}$ be indecomposable and let $M_1,M_2 \in \bder_\gend$ be representatives of $Y_1$ and $Y_2$ respectively. Note that by Remark~\ref{rem:ActionOnARQ} that $Y_1$ resides in the same column as $X_1$ and $Y_2$ resides in the same column as $X_2$ in the Auslander-Reiten quiver of $\clus_\gend$.
		
		If $M_1$ and $M_2$ are concentrated in the same degree, then they also reside in adjacent columns of the Auslander-Reiten quiver of $\bder_\gend$, and clearly we have
		\begin{equation*}
			\Ext^1_{\bder_\gend}(M_1,M_2)=\Ext^1_{\bder_\gend}(M_2,M_1)=0
		\end{equation*}
		
		So assume $M_1$ and $M_2$ are concentrated in the degrees $k$ and $k+1$ (where we assume without loss of generality that $M_1$ is concentrated in degree $k$), then clearly we have $\Ext^1_{\bder_\gend}(M_1,M_2) = \Hom_{\bder_\gend}(M_1,\sus M_2) = 0$, since $\bder_\gend$ is hereditary. To compute $\Ext^1_{\bder_\gend}(M_2,M_1) = \Hom_{\bder_\gend}(M_2,\sus M_1)$, we note that $M_2$ and $\sus M_1$ are concentrated in the same degree. By construction, there are two possibilities: either $M_2$ and $\sus M_1$ are non-isomorphic simple objects (which occurs only if $M_1$ and $M_2$ reside in adjacent columns of $\bder_\gend$) or $\sus M_1$ necessarily resides in a column to the left of $M_2$ in the Auslander-Reiten quiver of $\bder_\gend$ (since $Y_1$ and $Y_2$ reside in adjacent columns in the Auslander-Reiten quiver of $\clus_\gend$ and $\sus \simeq \tau$ in $\clus_\gend$). In both cases, we trivially have $\Ext^1_{\bder_\gend}(Y_2,Y_1)=0$.
		
		Now assume $M_1$ and $M_2$ are concentrated in the degrees $k$ and $k+l$ with $l \geq 2$ respectively. Again we have $\Ext^1_{\bder_\gend}(M_1,M_2) = 0$ from the fact that $\bder_\gend$ is hereditary and that $\Ext^1_{\bder_\gend}(M_2,M_1) = 0$ since $\sus M_1$ resides in a column to the left of $M_2$ in the Auslander-Reiten quiver of $\bder_\gend$. 
		
		Since $\Gamma$ is a set of $\chebsr{\gend}$-generators, and every object $\chebsr{\gend}$-generated by $X_i$ resides in the same column as $X_i$, and we have made an argument based on arbitrary indecomposables in each column, we have $\Ext_{\clus_\gend}(X'_1,X'_2)=\Ext_{\clus_\gend}(X'_2,X'_1)=0$ for any $X'_1 \in \gen_{X_1}$ and $X'_2 \in \gen_{X_2}$. By similar arguments to the above, one also has $\Ext_{\clus_\gend}(X'_1,X'_1)=\Ext_{\clus_\gend}(X'_2,X'_2)=0$ for any $X'_1 \in \gen_{X_1}$ and $X'_2 \in \gen_{X_2}$. Thus, $T$ must be $\chebsr{\gend}$-rigid. It is also basic, since $X_1$ and $X_2$ are non-isomorphic by assumption. That $T$ is maximal follows from the fact that there exists an element $r \in \chebsr{\gend}$ such that every indecomposable object residing in the columns of both $X_1$ and $X_2$ belong as subobjects of $rT$. Namely,
		\begin{equation*}
			r =
			\begin{cases}
				\sum_{j = 0}^{n-1} \croot_{2j}				& \text{if } \gend = \coxA_{2n-1}, \\
				\sum_{j = 0}^{2\left\lfloor\frac{n-1}{2}\right\rfloor} \croot_{2j}^+ + \croot_{2j}^-																	& \text{if } \gend = \coxD_{n+1}, \\
				1 + \croot_0^- + \croot_2						& \text{if } \gend = \coxE_{6}, \\
				1 + \croot_2 + \wt\croot_2 + \croot_2\wt\croot_2	& \text{if } \gend = \coxE_{7}, \\
				1 + \croot_2 + \gratio + \gratio\croot_2		& \text{if } \gend = \coxE_{8}.
			\end{cases}
		\end{equation*}
		In particular, $|rT| = |Q_0^\gend|$, where $|rT|$ denotes the number of distinct (up to isomorphism) indecomposable direct summands of $rT$. Thus, $rT$ is a rigid object that contains a basic cluster tilting object as a direct summand. But then if $X_3$ is distinct from $X_1$ and $X_2$ (which since $\Gamma$ is $G_\gend$-minimal, implies it resides in a distinct column), we must then have direct summands $Z_1,Z_2 \subseteq r(T \oplus X_3)$ such that $\Ext_{\clus_\gend}(Z_1,Z_2)\neq 0$. Hence, $T$ must be maximal and hence basic $\chebsr{\gend}$-tilting with respect to $\Gamma$, as required.
		
		(a) $\Rightarrow$ (c): This is immediate from the last part of the proof of (b) $\Rightarrow$ (a).
	\end{proof}
	
	Recall from \cite{DTOdd} that we call an object $T$ in the cluster category $\clus$ an \emph{almost complete basic $R_+$-tilting object (with respect to $\Gamma$)} if $T$ is $R_+$-rigid and there exists $X \in \Gamma$ such that $T \oplus X$ is basic $R_+$-tilting. We call such an object $X$ a \emph{complement} to $T$. The following is a trivial consequence of the above theorem.
	
	\begin{cor} \label{cor:Complements}
		Let $T \in \clus_\gend$.
		\begin{enumerate}[label=(\alph*)]
			\item If $T$ is basic $\chebsr{\gend}$-tilting with respect to $\Gamma$, then $T$ has precisely $|Q_0^{\Iiin}|=2$ direct summands.
			\item If $T$ is almost complete basic $\chebsr{\gend}$-tilting with respect to $\Gamma$, then $T$ is indecomposable and has precisely two complements.
		\end{enumerate}
	\end{cor}
	
	Also recall from \cite{DTOdd} that given a basic $R_+$-tilting object $T \in \clus$, we call the algebra $\End_{\clus}(\wh T)\op$ the \emph{cluster-$R_+$-tilted algebra}, where $\wh T$ is the cluster-tilting object given in Theorem~\ref{thm:RPTilting}(c).
	
	\begin{rem}
		For any basic $\chebsr{\gend}$-tilting objects $T_1 = X \oplus X_1, T_2 = X \oplus X_2 \in \clus_\gend$, the respective cluster-$\chebsr{\gend}$-tilted algebras $A_{T_1}$ and $A_{T_2}$ are Morita equivalent to either $KQ^\gend$ or $(KQ^\gend)\op$. In particular, $A_{T_1} \simeq KQ^\gend$ if and only if $A_{T_2} \simeq (KQ^\gend)\op$, and thus changing complement corresponds to mutation of the folded quiver $Q^{\Iiin}$. Furthermore, the associated module, bounded derived and cluster categories of the cluster-$\chebsr{\gend}$-tilted algebras have the structure of $\chebsr{\gend}$-coefficient categories in the natural way.
	\end{rem}
	
	\section{The $\cv$- and $\gv$-vectors of exchange matrices of type $\Iiin$}\label{sec:Tropical}
	The $\cv$- and $\gv$-vectors of an integer exchange matrix were introduced in \cite{FominZelevinskyIV}. In \cite{DTOdd}, we defined $\cv$- and $\gv$-vectors for quivers of type $\coxH_4$, $\coxH_3$ and $\coxI_2(2n+1)$. We will do the same now for quivers of type $\Iiin$ using similar methods.
	
	\subsection{Representations of $\hchebr{\gend}$} \label{sec:Tropical-Reps}
	The first step in defining $\cv$- and $\gv$-vectors for $\Iiin$ is to realise the exchange matrix of the unfolded quiver $Q^\gend$ as a block matrix with blocks given (in some way) by representations of an associated ring. The situation is slightly more complicated in the $\Iiin$ case, as on the one hand we have a subring $\chebr{\gend} \subset \hchebr{\gend}$ whose role is to act on the larger ring, and on the other hand, the exchange matrix of $I_2(2n)$ has a rescaling due to the simple roots being of different size.
	
	It is helpful to bypass the problems caused by the embedding of rings and different root sizes by instead considering the doubled quiver $\dbar{Q}^\gend=Q^\gend \sqcup (Q^\gend)\op$ with vertex weights given as before. That is, $Q^\gend$ has the same $R$-valuation as given in the folding $F^\gend$, and $(Q^\gend)\op$ is a copy of this quiver with the same $R$-valuation, but with the orientation of the arrows reversed.
	
	Part of the idea here is that the exchange matrix of $\dbar{Q}^\gend$ has nicer properties (which we will later demonstrate), however it is also the most canonical setting when considering unfoldings of the rescaled/skew-symmetric form of $\Iiin$ (where roots are considered to be of the same length). That is, the doubled-quiver $\dbar{Q}^\gend$ allows us to consider a folding onto the \emph{`rescaled'} quiver $\vec{Q}^{\Iiin}$, which is the same as $Q^{\Iiin}$, except we now have an $R$-valuation of $\mw([0])=\mw([1])=\lambda$, where we define
	\begin{equation*}
		\lambda = \begin{cases}
			1	& \text{if }\gend \neq \coxD_{n+1}, \\
			2 	& \text{if }\gend = \coxD_{n+1}.
		\end{cases}
	\end{equation*}
	We thus have a corresponding weighted (double-)folding $\dbar{F}^\gend\colon \dbar{Q}^\gend \rightarrow \vec{Q}^{\Iiin}$ such that $\dbar{F}^\gend$ maps source vertices in $\dbar{Q}^\gend$ to the source vertex $[0]\in Q^{\Iiin}_0$, and likewise maps sink vertices to the sink vertex $[1] \in Q^{\Iiin}_0$.
	
	We will recall the definition of the regular representation of a ring $R$ from \cite{DTOdd}.
	\begin{defn}
		Suppose $R$ has a finite $\integer$-basis $\mathcal{B}$. Define a ring homomorphism
		\begin{equation*}
			\regrep{}\colon R \rightarrow \integer^{|\mathcal{B}| \times |\mathcal{B}|}: r \mapsto \regrep{}(r),
		\end{equation*}
		where the $j$-th column of $\regrep{}(r)$ (with $j \in \mathcal{B}$) is the vector $(a_{ij})_{i \in \mathcal{B}}$ given by $j r =\sum_{i \in \mathcal{B}} a_{ij} i$. We call the map $\regrep{}$ the \emph{regular representation of $R$ with respect to $\mathcal{B}$}.
	\end{defn}
		
	For $\gend \in \{\coxA_{2n-1}, \coxD_4, \coxE_6, \coxE_7,\coxE_8\}$, we will later see that the regular representation of $\hchebr{\gend}$ is sufficient to describe the blocks of the exchange matrix of $\dbar{Q}^\gend$. This motivates the following definition.
		
	\begin{defn}
		Let $\gend \in \{\coxA_{2n-1},\coxD_4, \coxE_6, \coxE_7,\coxE_8\}$ and consider the $\integer$-bases $\mathcal{B}^\gend$ from Section~\ref{sec:Prelim-Cheb} (and in particular, from Lemmata~\ref{lem:ABasis} and \ref{lem:EBasis} and Remark~\ref{lem:D4Basis}). We call the regular representation $\regrep{\gend}$ of $\hchebr{\gend}$ with respect to $\mathcal{B}^\gend$ the \emph{representation of $\hchebr{\gend}$ corresponding to $\dbar{F}^\gend$}.
	\end{defn}
	
	The situation for type $\coxD_{n+1}$ with $n > 3$ is slightly more complicated. In this case, Lemma~\ref{lem:DBasis} indicates that the ring $\hdchebr{n+1}$ has a positive $\integer$-basis
	\begin{equation*}
		{\mathcal{B}}^{\coxD_{n+1}} = \{\croot_0^\pm,\croot_1^\pm,\ldots,\croot_{n-1}^\pm\}
	\end{equation*}
	that is useful to us, but is too big to describe the blocks of the exchange matrix of $\dbar{Q}^{\coxD_{n+1}}$. Instead, we have the following.
	
	\begin{defn}
		Let $n>3$ and consider the ideal (and proper submodule) of $\hdchebr{n+1}$
		\begin{equation*}
			\langle \croot_0^+ + \croot_0^-, \croot_1^+ + \croot_1^-, \ldots, \croot_{n-2}^+ + \croot_{n-2}^-, \croot_{n-1}^+,\croot_{n-1}^-\rangle \subset \hdchebr{n+1},
		\end{equation*}
		whose integral basis $\wh{\mathcal{B}}^{\coxD_{n+1}}$ is given by its listed set of generators. From this, define a ring homomorphism 
		\begin{equation*}
			\Dregrep\colon \hdchebr{n+1} \rightarrow \integer^{n+1 \times n+1}: r \mapsto \Dregrep(r),
		\end{equation*}
		where the $j$-th column of $\Dregrep(r)$ ($j \in \wh{\mathcal{B}}^{\coxD_{n+1}}$) is the vector $(a_{ij})_{i \in \wh{\mathcal{B}}^{\coxD_{n+1}}}$ given by $j r =\sum_{i\in \wh{\mathcal{B}}^{\coxD_{n+1}}} a_{ij} i$. We call the map $\Dregrep$ the \emph{representation of $\hdchebr{n+1}$ corresponding to $\dbar{F}^{\coxD_{n+1}}$}.
	\end{defn}
	
	The bases for the above representations have been deliberately chosen such that one has a bijective map
	\begin{align*}
		\vartheta^\gend\colon \mathcal{B}^\gend &\rightarrow Q^\gend_0  & (\gend &\in \{\coxA_{2n-1},\coxD_4, \coxE_6, \coxE_7,\coxE_8\})\\
		\wh{\mathcal{B}}^\gend &\rightarrow Q^\gend_0 & (\gend &= \coxD_{n+1}, n>3) \\ 
		\croot_i &\mapsto i & (\gend &\in \{\coxA_{2n-1},\coxE_6,\coxE_7,\coxE_8\}) \\
		g &\mapsto 2^+ & (\gend &= \coxD_4) \\
		g^2 &\mapsto 2^- & (\gend &= \coxD_4) \\
		\croot^+_i + \croot^-_i &\mapsto i & (\gend &= \coxD_{n+1}, n>3) \\
		\croot^\pm_{n-1} &\mapsto (n-1)^\pm & (\gend &= \coxD_{n+1}, n>3) \\
		\croot^\pm_i &\mapsto i^\pm & (\gend &=\coxE_6) \\
		\gratio\croot_i &\mapsto\phi_i &(\gend &= \coxE_8, i\in\{0,1\}) \\
		\gratio\croot_{v_8} &\mapsto3 &(\gend &= \coxE_8) \\
		\gratio\croot_{2} &\mapsto4 &(\gend &= \coxE_8)
	\end{align*}
	which satisfies $\vw\vartheta^\gend(e) = \hrhom{\gend}(e)$ for each $e$.
	
	\begin{rem} \label{rem:BlockReps}
		Consider the element $s \in \hchebr{\gend}$ with
		\begin{equation*}
			s= \begin{cases}
				\croot_1		& \gend \in \{\coxA_{2n-1},\coxD_4,\coxE_7,\coxE_8\}, \\
				\croot_1^+	& \gend \in \{\coxD_{n+1},\coxE_6\}.
			\end{cases}
		\end{equation*}
		Then we have the following that result from the structure of $\hchebr{\gend}$.
		\begin{enumerate}[label=(\alph*)]
			\item For each $e \in \mathcal{B}^\gend$, the product $se = \sum_{e' \in \mathcal{B}^\gend} a_{e'} e'$ is such that each $a_e \in \{0,1\}$.
			\item There exists an arrow $\vartheta^\gend(e) \rightarrow \vartheta^\gend(e')$ or an arrow $\vartheta^\gend(e') \rightarrow \vartheta^\gend(e)$ in $Q^\gend$ if and only if $e'$ is a summand of the product $se$.
			\item Let $\ll$ be a total ordering of the set $\dbar{Q}^\gend_0$ satisfying $v_0 \ll v_1$ for any vertices $v_0,v_1 \in \dbar{Q}_0^\gend$ such that $\dbar{F}^\gend(v_0)=[0]$ and $\dbar{F}^\gend(v_1)=[1]$. Then the exchange matrix of $\dbar{Q}^\gend$ with rows and columns ordered with respect to $\ll$ is given by the block matrix
			\begin{equation*}
				\dbar{B}^\gend=
				\begin{pmatrix}
					0 & \regrep{\gend}(s) \\
					-\regrep{\gend}(s) & 0
				\end{pmatrix}.
			\end{equation*}
			\item Let $<$ be a total ordering of the set $Q^\gend_0$. Then the exchange matrix of $Q^\gend$ is
			\begin{equation*}
				B^\gend=\Lambda^\gend \dbar{B}^\gend \Theta^\gend,
			\end{equation*}
			where $\Theta^\gend$ is the matrix corresponding to the canonical inclusion $Q_0^\gend \rightarrow \dbar{Q}_0^\gend$ and $\Lambda^\gend$ is the matrix corresponding to the canonical surjection $\dbar{Q}_0^\gend \rightarrow Q_0^\gend$ of $\integer$-bases/index sets with respect to the order $<$.
			\item Since $\dbar{Q}_\gend = Q^\gend \sqcup (Q^\gend)\op$, the exchange matrix $B^\gend \oplus (B^\gend)^T$ is indexed by the set $\dbar{Q}^\gend_0$, where $\oplus$ is the direct sum of linear maps and $(B^\gend)^T$ is the matrix transpose of $B^\gend$. In particular,
			\begin{equation*}
				\dbar{B}^\gend = V(B^\gend \oplus (B^\gend)^T)V\inv
			\end{equation*}
			where $V$ is the orthogonal permutation matrix that orders the $\integer$-basis/index set for the rows of $B^\gend \oplus (B^\gend)^T$ with respect to $\ll$ from (c). It therefore follows that, the operation $\Lambda^\gend ( {-}' ) \Theta^\gend$ is left inverse to $V({-} \oplus {-}^T)V\inv$, where ${-}$ is a placeholder for a $|Q_0^\gend| \times |Q_0^\gend|$ integer matrix with respect to the ordering $<$, and ${-'}$ is a placeholder for a $|\dbar{Q}_0^\gend| \times |\dbar{Q}_0^\gend|$ integer matrix with respect to the ordering $\ll$.
		\end{enumerate}
	\end{rem}
	
	\begin{rem} \label{rem:RepSigns}
		Let $e,e' \in \mathcal{B}^\gend$.
		\begin{enumerate}[label=(\alph*)]
			\item If $\gend \neq \coxE_7$, then $ee'$ is a sum of positive multiples of elements of $\mathcal{B}^\gend$. Thus the entries of the matrix $\regrep{\gend}(e)$ are all positive.
			\item If $\gend = \coxE_7$ and $e \not\in \{\wt{\croot}_1,\croot_{v_7}\}$, then $ee'$ is also sum of positive multiples of elements of $\mathcal{B}^{\coxE_7}$ and hence the entries of the matrix $\regrep{\coxE_7}(b)$ are all positive. On the other hand, we have
			\begin{equation*}
				\regrep{\coxE_7}(\wt{\croot}_1) =
				\left(\begin{smallmatrix}
					0 	& 1 	& 0 	& 0 	& 0 	& 0 	& 0 \\
					1 	& 0 	& 1 	& 0 	& 0 	& 0 	& 0 \\
					0 	& 1 	& 0 	& 1 	& 0 	& 0 	& 0 \\
					0 	& 0 	& 1 	& 1 	& 0 	& 0 	& 0 \\
					0 	& 0 	& 0 	& 0 	& 0 	& 1 	& 0 \\
					0 	& 0 	& 0 	& 0 	& 1 	& 1 	& 1 \\
					0 	& 0 	& 0 	& 0 	& 0 	& 1 	& -1 \\
				\end{smallmatrix}\right)
				\qquad \text{and} \qquad
				\regrep{\coxE_7}(\croot_{v_7}) =
				\left(\begin{smallmatrix}
					0 	& 0 	& 0 	& 0 	& 0 	& 0 	& 1 \\
					0 	& 0 	& 0 	& 0 	& 0 	& 1 	& -1 \\
					0 	& 0 	& 0 	& 0 	& 1 	& 0 	& 1 \\
					0 	& 0 	& 0 	& 0 	& 0 	& 1 	& 0 \\
					0 	& 0 	& 1 	& 0 	& 0 	& 0 	& 0 \\
					0 	& 1 	& 0 	& 1 	& 0 	& 0 	& 0 \\
					1 	& -1 	& 1 	& 0 	& 0 	& 0 	& 0 \\
				\end{smallmatrix}\right)
			\end{equation*}
			with respect to the ordering $\croot_0 < \wt{\croot}_1 < \wt{\croot}_2 < \croot_2 < \croot_1 < \croot_3 < \croot_{v_7}$.
		\end{enumerate}
	\end{rem}
	
	\subsection{$C$-matrices and $\cv$-vectors} \label{sec:Tropical-C}
	We shall recall the following general definitions for $C$-matrices and $\cv$-vectors of $R$-quivers from \cite{DTOdd}.
	\begin{defn}
		Let $J$ be an index set and let $\tree_J$ be the $|J|$-regular tree, where any pair of distinct edges incident to a common vertex in $\tree_J$ are respectively labelled by distinct indices in $J$. Choose a distinguished vertex $t_0 \in \tree_J$ and consider an exchange matrix $B$ over a ring $R$ with rows and columns indexed by $J$. We define the \emph{extended exchange matrix of $B$ over $R$} to be the $2|J| \times |J|$ matrix
		\begin{equation*}
			\wt B_{t_0} = \begin{pmatrix} B_{t_0} \\ C_{t_0} \end{pmatrix},
		\end{equation*}
		where $B_{t_0}=B$ and $C_{t_0}$ is the identity matrix. For any edge $\xymatrix@1{t \ar@{-}[r]^-k & t'}$ in $\tree_J$, one then defines matrices
		\begin{equation*}
			\wt B_{t} = \begin{pmatrix} B_{t} \\ C_{t} \end{pmatrix} 
			\qquad \text{and} \qquad
			\wt B_{t'} = \begin{pmatrix} B_{t'} \\ C_{t'} \end{pmatrix} 
		\end{equation*}
		such that $\wt B_{t'}$ is obtained from $\wt B_{t}$ by mutation at $k \in J$. We thus define the mutation of $C_t$ at $k \in J$ as $\mu_k(C_t)=C_{t'}$.
		
		With this in mind, we define a \emph{tropical $y$-seed pattern} by assigning to each vertex $t \in \tree_J$ a \emph{tropical $y$-seed} $\{B_{t};\mathbf{c}_{t,j} : j \in J\}$, where each $\mathbf{c}_{t,j}$ is the $j$-th column of the matrix $C_t$. We call the tropical $y$-seed at the vertex $t_0$ the \emph{initial tropical $y$-seed}. We call the matrices $C_t$ (where $t \in \tree_J$) \emph{$C$-matrices}, and we call the vectors $\mathbf{c}_{t,j}$ \emph{$\cv$-vectors}. The $\cv$-vectors $\mathbf{c}_{t_0,j}$ in particular are called the \emph{initial $\cv$-vectors}.
	\end{defn}
	
	The folded quivers $Q^{\Iiin}$ arise from skew-symmetrisable exchange matrices $B^{\Iiin}=(b_{[i][j]})_{[i][j] \in Q_0^{\Iiin}}$ that have been rescaled to a skew-symmetric exchange matrix
	\begin{equation*}
		\vec{B}^{\Iiin} = P B^{\Iiin} P\inv=(\vec{b}_{[i][j]})_{[i][j] \in Q_0^{\Iiin}}.
	\end{equation*}
	The relationship between these exchange matrices, their respective $C$-matrices and mutation can be summarised by the following due to \cite{Reading}.
	
	\begin{lem} \label{lem:BackLeftSquare}
		Consider the initial extended exchange matrices
		\begin{equation*}
			\begin{pmatrix}
				B_{t_0}^{\Iiin} \\
				C_{t_0}^{\Iiin}
			\end{pmatrix}
			\qquad \text{and} \qquad
			\begin{pmatrix}
				\vec{B}_{t_0}^{\Iiin} \\
				\vec{C}_{t_0}^{\Iiin}
			\end{pmatrix}
		\end{equation*}
		of $B^{\Iiin}$ and $\vec{B}^{\Iiin}$, respectively. Mutation of $B^{\Iiin}$ and its $C$-matrices commutes with rescaling. In particular, for any vertex $t \in \tree_{Q_0^{\Iiin}}$ we have
		\begin{enumerate}[label=(\alph*)]
			\item $\vec{B}^{\Iiin}_t = P B^{\Iiin}_t P\inv$,
			\item $\vec{C}^{\Iiin}_t = P C^{\Iiin}_t P\inv$.
		\end{enumerate}
	\end{lem}
	\begin{proof}
		(a) The effect of rescaling by the diagonal matrix $P =(p_{[i]})_{[i] \in Q^{\Iiin}_0}$ is such that 
		\begin{equation*}
			\vec{b}_{[i][j]} = \frac{p_{[i]}}{p_{[j]}} b_{[i][j]}.
		\end{equation*}
		Since the rescaling matrix has positive entries, it is clear from the above and the mutation formula for exchange matrices that rescaling commutes with mutation. That is,
		\begin{equation*}
			\mu_{[k]}(\vec{B}^{\Iiin}) = P \mu_{[k]}(B^{\Iiin}) P\inv
		\end{equation*}
		for each $[k] \in Q_0^{\Iiin}$, and thus $\vec{B}^{\Iiin}_t = P B^{\Iiin}_t P\inv$ as required.
		
		(b) Since initial $C$-matrices are identity matrices, the extension of the initial rescaled exchange matrix of type $Q^{\Iiin}$ is such that 
			\begin{equation*}
				\begin{pmatrix}
				\vec{B}_{t_0}^{\Iiin} \\
				\vec{C}_{t_0}^{\Iiin}
			\end{pmatrix} =
			\begin{pmatrix}
				P B_{t_0}^{\Iiin} P\inv \\
				\vec{C}_{t_0}^{\Iiin}
			\end{pmatrix} =
			\begin{pmatrix}
				P B_{t_0}^{\Iiin} P\inv \\
				P C_{t_0}^{\Iiin}P\inv
			\end{pmatrix}.
		\end{equation*}
		Also recall the mutation formula for $C$-matrices corresponding to an edge $\xymatrix@1{t \ar@{-}[r]^-{[k]} & t'}$:
		\begin{equation*}
			c_{[i][j],t'} =
			\begin{cases}
				-c_{[i][k],t}														& \text{if } [j]=[k], \\
				c_{[i][j],t} + \sgn(c_{[i][k],t})[c_{[i][k],t}b_{[k][j],t}]_+	& \text{otherwise,}
			\end{cases}
		\end{equation*}
		where for any value $a$, $\sgn(a)$ denotes the sign of $a$ and $[a]_+ = \max(0,a)$. From this, it is clear that mutation of $C$-matrices also commutes with rescaling. Thus,
		\begin{equation*}
			\vec{C}^{\Iiin}_{t'} = P C^{\Iiin}_{t'}  P\inv,
		\end{equation*}
		for each vertex $t' \in \tree_{Q_0^{\Iiin}}$.
	\end{proof}
	
	Since we have weighted foldings $F^\gend\colon Q^\gend \rightarrow Q^{\Iiin}$ and $\dbar{F}^\gend\colon \dbar{Q}^\gend \rightarrow Q^{\Iiin}$, each vertex $t \in \tree_{Q^{\Iiin}_\gend}$ is associated to vertices $\wh{t} \in \tree_{Q^\gend_0}$ and $\dbar{t} \in \tree_{\dbar{Q}^\gend_0}$ obtained via the unfolding procedure. That is, if $t$ is obtained from $t_0$ by traversing edges $k_1,\ldots,k_m$ in $\tree_{Q^{\Iiin}_0}$, then $ \wh{t}$ is obtained from the initial vertex $\wh{t}_0 \in \tree_{Q^\gend_0}$ by traversing sequences of edges $\wh{k}_1,\ldots,\wh{k}_m$ in $\tree_{Q^{\gend}_0}$, where $\wh{k}_l$ corresponds to the composite mutation $\wh{\mu}_{[k_l]} = \prod_{F^\gend(i)=[k_l]} \mu_i$. Similarly, $\dbar{t}$ is obtained from the initial vertex $\dbar{t}_0 \in \tree_{\dbar{Q}^\gend_0}$ by traversing sequences of edges $\dbar{k}_1,\ldots,\dbar{k}_m$ in $\tree_{\dbar{Q}^{\gend}_0}$, where $\dbar{k}_l$ corresponds to the composite mutation $\dbar{\mu}_{[k_l]} = \prod_{\dbar{F}^\gend(i)=[k_l]} \mu_i$. The only vertices of $\tree_{Q_0^\gend}$ and $\tree_{\dbar{Q}_0^\gend}$ that we will consider in this paper are those associated to the vertices $t \in \tree_{Q_0^{\Iiin}}$ by unfolding. Thus for the purposes of readability, we will abuse notation and write $\wh{t}$ and $\dbar{t}$ as $t$. Likewise, if the context is clear, we will write $\wh{\mu}_{[k]}$ and $\dbar{\mu}_{[k]}$ as $\mu_{[k]}$.
	
	The relationship between the $C$-matrices of the double-unfolded quiver $\dbar{Q}^\gend$ and the $C$-matrices of the single-unfolded quiver $Q^\gend$ is similar to what happens with rescaling. In particular, we have the following commutativity properties with the operations $\Lambda^\gend ({-}) \Theta^\gend$ and $V({-}\oplus{-}^T)V\inv$ defined in Remark~\ref{rem:BlockReps}(d) and (e).

	\begin{lem} \label{lem:FrontLeftSquare}
		Consider the initial extended exchange matrices
		\begin{equation*}
			\begin{pmatrix}
				B_{t_0}^{\gend} \\
				C_{t_0}^{\gend}
			\end{pmatrix}
			\qquad \text{and} \qquad
			\begin{pmatrix}
				\dbar{B}_{t_0}^{\gend} \\
				\dbar{C}_{t_0}^{\gend}
			\end{pmatrix}
		\end{equation*}
		of $B^{\gend}$ and $\dbar{B}^{\gend}$, respectively. Composite mutation (with respect to unfolding) of $\dbar{B}^{\gend}$ and its $C$-matrices commutes with the operation $\Lambda^\gend ({-}) \Theta^\gend$. Similarly, composite mutation of $B^\gend$ and its $C$-matrices commutes with the operation $V({-}\oplus{-}^T)V\inv$. In particular, for any vertex $t \in \tree_{Q_0^{\Iiin}}$ we have
		\begin{enumerate}[label=(\alph*)]
			\item $B^{\gend}_{t} = \Lambda^\gend \dbar{B}^{\gend}_{t} \Theta^\gend$ and $\dbar{B}^{\gend}_{t} = V(B^{\gend}_{t} \oplus (B^{\gend}_{t})^T)V\inv$,
			\item $C^{\gend}_{t} = \Lambda^\gend \dbar{C}^{\gend}_{t} \Theta^\gend$  and $\dbar{C}^{\gend}_{t} = V(C^{\gend}_{t} \oplus (C^{\gend}_{t})^T)V\inv$.
		\end{enumerate}
	\end{lem}
	\begin{proof}
		Similar to the previous lemma, we have 
		\begin{equation*}
			\begin{pmatrix}
				B_{t_0}^{\gend} \\
				C_{t_0}^{\gend}
			\end{pmatrix} =
			\begin{pmatrix}
				\Lambda^\gend \dbar{B}_{t_0}^{\gend} \Theta^\gend \\
				\Lambda^\gend \dbar{C}_{t_0}^{\gend} \Theta^\gend
			\end{pmatrix}.
		\end{equation*}
		Since $\dbar{Q}^\gend$ is a disjoint union of $Q^\gend$ and $(Q^\gend)\op$, the result follows from the unfolding procedure and mutation formulae. The commutativity of $V({-}\oplus{-}^T)V\inv$ with mutation follows by the same argument.
	\end{proof}
	
	In what follows, we will utilise the block matrix structure of unfolded exchange matrices and $C$-matrices. In particular, for the double-folding $\dbar{F}^\gend$, we will denote by $\dbar{B}^\gend_{[i][j],t}$ and $\dbar{C}^\gend_{[i][j],t}$ the $([i],[j])$-th block of $\dbar{B}^\gend_{t}$ and $\dbar{C}^\gend_{t}$, respectively. Here, we note that $\dbar{B}_{[i][j],t}$ corresponds to the $([i],[j])$-th entry of $B_t$ under the unfolding procedure.
	
	\begin{prop} \label{prop:CBlocks}
		Let $t \in \tree_{Q_0^{\Iiin}}$ and let $\gend \in \{\coxA_{2n-1},\coxD_{n+1},\coxE_6,\coxE_7,\coxE_8\}$.
		\begin{enumerate}[label=(\alph*)]
			\item For any $[i],[j] \in Q_0^{\Iiin}$, there exists $r_{[i][j],t} \in \hchebr{\gend}$  such that
			\begin{equation*}
				\dbar{C}^\gend_{[i][j],t} = \regrep{\gend}(r_{[i][j],t}).
			\end{equation*}
			In particular, $r_{[i][j],t} = \sum_{e \in \mathcal{B}^\gend} a_e e$, where each $a_e \in \integer$ is such that $\sgn(a_e) = \sgn(a_{e'})$ for any $e, e' \in \mathcal{B}^\gend$.
			\item $\dbar{C}_{t}^\gend$ is block-sign-coherent. That is for any $[i],[j] \in Q_0^{\Iiin}$, the entries of $\dbar{C}^\gend_{[i][j],t}$ are either all positive or all negative.
			\item Let $\xymatrix@1{t \ar@{-}[r]^-{[k]} & t'}$ be an edge of $\tree_{Q_0^{\Iiin}}$. Then the matrix $\dbar{C}^\gend_{t'} = \mu_{[k]}(\dbar{C}^\gend_{t})$ is such that
			\begin{equation*}
				\dbar{C}^\gend_{[i][j],t'} =
				\begin{cases}
					-\dbar{C}^\gend_{[i][k],t}
					& \text{if }[j]=[k], \\
					\dbar{C}^\gend_{[i][j],t}+\sgn(\dbar{C}^\gend_{[i][k],t})[\dbar{C}^\gend_{[i][k],t}\dbar{B}^\gend_{[k][j],t}]_+
					& \text{otherwise,}
				\end{cases}
			\end{equation*}
			where $\sgn(\dbar{C}^\gend_{[i][k],t})$ is the sign of the $([i],[k])$-th block of $\dbar{C}^\gend_t$, and for any matrix $A=(a_{ij})$, we define $[A]_+$ as the matrix whose entries are given by $[a_{ij}]_+ = \max(0,a_{ij})$.
			\item The blocks of $\dbar{C}^\gend_t$ commute. That is, 
			\begin{equation*}
				\dbar{C}^\gend_{[i][j],t}\dbar{C}^\gend_{[k][l],t}=\dbar{C}^\gend_{[k][l],t}\dbar{C}^\gend_{[i][j],t}
			\end{equation*}
			for any $[i],[j],[k],[l] \in Q_0^{\Iiin}$.
		\end{enumerate}
	\end{prop}
	\begin{proof}
		The proof of (a)-(d) is completed by induction on composite mutation. The initial $C$-matrix of $\dbar{Q}^\gend$ is
		\begin{equation*}
			\dbar{C}_{t_0}^\gend
			\begin{pmatrix}
				\regrep{\gend}(1) & \regrep{\gend}(0) \\
				\regrep{\gend}(0) & \regrep{\gend}(1)
			\end{pmatrix},
		\end{equation*}
		which clearly satisfies (a), (b) and (d). For the induction argument, we note that the blocks of any $C$-matrix $\dbar{C}^\gend_t$ that satisfies (a) are just representations of the commutative ring $\chebr{\gend}$. Thus, (a) clearly implies (d). Moreover, any $C$-matrix $\dbar{C}^\gend_t$ that satisfies both (a) and (b) also satisfies (c) --- the proof of this is identical to the proof of \cite[Proposition 8.7(b)]{DTOdd}. Thus, the induction argument is completed if we can show that given $\dbar{C}^\gend_t$ satisfying (a)-(d), any matrix $\dbar{C}^\gend_{t'}$ that arises from applying (c) satisfies both (a) and (b). The proof is very similar to that used in \cite[Proposition 8.7]{DTOdd}, except in this paper, we cannot take for granted that (a) implies (b) --- see Remark~\ref{rem:RepSigns}(b).
		
		So suppose $\dbar{C}^\gend_t$ satisfies (a)-(d) and let $\dbar{C}^\gend_{t'} = \mu_{[k]}(\dbar{C}^\gend_{t})$. Clearly we have
		\begin{equation*}
			\dbar{C}^\gend_{[i][k],t'} = -\regrep{\gend}(r_{[i][k],t}) = \regrep{\gend}(- r_{[i][k],t}).
		\end{equation*}
		This shows that the particular blocks $\dbar{C}^\gend_{[i][k],t'}$ are expressible as in (a). Moreover, the block-sign-coherence of $\dbar{C}^\gend_t$ then implies that each block $\dbar{C}^\gend_{[i][k],t'}$ is sign-coherent, so (b) is satisfied for these blocks too.
		
		For the blocks with $[j] \neq [k]$, we have
		\begin{align*}
			\dbar{C}^\gend_{[i][j],t'} &= \regrep{\gend}(r_{[i][j],t}) \pm [\regrep{\gend}(r_{[i][k],t})\regrep{\gend}(\pm s)]_+ \\
			&= \regrep{\gend}(r_{[i][j],t} \pm [\pm s r_{[i][k],t}]_+),
		\end{align*}
		where $s$ is as in Remark~\ref{rem:BlockReps}. To see that $r_{[i][j],t'} = r_{[i][j],t} \pm [\pm s r_{[i][k],t}]_+$ is as in (a), we note that since $\dbar{C}^\gend_{[i][j],t'}$ is the $C$-matrix of a direct product of Dynkin quivers, its columns are sign-coherent (c.f. \cite{FominZelevinskyIV,DWZ2,Nagao}). In particular, if $r_{[i][j],t'} = \sum_{e \in \mathcal{B}^\gend} a_e e$ with $\sgn(a_e) \neq \sgn(a_{e'})$ for some $e,e' \in \mathcal{B}^\gend$, then the column of $\dbar{C}^\gend_{[i][j],t'} = \regrep{\gend}(r_{[i][j],t'})$ that is indexed by $1 \in \hchebr{\gend}$ is not sign-coherent --- a contradiction. So we must have $\sgn(a_e) = \sgn(a_{e'})$ for all $e,e' \in \mathcal{B}^\gend$ instead.
		
		If $\gend \neq \coxE_7$, then Remark~\ref{rem:RepSigns} automatically implies that $\dbar{C}^\gend_{[i][j],t'}$ satisfies (b). Otherwise if $\gend = \coxE_7$, we note that (b) fails only if $a_{\wt{\croot}_1},a_{\croot_{v_7}} \neq 0$. But by Remark~\ref{rem:RepSigns}, the columns of $\regrep{\coxE_7}(\wt{\croot}_1)$ and $\regrep{\coxE_7}(\croot_{v_7})$ are not sign-coherent. In particular, the columns indexed by $\wt{\croot}_1$ and/or $\croot_{v_7}$ in the matrices $\regrep{\coxE_7}(\wt{\croot}_1)$ and $\regrep{\coxE_7}(\croot_{v_7})$ have both positive and negative entries. But since $\dbar{C}^\gend_{[i][j],t'}$ must be column-sign-coherent, this implies that there exist some $e_1,\ldots,e_m \in \mathcal{B}^{\coxE_7} \setminus \{\wt{\croot}_1,\croot_{v_7}\}$ such that
		\begin{equation*}
			\regrep{\coxE_7}\left(a_{\wt{\croot}_1}\wt{\croot}_1 + a_{\croot_{v_7}} \croot_{v_7} + \sum_{i=1}^m a_{e_i} e_i\right)
		\end{equation*}
		is column-sign-coherent (where each $a_e$ is as determined from (a) of the proposition). But since (a) is satisfied, all entries of such a matrix must have the same sign. So $\dbar{C}^\gend_{[i][j],t'}$ satisfies (b), as required.
		
		We have thus shown that any $C$-matrix produced by composite mutation from some other $C$-matrix satisfying (a)-(d) also satisfies (a)-(d). Since the initial $C$-matrix satisfies (a)-(d), the induction argument is complete.
	\end{proof}
	
	The connection between $C$-matrices of type $\gend$ and type $\Iiin$ is formalised by the following.
	
	\begin{defn} \label{defn:MProj}
		Let $R = \ZUi{2n}$ and $\Frac(R)$ be its field of fractions. Given a weighted folding $F\colon Q \rightarrow Q'$, where $Q'$ is a quiver of type $\Iiin$, let $d_{F}$ be the map from Definition~\ref{defn:ProjMap}. Suppose there exist $i_0,i_1 \in Q_0$ such that $\mw([0])=\vw(i_0)$ and $\mw([1])=\vw(i_1)$. Then define the \emph{matrix $F$-projection map} with respect to $(i_0,i_1)$ to be the map
		\begin{align*}
			\mathbf{d}_{F,i_0,i_1}\colon \integer^{|Q_0| \times |Q_0|} &\rightarrow (\Frac(R))^{2 \times 2} \\
			(\mathbf{u}_k)_{k \in Q_0} &\mapsto 
			\begin{pmatrix}
					d_{F}(\mathbf{u}_{i_0}) & d_{F}(\mathbf{u}_{i_1})
				\end{pmatrix}
		\end{align*}
		where each $\mathbf{u}_k$ is a column vector. Similarly, define the \emph{matrix transpose $F$-projection map} with respect to $(i_0,i_1)$ to be the map
		\begin{align*}
			\mathbf{d}^T_{F,i_0,i_1}\colon \integer^{|Q_0| \times |Q_0|} &\rightarrow (\Frac(R))^{2 \times 2} \\
			(\mathbf{u}_k)_{k \in Q_0} &\mapsto 
			\begin{pmatrix}
					d_{F}(\mathbf{u}_{i_0}) \\ d_{F}(\mathbf{u}_{i_1})
				\end{pmatrix}
		\end{align*}
		where each $\mathbf{u}_k$ is a row vector.
	\end{defn}
	
	We will henceforth adopt the following notation. For each vertex $i \in Q^\gend_0 \subset \dbar{Q}^\gend$, we will label the corresponding vertex of the opposite quiver by $i' \in Q^{\gend,\mathrm{op}}_0 \subset \dbar{Q}^\gend$. The matrix projection maps of primary importance are then the maps
	\begin{enumerate}[label=(\roman*)]
		\item $\mathbf{d}_{\dbar{F}^\gend,0,0'}$ and $\mathbf{d}_{F^\gend,0,1}$ if $\gend \neq \coxE_6$,
		\item $\mathbf{d}_{\dbar{F}^\gend,0^+,(0^+)'}$ and $\mathbf{d}_{F^\gend,0^+,1^+}$ if $\gend = \coxE_6$.
	\end{enumerate}
	The only reason to distinguish between the $\gend \neq \coxE_6$ and $\gend = \coxE_6$ cases here is because, technically, there are no vertices labelled by $0$, $0'$ or $1$ in $Q^{\coxE_6}$ or $\dbar{Q}^{\coxE_6}$. For readability purposes, we will thus write without any ambiguity $\mathbf{d}_{\dbar{F}^\gend}$ and $\mathbf{d}_{F^\gend}$ to represent these maps respectively.
	
	\begin{rem}
		It is not difficult to verify that since the vertices of $\vec{Q}^\Iiin$ have the same $R$-valuation, we have $\mathbf{d}_{\dbar{F}^\gend} = \mathbf{d}^T_{\dbar{F}^\gend}$. However, the same is not true for $\mathbf{d}_{{F}^\gend}$ and $\mathbf{d}^T_{{F}^\gend}$.
	\end{rem} 
	
	Our next set of results follow naturally from the previous proposition. In particular, we show an analogue of \cite[Corollaries 8.8, 8.10]{DTOdd}.
	
	\begin{cor} \label{cor:MiddleLeftSquare}
		Let $t \in \tree_{Q_0^{\Iiin}}$. Then we have the following.
		\begin{enumerate}[label=(\alph*)]
			\item $\mathbf{d}_{\dbar{F}^\gend}(\dbar{C}^\gend_t) = \vec{C}_t^{\Iiin}$.
			\item $\mathbf{d}_{F^\gend}(C^\gend_t) = C_t^{\Iiin}$
			\item The $\cv$-vectors of $\vec{B}^{\Iiin}$ and $B^{\Iiin}$ are sign-coherent.
			\item The $\cv$-vectors of $\vec{B}^{\Iiin}$ are rescaled roots of $\Iiin$ and the $\cv$-vectors of $B^{\Iiin}$ are roots of $\Iiin$.
		\end{enumerate}
	\end{cor}
	\begin{proof}
		For clarity, we will denote the $\ZUi{2n}$-valuation of $\vec{Q}^\Iiin$ and $Q^\Iiin$ by the functions $\vec{\mw}$ and $\mw$ respectively. In particular, we note that $\mw([0]) = \vec{\mw}([0])=\lambda$ and $\mw([1]) = 2\vec{\mw}([1])\cos(\frac{\pi}{2n})=2\lambda\cos(\frac{\pi}{2n})$, where $\lambda = 2$ if $\gend = \coxD_{n+1}$ with $n>3$ and $\lambda = 1$ otherwise.
	
		(a) By Proposition~\ref{prop:CBlocks}(a), we have
		\begin{equation*}
			\dbar{C}^\gend_{t} =
			\begin{pmatrix}
				\dbar{C}^\gend_{[0][0],t} & \dbar{C}^\gend_{[0][1],t} \\ 
				\dbar{C}^\gend_{[1][0],t} & \dbar{C}^\gend_{[1][1],t}
			\end{pmatrix}
			=
			\begin{pmatrix}
				\regrep{\gend}(r_{[0][0],t}) & \regrep{\gend}(r_{[0][1],t}) \\ 
				\regrep{\gend}(r_{[1][0],t}) & \regrep{\gend}(r_{[1][1],t})
			\end{pmatrix},
		\end{equation*}
		where $r_{[i][j],t} = \sum_{e \in \mathcal{B}^\gend} a_{e,[i][j]} e$ with each $a_{e,[i][j]} \in \integer$. It then follows from the definitions that we have
		\begin{align*}
			\mathbf{d}_{\dbar{F}^\gend}(\dbar{C}^\gend_{t}) &= \frac{1}{\lambda}\sum_{e \in \mathcal{B}^\gend} 
			\begin{pmatrix}
				a_{e,[0][0]}\vw\vartheta^\gend(se) & a_{e,[0][1]}\vw\vartheta^\gend(se) \\ 
				a_{e,[1][0]}\vw\vartheta^\gend(se) & a_{e,[1][1]}\vw\vartheta^\gend(se)
			\end{pmatrix} 
			\\ &= \frac{1}{\lambda}
			\begin{pmatrix}
				\hrhom{\gend}(\lambda r_{[0][0],t}) & \hrhom{\gend}(\lambda r_{[0][1],t}) \\ 
				\hrhom{\gend}(\lambda r_{[1][0],t}) & \hrhom{\gend}(\lambda r_{[1][1],t})
			\end{pmatrix}
			\\ &=
			\begin{pmatrix}
				\hrhom{\gend}( r_{[0][0],t}) & \hrhom{\gend}( r_{[0][1],t}) \\ 
				\hrhom{\gend}( r_{[1][0],t}) & \hrhom{\gend}( r_{[1][1],t})
			\end{pmatrix},
		\end{align*}
		where $s= \croot_0^+ + \croot_0^-$ if $\gend = \coxD_{n+1}$ with $n>3$ and $s=1$ otherwise. In particular, for the initial $C$-matrix, we have
		\begin{equation*}
			\mathbf{d}_{\dbar{F}^\gend}(\dbar{C}^\gend_{t_0}) =
			\begin{pmatrix}
				\hrhom{\gend}(1) & \hrhom{\gend}(0) \\ 
				\hrhom{\gend}(0) & \hrhom{\gend}(1)
			\end{pmatrix} = \vec{C}_{t_0}^{\Iiin}
		\end{equation*}
		The proof of statement $\mathbf{d}_{\dbar{F}^\gend}(\dbar{C}^\gend_t) = \vec{C}_t^{\Iiin}$ for each $t \in \tree_{Q_0^\Iiin}$ is then given by an induction argument on composite mutation, which is identical to the proof of \cite[Corollary 8.8]{DTOdd}.
		
		(b) We will first show that 
		\begin{equation} \tag{$\ast$} \label{eq:LeftCubeTopSq}
			\mathbf{d}_{F^\gend}(\Lambda^\gend \dbar{C}^\gend_t \Theta^\gend) = P\inv\mathbf{d}_{\dbar{F}^\gend}(\dbar{C}^\gend_t)P,
		\end{equation}
		where the operations $\Lambda^\gend ({-}) \Theta^\gend$ and $P\inv({-})P$ are as in Remark~\ref{rem:BlockReps}(d) and Lemma~\ref{lem:BackLeftSquare}. We begin by noting that
		\begin{equation*}
			\Lambda^\gend \dbar{C}^\gend_t \Theta^\gend =
			\begin{pmatrix}
				\Lambda^\gend_{[0]}\regrep{\gend}(r_{[0][0],t}) \Theta^\gend_{[0]} &
				\Lambda^\gend_{[0]}\regrep{\gend}(r_{[0][1],t}) \Theta^\gend_{[1]} \\
				\Lambda^\gend_{[1]}\regrep{\gend}(r_{[1][0],t}) \Theta^\gend_{[0]} &
				\Lambda^\gend_{[1]}\regrep{\gend}(r_{[1][1],t}) \Theta^\gend_{[1]}
			\end{pmatrix},
		\end{equation*}
		where $\Theta^\gend_{[i]}$ is the matrix corresponding to the canonical inclusion $\{j \in Q_0^\gend : F^\gend(j)=[i]\}\rightarrow \dbar{Q}_0^\gend$ and $\Lambda^\gend_{[i]}$ is the matrix corresponding to the canonical surjection $\dbar{Q}_0^\gend \rightarrow \{j \in Q_0^\gend : F^\gend(j)=[i]\}$ of $\integer$-bases/index sets. That is, $\Lambda^\gend_{[i]}\regrep{\gend}(r_{[i][j],t}) \Theta^\gend_{[j]}$ is obtained from $\regrep{\gend}(r_{[i][j],t})$ by removing the rows indexed by $\{k \in \dbar{Q}_0^\gend : F^\gend(k)\neq[i]\}$ and removing the columns indexed by $\{k \in \dbar{Q}_0^\gend : F^\gend(k)\neq[j]\}$. Thus,
		\begin{align*}
			\mathbf{d}_{F^\gend}(\Lambda^\gend \dbar{C}^\gend_t \Theta^\gend) &= 
			\begin{pmatrix}
				\frac{1}{\mw([0])} \hrhom{\gend}(s_{[0]}r_{[0][0],t}) & 
				\frac{1}{\mw([0])} \hrhom{\gend}(s_{[1]}r_{[0][1],t}) \\ 
				\frac{1}{\mw([1])} \hrhom{\gend}(s_{[0]}r_{[1][0],t}) & 
				\frac{1}{\mw([1])} \hrhom{\gend}(s_{[1]}r_{[1][1],t})
			\end{pmatrix}
			\\ &=
			\begin{pmatrix}
				\frac{\hrhom{\gend}(s_{[0]})}{\mw([0])} \hrhom{\gend}(r_{[0][0],t}) & 
				\frac{\hrhom{\gend}(s_{[1]})}{\mw([0])} \hrhom{\gend}(r_{[0][1],t}) \\ 
				\frac{\hrhom{\gend}(s_{[0]})}{\mw([1])} \hrhom{\gend}(r_{[1][0],t}) & 
				\frac{\hrhom{\gend}(s_{[1]})}{\mw([1])} \hrhom{\gend}(r_{[1][1],t})
			\end{pmatrix}
			\\ &=
			\begin{pmatrix}
				\mw([0]) & 0 \\
				0 & \mw([1])
			\end{pmatrix} \inv
			\begin{pmatrix}
				\hrhom{\gend}(r_{[0][0],t}) & 
				\hrhom{\gend}(r_{[0][1],t}) \\ 
				\hrhom{\gend}(r_{[1][0],t}) & 
				\hrhom{\gend}(r_{[1][1],t})
			\end{pmatrix}
			\begin{pmatrix}
				\hrhom{\gend}(s_{[0]}) & 0 \\
				0 & \hrhom{\gend}(s_{[1]}) 
			\end{pmatrix}
		\end{align*}
		where
		\begin{align*}
			s_{[0]} &=
			\begin{cases}
				\croot_0^+ + \croot^-_0 	& \text{if } \gend = \coxD_{n+1} \text{ with } n>3, \\
				1								& \text{otherwise,}
			\end{cases}
			\\
			s_{[1]} &=
			\begin{cases}
				\croot_1^+ + \croot^-_1 	& \text{if } \gend = \coxD_{n+1} \text{ with } n>3, \\
				\croot^+_1					& \text{if } \gend = \coxE_{6}, \\
				\croot_1						& \text{otherwise.}
			\end{cases}
		\end{align*}
		But then $\hrhom{\gend}(s_{[i]}) = \mw([i])$ for both $i \in \{0,1\}$, and hence the rightmost matrix in the above product is precisely the rescaling matrix $P$. Thus, the required commutativity relation (\ref{eq:LeftCubeTopSq}) holds. The result then follows from (a) alongside Lemma~\ref{lem:BackLeftSquare}(b) and Lemma~\ref{lem:FrontLeftSquare}(b) and the fact that the operations $\Lambda^\gend({-})\Theta^\gend$ and $P\inv ({-}) P$ are invertible.

		(c) This follows from (a), (b), and Proposition~\ref{prop:CBlocks}(a), as all morphisms involved respect positivity/negativity.
		
		(d) We note that the module category of $K\dbar{Q}^\gend$ is equivalent to the category $\mod* KQ^\gend \times \mod* KQ^{\gend,\mathrm{op}}$. The category $\mod* KQ^{\gend,\mathrm{op}}$ also satisfies an analogue of Theorem~\ref{thm:FoldingProjection}, where the rows of the Auslander-Reiten quiver corresponding to sink vertices now correspond to short roots and vice versa. Since the $c$-vectors of $\dbar{Q}^\gend$ are $\pm 1$-multiples of dimension vectors of indecomposable $K\dbar{Q}^\gend$-modules, (a), (b) and the definition of $\mathbf{d}_{\dbar{F}^\gend}$ imply that for a $C$-matrix
		\begin{equation*}
			\vec{C}_t^{\Iiin} = 
			\begin{pmatrix}
				\vec{c}_{[0][0]} & \vec{c}_{[0][1]} \\ 
				\vec{c}_{[1][0]} & \vec{c}_{[1][1]}
			\end{pmatrix} =
			\begin{pmatrix}
				\vec{\mathbf{c}}_{[0]} & \vec{\mathbf{c}}_{[1]} \\ 
			\end{pmatrix}
		\end{equation*}
		we have $\vec{\mathbf{c}}_{[0]} = \dimproj_{F^\gend}^{\mod*KQ^\gend}(M)$ and $\vec{\mathbf{c}}_{[1]} = \dimproj_{F^{\gend,\mathrm{op}}}^{\mod*KQ^{\gend,\mathrm{op}}}(M')$ for some $M \in \mod* KQ^\gend$ and $M' \in \mod* KQ^{\gend,\mathrm{op}}$ that reside in rows of weight 1, where
		\begin{equation*}
			F^{\gend,\mathrm{op}}\colon Q^{\gend,\mathrm{op}} \rightarrow Q^{\Iiin}
		\end{equation*}
		is the folding of the opposite quiver defined in the natural way by mapping sources vertices to sources and sinks to sinks. It is then clear from Theorem~\ref{thm:FoldingProjection} that the columns of the $C$-matrix $C_t^{\Iiin}$ given by the inverse rescaling of $\vec{C}_t^{\Iiin}$ are roots of $\Iiin$, as described in Section~\ref{sec:Prelim-ExMats}.
	\end{proof}
	
	\subsection{$G$-matrices and $\gv$-vectors}
	Since $\cv$-vectors of exchange matrices of type $\Iiin$ are sign-coherent, we are permitted (by the results of \cite{NZ} on the tropical duality of $c$- and $g$-vectors) to use the following definition of $G$-matrices and $\gv$-vectors.
	
	\begin{defn}
		Let $J$ be an index set and let $B$ be a $|J| \times |J|$ exchange matrix over $R$ whose $\cv$-vectors are sign-coherent. For any $t \in \tree_J$, we define for each $C$-matrix $C_t$ the corresponding \emph{$G$-matrix}
		\begin{equation*}
			G_t = (C_t^T)\inv.
		\end{equation*}
		We call the collection $\{G_t : t \in \tree_J\}$ the \emph{$G$-matrices} of $B$ and call each column of each $G_t$ a $\gv$-vector of $B$. We define the mutation of a $G$-matrix $G_t$ at index $k$ to be the $G$-matrix $\mu_k (G_t) = G_{t'}$, where $\xymatrix@1{t \ar@{-}[r]^-{k} & t'}$ is an edge in $\tree_J$.
	\end{defn}

        \begin{rem}
The results of~\cite{NZ} show that the mutations of $G$-matrices defined above are given by the explicit formula provided in~\cite[(6.12)--(6.13)]{FominZelevinskyIV}.
\end{rem}

	The next theorem shows that all of the definitions of this section are compatible with each other.
	
	\begin{thm}\label{thm:Tesseract}
		Let $\xymatrix@1{t \ar@{-}[r]^-{[k]} & t'} \in \tree_{Q_0^{\Iiin}}$. Then the diagram of Figure~\ref{fig:Tesseract} commutes.
	\end{thm}
	
	\begin{figure}[b]
		\centering
		\def\il{2}
\def\ol{1.7}
\begin{tikzpicture}[scale=1.5]
	\coordinate (dc1) at (0,0,0);
	\coordinate (dc2) at ($(dc1)+\il*(0,-1,0)$);
	\coordinate (dg1) at ($(dc1)+\il*(1,0,0)$);
	\coordinate (dg2) at ($(dc2)+\il*(1,0,0)$);
	
	\coordinate (skew) at (-0.4,-0.1,-0.1);
	
	\coordinate (vc1) at ($(dc1)+\il*(0,0,-1)+(skew)$);
	\coordinate (vc2) at ($(dc2)+\il*(0,0,-1)+(skew)$);
	\coordinate (vg1) at ($(dg1)+\il*(0,0,-1)+(skew)$);
	\coordinate (vg2) at ($(dg2)+\il*(0,0,-1)+(skew)$);
	
	\coordinate(offset) at (0.17,0.1,0);
	
	\coordinate (c1) at ($(dc1)+\ol*(-1,1,1)+(offset)$);
	\coordinate (c2) at ($(dc2)+\ol*(-1,-1,1)+(offset)$);
	\coordinate (g1) at ($(dg1)+\ol*(1,1,1)+(offset)$);
	\coordinate (g2) at ($(dg2)+\ol*(1,-1,1)+(offset)$);
	
	\coordinate (ic1) at ($(vc1)+\ol*(-1,1,-1)+(skew)+(offset)$);
	\coordinate (ic2) at ($(vc2)+\ol*(-1,-1,-1)+(skew)+(offset)$);
	\coordinate (ig1) at ($(vg1)+\ol*(1,1,-1)+(skew)+(offset)$);
	\coordinate (ig2) at ($(vg2)+\ol*(1,-1,-1)+(skew)+(offset)$);

	\draw (c1) node {$C^\gend_{t}$};
	\draw (g1) node {$G^\gend_{t}$};
	\draw (c2) node {$C^\gend_{t'}$};
	\draw (g2) node {$G^\gend_{t'}$};
	
	\draw (dc1) node {$\bar{\bar{C}}^\gend_{t}$};
	\draw (dg1) node {$\bar{\bar{G}}^\gend_{t}$};
	\draw (dc2) node {$\bar{\bar{C}}^\gend_{t'}$};
	\draw (dg2) node {$\bar{\bar{G}}^\gend_{t'}$};
	
	\draw (vc1) node {$\vec{C}^{\Iiin}_{t}$};
	\draw (vg1) node {$\vec{G}^{\Iiin}_{t}$};
	\draw (vc2) node {$\vec{C}^{\Iiin}_{t'}$};
	\draw (vg2) node {$\vec{G}^{\Iiin}_{t'}$};
	
	\draw (ic1) node {$ C^{\Iiin}_{t}$};
	\draw (ig1) node {$ G^{\Iiin}_{t}$};
	\draw (ic2) node {$ C^{\Iiin}_{t'}$};
	\draw (ig2) node {$ G^{\Iiin}_{t'}$};
	
	\draw[|->, shorten <= 4.4ex, shorten >= 4.4ex,blue] (ic1) -- (ig1);
	\draw[|->, shorten <= 2.3ex, shorten >= 2.3ex, red] (ic1) -- (ic2);
	\draw[|->, shorten <= 2.3ex, shorten >= 2.3ex, red] (ig1) -- (ig2);
	\draw[|->, shorten <= 3.8ex, shorten >= 4.4ex,blue] (ic2) -- (ig2);
	
	\draw [white,fill=white] ($(ic1)+0.38*(ic2)-0.38*(ic1)+(-0.1,0,0)$) rectangle ($(ic1)+0.35*(ic2)-0.35*(ic1)+(0.1,0,0)$);
	\draw [white,fill=white] ($(ic1)+0.505*(ic2)-0.505*(ic1)+(-0.1,0,0)$) rectangle ($(ic1)+0.48*(ic2)-0.48*(ic1)+(0.1,0,0)$);
	\draw [white,fill=white] ($(ic1)+0.585*(ic2)-0.585*(ic1)+(-0.1,0,0)$) rectangle ($(ic1)+0.53*(ic2)-0.53*(ic1)+(0.1,0,0)$);
	\draw [white,fill=white] ($(ic1)+0.84*(ic2)-0.84*(ic1)+(-0.1,0,0)$) rectangle ($(ic1)+0.89*(ic2)-0.89*(ic1)+(0.1,0,0)$);
	%\draw [white,fill=white] ($(ic1)+0.89*(ic2)-0.89*(ic1)+(-0.1,0,0)$) rectangle ($(ic1)+0.84*(ic2)-0.84*(ic1)+(0.1,0,0)$);
	\draw [white,fill=white] ($(ic2)+0.09*(ig2)-0.09*(ic2)+(0,-0.1,0)$) rectangle ($(ic2)+0.12*(ig2)-0.12*(ic2)+(0,0.1,0)$);
	\draw [white,fill=white] ($(ic2)+0.54*(ig2)-0.54*(ic2)+(0,-0.1,0)$) rectangle ($(ic2)+0.57*(ig2)-0.57*(ic2)+(0,0.1,0)$);
	\draw [white,fill=white] ($(ic2)+0.735*(ig2)-0.735*(ic2)+(0,-0.1,0)$) rectangle ($(ic2)+0.765*(ig2)-0.765*(ic2)+(0,0.1,0)$);
	
	\draw[|->, shorten <= 3ex, shorten >= 3ex, cyan!80!black!100] (vc1) -- (ic1);
	\draw[|->, shorten <= 3ex, shorten >= 3ex, cyan!80!black!100] (vg1) -- (ig1);
	\draw[|->, shorten <= 3ex, shorten >= 4ex, cyan!80!black!100] (vc2) -- (ic2);
	\draw[|->, shorten <= 3ex, shorten >= 4ex, cyan!80!black!100] (vg2) -- (ig2);
	
	\draw [white,fill=white] ($(vc1)+0.22*(ic1)-0.22*(vc1)$) rectangle ($(vc1)+0.15*(ic1)-0.15*(vc1)$);
	\draw [white,fill=white] ($(vg1)+0.22*(ig1)-0.22*(vg1)$) rectangle ($(vg1)+0.15*(ig1)-0.15*(vg1)$);
	\draw [white,fill=white] ($(vc2)+0.4*(ic2)-0.4*(vc2)$) rectangle ($(vc2)+0.25*(ic2)-0.25*(vc2)$);
	\draw [white,fill=white] ($(vg2)+0.42*(ig2)-0.42*(vg2)$) rectangle ($(vg2)+0.34*(ig2)-0.34*(vg2)$);
	
	\draw[|->, shorten <= 4.4ex, shorten >= 4.4ex,blue] (vc1) -- (vg1);
	\draw[|->, shorten <= 2ex, shorten >= 2ex, red] (vc1) -- (vc2);
	\draw[|->, shorten <= 1.9ex, shorten >= 2ex, red] (vg1) -- (vg2);
	\draw[|->, shorten <= 4.4ex, shorten >= 4.4ex,blue] (vc2) -- (vg2);
	
	\draw [white,fill=white] ($(vc1)+0.3*(vc2)-0.3*(vc1)+(-0.1,0,0)$) rectangle ($(vc1)+0.4*(vc2)-0.4*(vc1)+(0.1,0,0)$);
	\draw [white,fill=white] ($(vg1)+0.13*(vg2)-0.13*(vg1)+(-0.1,0,0)$) rectangle ($(vg1)+0.19*(vg2)-0.19*(vg1)+(0.1,0,0)$);
	
	\draw[|->, shorten <= 2ex, shorten >= 3ex, green!40!black!100!] (dc1) -- (vc1);
	\draw[|->, shorten <= 2ex, shorten >= 3ex, green!40!black!100!] (dg1) -- (vg1);
	\draw[|->, shorten <= 2ex, shorten >= 3ex, green!40!black!100!] (dc2) -- (vc2);
	\draw[|->, shorten <= 2ex, shorten >= 3ex, green!40!black!100!] (dg2) -- (vg2);
	
	\draw[|->, shorten <= 2ex, shorten >= 2ex,blue] (dc1) -- (dg1);
	\draw[|->, shorten <= 2ex, shorten >= 2ex, red] (dc1) -- (dc2);
	\draw[|->, shorten <= 2ex, shorten >= 2ex, red] (dg1) -- (dg2);
	\draw[|->, shorten <= 2ex, shorten >= 2ex,blue] (dc2) -- (dg2);
	
	\draw[<-|, shorten <= 3ex, shorten >= 2ex, cyan!80!black!100] (c1) -- (dc1);
	\draw[<-|, shorten <= 2ex, shorten >= 2.5ex, cyan!80!black!100] (g1) -- (dg1);
	\draw[<-|, shorten <= 2.5ex, shorten >= 2ex, cyan!80!black!100] (c2) -- (dc2);
	\draw[<-|, shorten <= 2ex, shorten >= 2ex, cyan!80!black!100] (g2) -- (dg2);
	
	\draw[|->, shorten <= 3ex, shorten >= 2ex,blue] (c1) -- (g1);
	\draw[|->, shorten <= 2ex, shorten >= 2ex, red] (c1) -- (c2);
	\draw[|->, shorten <= 2ex, shorten >= 2ex, red] (g1) -- (g2);
	\draw[|->, shorten <= 2ex, shorten >= 2ex,blue] (c2) -- (g2);
	
	\draw[|->, shorten <= 2.5ex, shorten >= 3ex, green!40!black!100!] (c1) -- (ic1);
	\draw[|->, shorten <= 2.5ex, shorten >= 4ex, green!40!black!100!] (g1) -- (ig1);
	\draw[|->, shorten <= 4ex, shorten >= 3ex, green!40!black!100!] (c2) -- (ic2);
	\draw[|->, shorten <= 2.5ex, shorten >= 3ex, green!40!black!100!] (g2) -- (ig2);
	
	\draw [white,fill=white] ($(g1)+0.12*(g2)-0.12*(g1) +(-.1,0)$) rectangle ($(g1)+0.17*(g2)-0.17*(g1)+(.1,0)$);
	
	\draw[blue] ($(c1)+0.5*(g1)-0.5*(c1) +(0.7,0.2,0)$) node {\footnotesize$({-}^T)^{-1}$};
	\draw[blue] ($(dc1)+0.5*(dg1)-0.5*(dc1) +(0,0.2,0)$) node {\footnotesize$({-}^T)^{-1}$};
	\draw[blue] ($(vc1)+0.5*(vg1)-0.5*(vc1) +(0,0.2,0)$) node {\footnotesize$({-}^T)^{-1}$};
	\draw[blue] ($(ic1)+0.5*(ig1)-0.5*(ic1) +(0,0.2,0)$) node {\footnotesize$({-}^T)^{-1}$};
	\draw[blue] ($(c2)+0.5*(g2)-0.5*(c2) +(0,0.2,0)$) node {\footnotesize$({-}^T)^{-1}$};
	\draw[blue] ($(dc2)+0.5*(dg2)-0.5*(dc2) +(0.25,0.2,0)$) node {\footnotesize$({-}^T)^{-1}$};
	\draw[blue] ($(vc2)+0.5*(vg2)-0.5*(vc2) +(0,0.2,0)$) node {\footnotesize$({-}^T)^{-1}$};
	\draw[blue] ($(ic2)+0.5*(ig2)-0.5*(ic2) +(-0.7,-0.2,0)$) node {\footnotesize$({-}^T)^{-1}$};
	
	\draw[red] ($(c1)+0.5*(c2)-0.5*(c1) +(-0.2,0,0)$) node {\footnotesize$\mu_{[k]}$};
	\draw[red] ($(dc1)+0.5*(dc2)-0.5*(dc1) +(-0.2,0,0)$) node {\footnotesize$\mu_{[k]}$};
	\draw[red] ($(vc1)+0.5*(vc2)-0.5*(vc1) +(-0.2,-0.2,0)$) node {\footnotesize$\mu_{[k]}$};
	\draw[red] ($(ic1)+0.5*(ic2)-0.5*(ic1) +(-0.2,-1,0)$) node {\footnotesize$\mu_{[k]}$};
	\draw[red] ($(g1)+0.5*(g2)-0.5*(g1) +(0.25,1,0)$) node {\footnotesize$\mu_{[k]}$};
	\draw[red] ($(dg1)+0.5*(dg2)-0.5*(dg1) +(0.2,0.2,0)$) node {\footnotesize$\mu_{[k]}$};
	\draw[red] ($(vg1)+0.5*(vg2)-0.5*(vg1) +(0.25,0,0)$) node {\footnotesize$\mu_{[k]}$};
	\draw[red] ($(ig1)+0.5*(ig2)-0.5*(ig1) +(0.25,0,0)$) node {\footnotesize$\mu_{[k]}$};
	
	\draw [cyan!80!black!100]($(c1)+0.5*(dc1)-0.5*(c1) +(0.2,-0.5,0)$) node {\footnotesize$\Lambda^\Delta ({-}) \Theta^\Delta$};
	\draw [cyan!80!black!100]($(g1)+0.5*(dg1)-0.5*(g1) +(0.4,-0.2,0)$) node {\footnotesize$\Lambda^\Delta ({-}) \Theta^\Delta$};
	\draw [cyan!80!black!100]($(c2)+0.5*(dc2)-0.5*(c2) +(0.5,-0.2,0)$) node {\footnotesize$\Lambda^\Delta ({-}) \Theta^\Delta$};
	\draw [cyan!80!black!100]($(g2)+0.5*(dg2)-0.5*(g2) +(-0.5,-0.2,0)$) node {\footnotesize$\Lambda^\Delta ({-}) \Theta^\Delta$};
	
	\draw [green!40!black!100!]($(dc1)+0.5*(vc1)-0.5*(dc1) +(-0.3,0.1,0)$) node {\footnotesize$\mathbf{d}_{\bar{\bar{F}}^\gend}$};
	\draw [green!40!black!100!]($(dg1)+0.5*(vg1)-0.5*(dg1) +(-0.35,0.1,0)$) node {\footnotesize$\mathbf{d}_{\bar{\bar{F}}^\gend}$};
	\draw [green!40!black!100!]($(dc2)+0.5*(vc2)-0.5*(dc2) +(0.35,-0.1,0)$) node {\footnotesize$\mathbf{d}_{\bar{\bar{F}}^\gend}$};
	\draw [green!40!black!100!]($(dg2)+0.5*(vg2)-0.5*(dg2) +(0.35,-0.1,0)$) node {\footnotesize$\mathbf{d}_{\bar{\bar{F}}^\gend}$};
	
	\draw [green!40!black!100!]($(c1)+0.5*(ic1)-0.5*(c1) +(-0.3,0.1,0)$) node {\footnotesize$\mathbf{d}_{F^\gend}$};
	\draw [green!40!black!100!]($(g1)+0.5*(ig1)-0.5*(g1) +(0.2,-0.3,0)$) node {\footnotesize$\mathbf{d}^T_{F^\gend}$};
	\draw [green!40!black!100!]($(c2)+0.5*(ic2)-0.5*(c2) +(-0.15,0.3,0)$) node {\footnotesize$\mathbf{d}_{F^\gend}$};
	\draw [green!40!black!100!]($(g2)+0.5*(ig2)-0.5*(g2) +(0.3,-0.1,0)$) node {\footnotesize$\mathbf{d}^T_{F^\gend}$};
	
	\draw [cyan!80!black!100]($(ic1)+0.5*(vc1)-0.5*(ic1) +(0.5,0.1,0)$) node {\footnotesize$P^{-1}({-})P$};
	\draw [cyan!80!black!100]($(ig1)+0.5*(vg1)-0.5*(ig1) +(-0.5,0.1,0)$) node {\footnotesize$P({-})P^{-1}$};
	\draw [cyan!80!black!100]($(ic2)+0.5*(vc2)-0.5*(ic2) +(-0.35,0.2,0)$) node {\footnotesize$P^{-1} ({-})P$};
	\draw [cyan!80!black!100]($(ig2)+0.5*(vg2)-0.5*(ig2) +(0.25,0.2,0)$) node {\footnotesize$P ({-}) P^{-1}$};
\end{tikzpicture}
		\caption{The relationship between the different notions of $C$- and $G$-matrices for foldings onto $\Iiin$ can be described by the above commutative diagram, which forms a tesseract. Arrows have been given different colours for clarity, where two arrows (and their corresponding operations) pointing in the same `dimension' have the same colour. Notation is derived from Remark~\ref{rem:BlockReps}(e), Section~\ref{sec:Tropical-C} and Definition~\ref{defn:MProj}.} \label{fig:Tesseract}
	\end{figure}
	
	\begin{proof}
		We begin by considering the front and back cubes of Figure~\ref{fig:Tesseract}:
		\begin{center}
			\begin{tikzpicture}[scale=1.5]
	\coordinate (flu) at (0,0,0);
	\coordinate (fld) at (0,-2,0);
	\coordinate (fru) at (2,0,0);
	\coordinate (frd) at (2,-2,0);
	\coordinate (blu) at (0,0,-2);
	\coordinate (bld) at (0,-2,-2);
	\coordinate (bru) at (2,0,-2);
	\coordinate (brd) at (2,-2,-2);
	
	\coordinate (o) at (-2.5,0,0);

	\draw ($(o)+(blu)$) node {$\bar{\bar{C}}^\gend_{t}$};
	\draw ($(o)+(bru)$) node {$\bar{\bar{G}}^\gend_{t}$};
	\draw ($(o)+(bld)$) node {$\bar{\bar{C}}^\gend_{t'}$};
	\draw ($(o)+(brd)$) node {$\bar{\bar{G}}^\gend_{t'}$};
	
	\draw ($(o)+(flu)$) node {${C}^{\gend}_{t}$};
	\draw ($(o)+(fru)$) node {${G}^{\gend}_{t}$};
	\draw ($(o)+(fld)$) node {${C}^{\gend}_{t'}$};
	\draw ($(o)+(frd)$) node {${G}^{\gend}_{t'}$};
	
	\draw[blue][|->, shorten <= 3ex, shorten >= 3ex] ($(o)+(blu)$) -- ($(o)+(bru)$);
	\draw[red][|->, shorten <= 2.3ex, shorten >= 2.3ex] ($(o)+(blu)$) -- ($(o)+(bld)$);
	\draw[red][|->, shorten <= 2.3ex, shorten >= 2.3ex] ($(o)+(bru)$) -- ($(o)+(brd)$);
	\draw[blue][|->, shorten <= 3ex, shorten >= 3ex] ($(o)+(bld)$) -- ($(o)+(brd)$);
	
	\draw [white,fill=white] ($(o)+0.6*(blu)+0.4*(bld)+(-0.1,0.1)$) rectangle ($(o)+0.6*(blu)+0.4*(bld)+(0.1,-0.05)$);
	\draw [white,fill=white]($(o)+0.4*(bld)+0.6*(brd)+(-0.05,0.1)$) rectangle ($(o)+0.4*(bld)+0.6*(brd)+(0.1,-0.05)$);
	
	\draw[cyan!80!black!100][<-|, shorten <= 3ex, shorten >= 3ex] ($(o)+(flu)$) -- ($(o)+(blu)$);
	\draw[cyan!80!black!100][<-|, shorten <= 3ex, shorten >= 3ex] ($(o)+(fld)$) -- ($(o)+(bld)$);
	\draw[cyan!80!black!100][<-|, shorten <= 3ex, shorten >= 3ex] ($(o)+(fru)$) -- ($(o)+(bru)$);
	\draw[cyan!80!black!100][<-|, shorten <= 3ex, shorten >= 3ex] ($(o)+(frd)$) -- ($(o)+(brd)$);
	
	\draw[blue][|->, shorten <= 3ex, shorten >= 3ex] ($(o)+(flu)$) -- ($(o)+(fru)$);
	\draw[red][|->, shorten <= 2.3ex, shorten >= 2.3ex] ($(o)+(flu)$) -- ($(o)+(fld)$);
	\draw[red][|->, shorten <= 2.3ex, shorten >= 2.3ex] ($(o)+(fru)$) -- ($(o)+(frd)$);
	\draw[blue][|->, shorten <= 3ex, shorten >= 3ex] ($(o)+(fld)$) -- ($(o)+(frd)$);

	\draw[blue] ($(o)+0.5*(flu)+0.5*(fru)+(0.2,0.2)$) node {\footnotesize$({-}^T)^{-1}$};
	\draw[blue] ($(o)+0.5*(blu)+0.5*(bru)+(0,0.2)$) node {\footnotesize$({-}^T)^{-1}$};
	\draw[blue] ($(o)+0.5*(fld)+0.5*(frd)+(0,-0.2)$) node {\footnotesize$({-}^T)^{-1}$};
	\draw[blue] ($(o)+0.5*(bld)+0.5*(brd)+(-0.2,0.2)$) node {\footnotesize$({-}^T)^{-1}$};
	
	\draw[red] ($(o)+0.5*(flu)+0.5*(fld)+(-0.2,0)$) node {\footnotesize$\mu_{[k]}$};
	\draw[red] ($(o)+0.5*(blu)+0.5*(bld)+(-0.2,-0.2)$) node {\footnotesize$\mu_{[k]}$};
	\draw[red] ($(o)+0.5*(fru)+0.5*(frd)+(0.25,0.3)$) node {\footnotesize$\mu_{[k]}$};
	\draw[red] ($(o)+0.5*(bru)+0.5*(brd)+(0.25,0)$) node {\footnotesize$\mu_{[k]}$};
	
	\draw[cyan!80!black!100] ($(o)+0.5*(flu)+0.5*(blu)+(-0.5,0.1)$) node {\footnotesize$\Lambda^\gend({-})\Theta^\gend$};
	\draw[cyan!80!black!100] ($(o)+0.5*(fld)+0.5*(bld)+(0.5,-0.1)$) node {\footnotesize$\Lambda^\gend({-})\Theta^\gend$};
	\draw[cyan!80!black!100] ($(o)+0.5*(fru)+0.5*(bru)+(-0.5,0.1)$) node {\footnotesize$\Lambda^\gend({-})\Theta^\gend$};
	\draw[cyan!80!black!100] ($(o)+0.5*(frd)+0.5*(brd)+(0.5,-0.1)$) node {\footnotesize$\Lambda^\gend({-})\Theta^\gend$};
	
	\coordinate (o) at (2.5,0,0);

	\draw ($(o)+(blu)$) node {$C^{\Iiin}_{t}$};
	\draw ($(o)+(bru)$) node {$G^{\Iiin}_{t}$};
	\draw ($(o)+(bld)$) node {$C^{\Iiin}_{t'}$};
	\draw ($(o)+(brd)$) node {$G^{\Iiin}_{t'}$};
	
	\draw ($(o)+(flu)$) node {$\vec{C}^{\Iiin}_{t}$};
	\draw ($(o)+(fru)$) node {$\vec{G}^{\Iiin}_{t}$};
	\draw ($(o)+(fld)$) node {$\vec{C}^{\Iiin}_{t'}$};
	\draw ($(o)+(frd)$) node {$\vec{G}^{\Iiin}_{t'}$};
	
	\draw[blue][|->, shorten <= 4.4ex, shorten >= 4.4ex] ($(o)+(blu)$) -- ($(o)+(bru)$);
	\draw[red][|->, shorten <= 2.3ex, shorten >= 2.3ex] ($(o)+(blu)$) -- ($(o)+(bld)$);
	\draw[red][|->, shorten <= 2.3ex, shorten >= 2.3ex] ($(o)+(bru)$) -- ($(o)+(brd)$);
	\draw[blue][|->, shorten <= 4.4ex, shorten >= 4.4ex] ($(o)+(bld)$) -- ($(o)+(brd)$);
	
	\draw [white,fill=white] ($(o)+0.6*(blu)+0.4*(bld)+(-0.1,0.1)$) rectangle ($(o)+0.6*(blu)+0.4*(bld)+(0.1,-0.05)$);
	\draw [white,fill=white]($(o)+0.4*(bld)+0.6*(brd)+(-0.05,0.1)$) rectangle ($(o)+0.4*(bld)+0.6*(brd)+(0.1,-0.05)$);
	
	\draw[cyan!80!black!100][|->, shorten <= 3.5ex, shorten >= 3.5ex] ($(o)+(flu)$) -- ($(o)+(blu)$);
	\draw[cyan!80!black!100][|->, shorten <= 3.5ex, shorten >= 3.5ex] ($(o)+(fld)$) -- ($(o)+(bld)$);
	\draw[cyan!80!black!100][|->, shorten <= 3.5ex, shorten >= 3.5ex] ($(o)+(fru)$) -- ($(o)+(bru)$);
	\draw[cyan!80!black!100][|->, shorten <= 3.5ex, shorten >= 3.5ex] ($(o)+(frd)$) -- ($(o)+(brd)$);
	
	\draw[blue][|->, shorten <= 4.4ex, shorten >= 4.4ex] ($(o)+(flu)$) -- ($(o)+(fru)$);
	\draw[red][|->, shorten <= 2.3ex, shorten >= 2.3ex] ($(o)+(flu)$) -- ($(o)+(fld)$);
	\draw[red][|->, shorten <= 2.3ex, shorten >= 2.3ex] ($(o)+(fru)$) -- ($(o)+(frd)$);
	\draw[blue][|->, shorten <= 4.4ex, shorten >= 4.4ex] ($(o)+(fld)$) -- ($(o)+(frd)$);

	\draw[blue] ($(o)+0.5*(flu)+0.5*(fru)+(0.2,0.2)$) node {\footnotesize$({-}^T)^{-1}$};
	\draw[blue] ($(o)+0.5*(blu)+0.5*(bru)+(0,0.2)$) node {\footnotesize$({-}^T)^{-1}$};
	\draw[blue] ($(o)+0.5*(fld)+0.5*(frd)+(0,-0.2)$) node {\footnotesize$({-}^T)^{-1}$};
	\draw[blue] ($(o)+0.5*(bld)+0.5*(brd)+(-0.2,0.2)$) node {\footnotesize$({-}^T)^{-1}$};
	
	\draw[red] ($(o)+0.5*(flu)+0.5*(fld)+(-0.2,0)$) node {\footnotesize$\mu_{[k]}$};
	\draw[red] ($(o)+0.5*(blu)+0.5*(bld)+(-0.2,-0.2)$) node {\footnotesize$\mu_{[k]}$};
	\draw[red] ($(o)+0.5*(fru)+0.5*(frd)+(0.25,0.3)$) node {\footnotesize$\mu_{[k]}$};
	\draw[red] ($(o)+0.5*(bru)+0.5*(brd)+(0.25,0)$) node {\footnotesize$\mu_{[k]}$};
	
	\draw[cyan!80!black!100] ($(o)+0.5*(flu)+0.5*(blu)+(-0.5,0.1)$) node {\footnotesize$P^{-1}({-})P$};
	\draw[cyan!80!black!100] ($(o)+0.5*(fld)+0.5*(bld)+(0.5,-0.1)$) node {\footnotesize$P^{-1}({-})P$};
	\draw[cyan!80!black!100] ($(o)+0.5*(fru)+0.5*(bru)+(-0.5,0.1)$) node {\footnotesize$P({-})P^{-1}$};
	\draw[cyan!80!black!100] ($(o)+0.5*(frd)+0.5*(brd)+(0.5,-0.1)$) node {\footnotesize$P({-})P^{-1}$};
\end{tikzpicture}
		\end{center}
		Firstly, the left faces commute as a consequence of Lemma~\ref{lem:FrontLeftSquare} and Lemma~\ref{lem:BackLeftSquare}. Secondly, the commutativity of the top/bottom squares of both cubes are a trivial exercise in linear algebra --- for the front cube it is perhaps easier to see with the inverse map $V({-}\oplus{-}^T)V\inv$. Thirdly, the front and back faces of both cubes follow by definition, and this is compatible with the classical notion of $G$-matrix mutation due to the fact that all $c$-vectors in this paper are sign-coherent (c.f. \cite{NZ,FominZelevinskyIV}). Finally, the fact that $({-}^T)\inv$ is (self-)invertible and all other squares of both cubes commute imply that the rightmost squares of both cubes commute. Thus the front and back cubes commute, as required.
		
		We now turn our attention to the central cube of Figure~\ref{fig:Tesseract}. Namely, the following.
		\begin{center}
			\begin{tikzpicture}[scale=1.5]
	\coordinate (flu) at (0,0,0);
	\coordinate (fld) at (0,-2,0);
	\coordinate (fru) at (2,0,0);
	\coordinate (frd) at (2,-2,0);
	\coordinate (blu) at (0,0,-2);
	\coordinate (bld) at (0,-2,-2);
	\coordinate (bru) at (2,0,-2);
	\coordinate (brd) at (2,-2,-2);
	
	\coordinate (o) at (0,0,0);

	\draw ($(o)+(flu)$) node {$\bar{\bar{C}}^\gend_{t}$};
	\draw ($(o)+(fru)$) node {$\bar{\bar{G}}^\gend_{t}$};
	\draw ($(o)+(fld)$) node {$\bar{\bar{C}}^\gend_{t'}$};
	\draw ($(o)+(frd)$) node {$\bar{\bar{G}}^\gend_{t'}$};
	
	\draw ($(o)+(blu)$) node {$\vec{C}^{\Iiin}_{t}$};
	\draw ($(o)+(bru)$) node {$\vec{G}^{\Iiin}_{t}$};
	\draw ($(o)+(bld)$) node {$\vec{C}^{\Iiin}_{t'}$};
	\draw ($(o)+(brd)$) node {$\vec{G}^{\Iiin}_{t'}$};
	
	\draw[blue][|->, shorten <= 4.4ex, shorten >= 4.4ex] ($(o)+(blu)$) -- ($(o)+(bru)$);
	\draw[red][|->, shorten <= 2.3ex, shorten >= 2.3ex] ($(o)+(blu)$) -- ($(o)+(bld)$);
	\draw[red][|->, shorten <= 2.3ex, shorten >= 2.3ex] ($(o)+(bru)$) -- ($(o)+(brd)$);
	\draw[blue][|->, shorten <= 4.4ex, shorten >= 4.4ex] ($(o)+(bld)$) -- ($(o)+(brd)$);
	
	\draw [white,fill=white] ($(o)+0.6*(blu)+0.4*(bld)+(-0.1,0.1)$) rectangle ($(o)+0.6*(blu)+0.4*(bld)+(0.1,-0.05)$);
	\draw [white,fill=white]($(o)+0.4*(bld)+0.6*(brd)+(-0.05,0.1)$) rectangle ($(o)+0.4*(bld)+0.6*(brd)+(0.1,-0.05)$);
	
	\draw[green!40!black!100!][|->, shorten <= 3ex, shorten >= 3.5ex] ($(o)+(flu)$) -- ($(o)+(blu)$);
	\draw[green!40!black!100!][|->, shorten <= 3ex, shorten >= 3.5ex] ($(o)+(fld)$) -- ($(o)+(bld)$);
	\draw[green!40!black!100!][|->, shorten <= 3ex, shorten >= 3.5ex] ($(o)+(fru)$) -- ($(o)+(bru)$);
	\draw[green!40!black!100!][|->, shorten <= 3ex, shorten >= 3.5ex] ($(o)+(frd)$) -- ($(o)+(brd)$);
	
	\draw[blue][|->, shorten <= 3ex, shorten >= 3ex] ($(o)+(flu)$) -- ($(o)+(fru)$);
	\draw[red][|->, shorten <= 2.3ex, shorten >= 2.3ex] ($(o)+(flu)$) -- ($(o)+(fld)$);
	\draw[red][|->, shorten <= 2.3ex, shorten >= 2.3ex] ($(o)+(fru)$) -- ($(o)+(frd)$);
	\draw[blue][|->, shorten <= 3ex, shorten >= 3ex] ($(o)+(fld)$) -- ($(o)+(frd)$);

	\draw[blue] ($(o)+0.5*(flu)+0.5*(fru)+(0.2,0.2)$) node {\footnotesize$({-}^T)^{-1}$};
	\draw[blue] ($(o)+0.5*(blu)+0.5*(bru)+(0,0.2)$) node {\footnotesize$({-}^T)^{-1}$};
	\draw[blue] ($(o)+0.5*(fld)+0.5*(frd)+(0,0.2)$) node {\footnotesize$({-}^T)^{-1}$};
	\draw[blue] ($(o)+0.5*(bld)+0.5*(brd)+(-0.2,0.2)$) node {\footnotesize$({-}^T)^{-1}$};
	
	\draw[red] ($(o)+0.5*(flu)+0.5*(fld)+(-0.2,0)$) node {\footnotesize$\mu_{[k]}$};
	\draw[red] ($(o)+0.5*(blu)+0.5*(bld)+(-0.2,-0.2)$) node {\footnotesize$\mu_{[k]}$};
	\draw[red] ($(o)+0.5*(fru)+0.5*(frd)+(0.25,0.3)$) node {\footnotesize$\mu_{[k]}$};
	\draw[red] ($(o)+0.5*(bru)+0.5*(brd)+(0.25,0)$) node {\footnotesize$\mu_{[k]}$};
	
	\draw[green!40!black!100!] ($(o)+0.5*(flu)+0.5*(blu)+(-0.2,0.1)$) node {\footnotesize$\mathbf{d}_{\bar{\bar{F}}}$};
	\draw[green!40!black!100!] ($(o)+0.5*(fld)+0.5*(bld)+(-0.2,0.1)$) node {\footnotesize$\mathbf{d}_{\bar{\bar{F}}}$};
	\draw[green!40!black!100!] ($(o)+0.5*(fru)+0.5*(bru)+(-0.2,0.1)$) node {\footnotesize$\mathbf{d}_{\bar{\bar{F}}}$};
	\draw[green!40!black!100!] ($(o)+0.5*(frd)+0.5*(brd)+(-0.2,0.1)$) node {\footnotesize$\mathbf{d}_{\bar{\bar{F}}}$};
\end{tikzpicture}
		\end{center}
		The left square commutes by Corollary~\ref{cor:MiddleLeftSquare}. The front/back squares again follow by definition. If in addition to this, the top and bottom squares commute, then the right square commutes due to the fact that $({-}^T)\inv$ is self-inverse. Thus this is the next step of the proof.
		
		Note that for an arbitrary exchange matrix over a ring $R$ (and indexed over a set $J$) whose $C$-matrices are sign-coherent, it follows that for each $t \in \tree_{J}$, the determinant of $C_t$ is $|C_t|=\pm 1$. This follows from the fact that $C_{t_0}$ is the identity matrix, and that $C$-matrix mutation involves changing the sign of a column and/or (when sign-coherence is satisfied) adding a multiple of one column to another --- an operation that preserves determinant up to sign. Now write
		\begin{equation*}
			\vec{C}_t^{\coxI_2(2n)} =
			\begin{pmatrix}
				\vec{c}_{[0][0]} & \vec{c}_{[0][1]} \\
				\vec{c}_{[1][0]} & \vec{c}_{[1][1]}
			\end{pmatrix}
			\qquad \text{and} \qquad
			\dbar{C}_t^{\gend} =
			\begin{pmatrix}
				\dbar{C}_{[0][0]} & \dbar{C}_{[0][1]} \\
				\dbar{C}_{[1][0]} & \dbar{C}_{[1][1]}
			\end{pmatrix} =
			\begin{pmatrix}
				\rho^\gend(r_{[0][0]}) & \rho^\gend(r_{[0][1]}) \\
				\rho^\gend(r_{[1][0]}) & \rho^\gend(r_{[1][1]})
			\end{pmatrix}.
		\end{equation*}
		Since the blocks of $\dbar{C}_t^{\gend}$ commute and the determinant of all $C$-matrices is $\pm1$, a simple computation shows that we have
		\begin{equation*}
			\vec{G}_t^{\coxI_2(2n)} = \pm
			\begin{pmatrix}
				\vec{c}_{[1][1]} & -\vec{c}_{[1][0]} \\
				-\vec{c}_{[0][1]} & \vec{c}_{[0][0]}
			\end{pmatrix}
			\qquad \text{and} \qquad
			\dbar{G}_t^{\gend} = \pm
			\begin{pmatrix}
				\rho^\gend(r_{[1][1]}) & \rho^\gend(-r_{[1][0]}) \\
				\rho^\gend(-r_{[0][1]}) & \rho^\gend(r_{[0][0]}).
			\end{pmatrix}
		\end{equation*}
		From this, we can see that $\mathbf{d}_{\dbar{F}^{\gend}}(\dbar{G}_t^{\gend}) = \vec{G}_t^{\gend}$, as required. Thus, the top and bottom squares commute, and hence, the entire central cube commutes.
		
		Next comes the left and right cubes of Figure~\ref{fig:Tesseract}:
		\begin{center}
			\begin{tikzpicture}[scale=1.5]
	\coordinate (flu) at (0,0,0);
	\coordinate (fld) at (0,-2,0);
	\coordinate (fru) at (2,0,0);
	\coordinate (frd) at (2,-2,0);
	\coordinate (blu) at (0,0,-2);
	\coordinate (bld) at (0,-2,-2);
	\coordinate (bru) at (2,0,-2);
	\coordinate (brd) at (2,-2,-2);
	
	\coordinate (o) at (-2.5,0,0);

	\draw ($(o)+(blu)$) node {$C^{\Iiin}_{t}$};
	\draw ($(o)+(bru)$) node {$\vec{C}^{\Iiin}_{t}$};
	\draw ($(o)+(bld)$) node {$C^{\Iiin}_{t'}$};
	\draw ($(o)+(brd)$) node {$\vec{C}^{\Iiin}_{t'}$};
	
	\draw ($(o)+(flu)$) node {${C}^{\gend}_{t}$};
	\draw ($(o)+(fru)$) node {$\bar{\bar{C}}^\gend_{t}$};
	\draw ($(o)+(fld)$) node {${C}^{\gend}_{t'}$};
	\draw ($(o)+(frd)$) node {$\bar{\bar{C}}^\gend_{t'}$};
	
	\draw[cyan!80!black!100][<-|, shorten <= 4.4ex, shorten >= 4.4ex] ($(o)+(blu)$) -- ($(o)+(bru)$);
	\draw[red][|->, shorten <= 2.3ex, shorten >= 2.3ex] ($(o)+(blu)$) -- ($(o)+(bld)$);
	\draw[red][|->, shorten <= 2.3ex, shorten >= 2.3ex] ($(o)+(bru)$) -- ($(o)+(brd)$);
	\draw[cyan!80!black!100][<-|, shorten <= 4.4ex, shorten >= 4.4ex] ($(o)+(bld)$) -- ($(o)+(brd)$);
	
	\draw [white,fill=white] ($(o)+0.6*(blu)+0.4*(bld)+(-0.1,0.1)$) rectangle ($(o)+0.6*(blu)+0.4*(bld)+(0.1,-0.05)$);
	\draw [white,fill=white]($(o)+0.4*(bld)+0.6*(brd)+(-0.05,0.1)$) rectangle ($(o)+0.4*(bld)+0.6*(brd)+(0.1,-0.05)$);
	
	\draw[green!40!black!100!][|->, shorten <= 3ex, shorten >= 3ex] ($(o)+(flu)$) -- ($(o)+(blu)$);
	\draw[green!40!black!100!][|->, shorten <= 3ex, shorten >= 3ex] ($(o)+(fld)$) -- ($(o)+(bld)$);
	\draw[green!40!black!100!][|->, shorten <= 3ex, shorten >= 3ex] ($(o)+(fru)$) -- ($(o)+(bru)$);
	\draw[green!40!black!100!][|->, shorten <= 3ex, shorten >= 3ex] ($(o)+(frd)$) -- ($(o)+(brd)$);
	
	\draw[cyan!80!black!100][<-|, shorten <= 3ex, shorten >= 3ex] ($(o)+(flu)$) -- ($(o)+(fru)$);
	\draw[red][|->, shorten <= 2.3ex, shorten >= 2.3ex] ($(o)+(flu)$) -- ($(o)+(fld)$);
	\draw[red][|->, shorten <= 2.3ex, shorten >= 2.3ex] ($(o)+(fru)$) -- ($(o)+(frd)$);
	\draw[cyan!80!black!100][<-|, shorten <= 3ex, shorten >= 3ex] ($(o)+(fld)$) -- ($(o)+(frd)$);

	\draw[cyan!80!black!100] ($(o)+0.5*(flu)+0.5*(fru)+(0.3,0.2)$) node {\footnotesize$\Lambda^\gend({-})\Theta^\gend$};
	\draw[cyan!80!black!100] ($(o)+0.5*(blu)+0.5*(bru)+(0,0.2)$) node {\footnotesize$P^{-1}({-})P$};
	\draw[cyan!80!black!100] ($(o)+0.5*(fld)+0.5*(frd)+(0,-0.2)$) node {\footnotesize$\Lambda^\gend({-})\Theta^\gend$};
	\draw[cyan!80!black!100] ($(o)+0.5*(bld)+0.5*(brd)+(-0.2,0.2)$) node {\tiny$P^{-1}({-})P$};
	
	\draw[red] ($(o)+0.5*(flu)+0.5*(fld)+(-0.2,0)$) node {\footnotesize$\mu_{[k]}$};
	\draw[red] ($(o)+0.5*(blu)+0.5*(bld)+(-0.2,-0.2)$) node {\footnotesize$\mu_{[k]}$};
	\draw[red] ($(o)+0.5*(fru)+0.5*(frd)+(0.25,0.3)$) node {\footnotesize$\mu_{[k]}$};
	\draw[red] ($(o)+0.5*(bru)+0.5*(brd)+(0.25,0)$) node {\footnotesize$\mu_{[k]}$};
	
	\draw[green!40!black!100!] ($(o)+0.5*(flu)+0.5*(blu)+(-0.3,0.1)$) node {\footnotesize$\mathbf{d}_{F^\gend}$};
	\draw[green!40!black!100!] ($(o)+0.5*(fld)+0.5*(bld)+(0.3,-0.1)$) node {\footnotesize$\mathbf{d}_{F^\gend}$};
	\draw[green!40!black!100!] ($(o)+0.5*(fru)+0.5*(bru)+(-0.3,0.1)$) node {\footnotesize$\mathbf{d}_{\bar{\bar{F}}^\gend}$};
	\draw[green!40!black!100!] ($(o)+0.5*(frd)+0.5*(brd)+(0.3,-0.1)$) node {\footnotesize$\mathbf{d}_{\bar{\bar{F}}^\gend}$};
	
	\coordinate (o) at (2.5,0,0);

	\draw ($(o)+(bru)$) node {$G^{\Iiin}_{t}$};
	\draw ($(o)+(blu)$) node {$\vec{G}^{\Iiin}_{t}$};
	\draw ($(o)+(brd)$) node {$G^{\Iiin}_{t'}$};
	\draw ($(o)+(bld)$) node {$\vec{G}^{\Iiin}_{t'}$};
	
	\draw ($(o)+(fru)$) node {${G}^{\gend}_{t}$};
	\draw ($(o)+(flu)$) node {$\bar{\bar{G}}^\gend_{t}$};
	\draw ($(o)+(frd)$) node {${G}^{\gend}_{t'}$};
	\draw ($(o)+(fld)$) node {$\bar{\bar{G}}^\gend_{t'}$};
	
	\draw[cyan!80!black!100][|->, shorten <= 4.4ex, shorten >= 4.4ex] ($(o)+(blu)$) -- ($(o)+(bru)$);
	\draw[red][|->, shorten <= 2.3ex, shorten >= 2.3ex] ($(o)+(blu)$) -- ($(o)+(bld)$);
	\draw[red][|->, shorten <= 2.3ex, shorten >= 2.3ex] ($(o)+(bru)$) -- ($(o)+(brd)$);
	\draw[cyan!80!black!100][|->, shorten <= 4.4ex, shorten >= 4.4ex] ($(o)+(bld)$) -- ($(o)+(brd)$);
	
	\draw [white,fill=white] ($(o)+0.6*(blu)+0.4*(bld)+(-0.1,0.1)$) rectangle ($(o)+0.6*(blu)+0.4*(bld)+(0.1,-0.05)$);
	\draw [white,fill=white]($(o)+0.4*(bld)+0.6*(brd)+(-0.05,0.1)$) rectangle ($(o)+0.4*(bld)+0.6*(brd)+(0.1,-0.05)$);
	
	\draw[green!40!black!100!][|->, shorten <= 3.5ex, shorten >= 3.5ex] ($(o)+(flu)$) -- ($(o)+(blu)$);
	\draw[green!40!black!100!][|->, shorten <= 3.5ex, shorten >= 3.5ex] ($(o)+(fld)$) -- ($(o)+(bld)$);
	\draw[green!40!black!100!][|->, shorten <= 3.5ex, shorten >= 3.5ex] ($(o)+(fru)$) -- ($(o)+(bru)$);
	\draw[green!40!black!100!][|->, shorten <= 3.5ex, shorten >= 3.5ex] ($(o)+(frd)$) -- ($(o)+(brd)$);
	
	\draw[cyan!80!black!100][|->, shorten <= 3ex, shorten >= 3ex] ($(o)+(flu)$) -- ($(o)+(fru)$);
	\draw[red][|->, shorten <= 2.3ex, shorten >= 2.3ex] ($(o)+(flu)$) -- ($(o)+(fld)$);
	\draw[red][|->, shorten <= 2.3ex, shorten >= 2.3ex] ($(o)+(fru)$) -- ($(o)+(frd)$);
	\draw[cyan!80!black!100][|->, shorten <= 3ex, shorten >= 3ex] ($(o)+(fld)$) -- ($(o)+(frd)$);

	\draw[cyan!80!black!100] ($(o)+0.5*(flu)+0.5*(fru)+(0.3,0.2)$) node {\footnotesize$\Lambda^\gend({-})\Theta^\gend$};
	\draw[cyan!80!black!100] ($(o)+0.5*(blu)+0.5*(bru)+(0,0.2)$) node {\footnotesize$P({-})P^{-1}$};
	\draw[cyan!80!black!100] ($(o)+0.5*(fld)+0.5*(frd)+(0,-0.2)$) node {\footnotesize$\Lambda^\gend({-})\Theta^\gend$};
	\draw[cyan!80!black!100] ($(o)+0.5*(bld)+0.5*(brd)+(-0.2,0.2)$) node {\tiny$P({-})P^{-1}$};
	
	\draw[red] ($(o)+0.5*(flu)+0.5*(fld)+(-0.2,0)$) node {\footnotesize$\mu_{[k]}$};
	\draw[red] ($(o)+0.5*(blu)+0.5*(bld)+(-0.2,-0.2)$) node {\footnotesize$\mu_{[k]}$};
	\draw[red] ($(o)+0.5*(fru)+0.5*(frd)+(0.25,0.3)$) node {\footnotesize$\mu_{[k]}$};
	\draw[red] ($(o)+0.5*(bru)+0.5*(brd)+(0.25,0)$) node {\footnotesize$\mu_{[k]}$};
	
	\draw[green!40!black!100!] ($(o)+0.5*(flu)+0.5*(blu)+(-0.3,0.1)$) node {\footnotesize$\mathbf{d}_{\bar{\bar{F}}^\gend}$};
	\draw[green!40!black!100!] ($(o)+0.5*(fld)+0.5*(bld)+(0.3,-0.1)$) node {\footnotesize$\mathbf{d}_{\bar{\bar{F}}^\gend}$};
	\draw[green!40!black!100!] ($(o)+0.5*(fru)+0.5*(bru)+(-0.3,0.1)$) node {\footnotesize$\mathbf{d}^T_{F^\gend}$};
	\draw[green!40!black!100!] ($(o)+0.5*(frd)+0.5*(brd)+(0.3,-0.1)$) node {\footnotesize$\mathbf{d}^T_{F^\gend}$};
\end{tikzpicture}
		\end{center}
		The commutativity of the left cube is a consequence of Corollary~\ref{cor:MiddleLeftSquare} and the relation (\ref{eq:LeftCubeTopSq}) in its proof. For the right cube, we have already proven that the left, front and back faces commute. The proof of the commutativity of the top and bottom faces of the right cube follows a similar argument to the proof of Corollary~\ref{cor:MiddleLeftSquare}(b). Namely, we note that
		\begin{align*}
			\mathbf{d}^T_{F^\gend}(\Lambda^\gend \dbar{G}^\gend_t \Theta^\gend) &= \pm
			\begin{pmatrix}
				\frac{1}{\mw([0])} \hrhom{\gend}(s_{[0]}r_{[1][1]}) & 
				-\frac{1}{\mw([1])} \hrhom{\gend}(s_{[0]}r_{[1][0]}) \\ 
				-\frac{1}{\mw([0])} \hrhom{\gend}(s_{[1]}r_{[0][1]}) & 
				\frac{1}{\mw([1])} \hrhom{\gend}(s_{[1]}r_{[0][0]})
			\end{pmatrix}
			\\ &= \pm
			\begin{pmatrix}
				\frac{\hrhom{\gend}(s_{[0]})}{\mw([0])} \hrhom{\gend}(r_{[1][1]}) & 
				\frac{-\hrhom{\gend}(s_{[0]})}{\mw([1])} \hrhom{\gend}(r_{[1][0]}) \\ 
				\frac{-\hrhom{\gend}(s_{[1]})}{\mw([0])} \hrhom{\gend}(r_{[0][1]}) & 
				\frac{\hrhom{\gend}(s_{[1]})}{\mw([1])} \hrhom{\gend}(r_{[0][0]})
			\end{pmatrix}
			\\ &= \pm
			\begin{pmatrix}
				\hrhom{\gend}(s_{[0]}) & 0 \\
				0 & \hrhom{\gend}(s_{[1]}) 
			\end{pmatrix}
			\begin{pmatrix}
				\vec{c}_{[1][1]} & -\vec{c}_{[1][0]} \\
				-\vec{c}_{[0][1]} & \vec{c}_{[0][0]}
			\end{pmatrix}
			\begin{pmatrix}
				\mw([0]) & 0 \\
				0 & \mw([1])
			\end{pmatrix}\inv
			\\ &= 
			\begin{pmatrix}
				\hrhom{\gend}(s_{[0]}) & 0 \\
				0 & \hrhom{\gend}(s_{[1]}) 
			\end{pmatrix}
			\mathbf{d}_{\dbar{F}^{\gend}}(\dbar{G}_t^{\gend})
			P\inv
		\end{align*}
		where each $s_{[i]}$ is as in the proof of Corollary~\ref{cor:MiddleLeftSquare}(b). But then $\hrhom{\gend}(s_{[i]}) = \mw([i])$ for both $i \in \{0,1\}$, and hence the leftmost matrix in the above product is precisely the rescaling matrix $P$. Thus, the top and bottom squares commute. Since $\Lambda^\gend({-})\Theta^\gend$ is invertible, the right face of the right cube is also commutative, as required. Thus, we have just shown that both the left and right cubes commute.
		
		The last thing to check are the top and bottom cubes. But for both of these remaining cubes, we have proved the commutativity of all but the topmost and bottommost faces respectively. Since all maps have an inverse, these squares are also commutative. Hence the entire tesseract of Figure~\ref{fig:Tesseract} is commutative.
	\end{proof}
        
	Using Theorem~\ref{thm:Tesseract} and the results of \cite{DehyKeller,Plamondon}, one can obtain $g$-vectors of $\Iiin$ directly from the semiring action on the cluster category $\clus_\gend$ of $Q^\gend$. Let $\mathcal{P} \subset \clus_\gend$ be the full subcategory whose objects are the classes that correspond to the complexes of projective objects of $\mod* KQ^\gend$ concentrated in degree 0. Then for any object $X \in \clus_\gend$, we know from \cite{KellerReiten} that there exists a triangle
	\begin{equation*}
		P' \rightarrow P \rightarrow X \rightarrow \sus P'
	\end{equation*}
	with $P,P' \in \mathcal{P}$. In our setting, we are particularly interested in objects of $\clus_\gend$ that reside in rows of the Auslander-Reiten quiver with weight $1$. That is, rows that contain an object corresponding to an indecomposable projective $P(i)$ with $\vw(i)=\mw(F^\gend(i))$. Given a choice of indecomposable objects $P_{0}\cong P(i_0)$ and $P_{1} \cong P(i_1)$ in $\mathcal{P}$ with $F^\gend(i_j)=[j]$ and $\vw(i_j) = \mw([j])$, and given an indecomposable object $X \in \clus_\gend$ in a row of the Auslander-Reiten quiver with weight $1$, we have a triangle
	\begin{equation*}
		r'_{0} P_{0} \oplus r'_{1} P_{1} \rightarrow r_{0} P_{0} \oplus r_{1} P_{1} \rightarrow X \rightarrow \sus (r'_{0} P_{0} \oplus r'_{1} P_{1})
	\end{equation*}
	for some $r_{0},r'_{0},r_{1},r'_{1} \in \chebsr{\gend}$.
	
	\begin{defn} \label{defn:CatGVect}
		Let $X \in \clus_\gend$ be an indecomposable object in a row of the Auslander-Reiten quiver with weight $1$. We call the vector
		\begin{equation*}
			g^X = (\rhom{\gend}(r_{1} - r'_{1}), \rhom{\gend}(r_{0} - r'_{0}))
		\end{equation*}
		the \emph{folded $g$-vector} of $X$ with respect to $P_{0}, P_{1} \in \mathcal{P}$, where $P_{0}$, $P_{1}$, $r_{0}$, $r'_{0}$, $r_{1}$ and $r'_{1}$ are as above.
	\end{defn}
	
	An immediate consequence of Theorem~\ref{thm:Tesseract} and known results on the categorification of $g$-vectors (c.f. \cite{DehyKeller,Plamondon}) is the following.
	
	\begin{cor} \label{cor:GVectors}
		Let $\mathbf{X}$ be the collection of all indecomposable objects of $\clus_\gend$ that reside in rows of the Auslander-Reiten quiver with weight $1$.
		\begin{enumerate}[label=(\alph*)]
			\item The folded $g$-vectors of the objects in $\mathbf{X}$ are precisely the $g$-vectors of $Q^{\Iiin}$.
			\item Basic $\chebsr{\gend}$-tilting objects (with respect to a basic, $G$-minimal set of $\chebsr{\gend}$-generators $\Gamma$) correspond to $G$-matrices of $Q^{\Iiin}$. In particular, let $T= Y_0 \oplus Y_1$ be basic $\chebsr{\gend}$-tilting with respect to $\Gamma$ such that $Y_i$ resides in a row that corresponds to $[i] \in Q_0^{\Iiin}$. Then we have the following.
			\begin{enumerate}[label=(\roman*)]
				\item If $X_0, X_1 \in\mathbf{X}$ are respectively $\chebsr{\gend}$-generated by  $Y_0$ and $Y_1$, then
				\begin{equation*}
					G^{X_1,X_0}=(g^{X_1},g^{X_0})
				\end{equation*}
				is a $G$-matrix of $Q^{\Iiin}$.
				\item Let $T\ps{1} = Y_0 \oplus Y'_1$ be given by changing complement and let $X'_1 \in \mathbf{X}$ be an object $\chebsr{\gend}$-generated by $Y'_1$. Then the $G$-matrix
				\begin{equation*}
					G^{X'_1,X_0}=(g^{X'_1},g^{X_0})
				\end{equation*}
				is obtained by mutating $G^{X_1,X_0}$ with respect to the index $[0] \in Q^\Iiin$.
				\item Let $T\ps{0} = Y'_0 \oplus Y_1$ be given by changing complement and let $X'_0 \in \mathbf{X}$ be an object $\chebsr{\gend}$-generated by $Y'_0$. Then the $G$-matrix
				\begin{equation*}
					G^{X_1,X'_0}=(g^{X_1},g^{X'_0})
				\end{equation*}
				is obtained by mutating $G^{X_1,X_0}$ with respect to the index $[1] \in Q^\Iiin$.
			\end{enumerate}
		\end{enumerate}
	\end{cor}
	
	\appendix
	\section{Worked examples} \label{sec:Examples}
	Many of the constructions in this paper are combinatorial and seemingly technical in nature. However, it is not too difficult to construct examples of how the theory works in practice. We will show two examples: One example highlights the theory for foldings of type $\coxA$ and the other will highlight the theory for foldings of type $\coxD$. The three exceptional foldings of type $\coxE$ work along a similar principle to the other two examples.
	
	\subsection{The folding $F^{\coxA_7}$} \label{sec:AExample}
	Consider the example given by $F^{\coxA_7}\colon Q^{\coxA_7} \rightarrow Q^{\coxI_2(8)}$. We already have an example of the projection maps $\dimproj_{F^{\coxA_7}}^{\acat}$ with $\acat = \mod*KQ^{\coxA_7}$ or $\acat = \bder(\mod*KQ^{\coxA_7})$ with Figure~\ref{fig:A7-I8}. From the figure, it is straightforward to see the implications of Theorem~\ref{thm:FoldingProjection} and Corollaries~\ref{cor:RowBijection} and \ref{cor:RowWeights}. In particular, we have $\dimproj_{F^{\coxA_7}}^{\acat}(0^\pm) = (1,0)$ and $\dimproj_{F^{\coxA_7}}^{\acat}(1^\pm)=(0,1)$, where the vectors are given with respect to the standard root system of $\Iiin$. We will thus focus on the impact of Sections~\ref{sec:Action}-\ref{sec:Tropical} on this example.
	
	The category $\acat$ has an action of $\achebsr{7} = \nnint[\croot_2,\croot_4,\croot_6]$. Here, the element $\croot_6$ represents the $\integer_2$-symmetry on the quiver, and thus, $\croot_4 = \croot_2\croot_6$. The isomorphism condition (\ref{eq:AIso}) in Section~\ref{sec:Action-isoclasses} shows that the Auslander-Reiten sequence/triangle
	\begin{equation*}
		\begin{smallmatrix} 2^+ \\ 1^+  \end{smallmatrix}
		\rightarrow 
		\begin{smallmatrix} 0^+ \ 2^+ \\ 1^+ \end{smallmatrix} 
		\rightarrow 0^+ 
		\rightarrow
	\end{equation*}
	maps to sequences/triangles that are isomorphic to the following Auslander-Reiten sequences/triangles under the following actions.
	\begin{align*}
		&\croot_2\colon &
		&\begin{smallmatrix} 0^+ \  2^+ \  2^- \\ 1^+ \  3 \ \end{smallmatrix}
		\rightarrow 
		\begin{smallmatrix} 0^+ \ 2^+ \\ 1^+ \end{smallmatrix} \oplus \begin{smallmatrix} 2^+ \ 2^- \\ 3 \ \end{smallmatrix}
		\rightarrow 2^+ 
		\rightarrow
		\\
		&\croot_4\colon &
		&\begin{smallmatrix} 0^- \  2^- \  2^+ \\ 1^- \  3 \ \end{smallmatrix}
		\rightarrow 
		\begin{smallmatrix} 0^- \ 2^- \\ 1^- \end{smallmatrix} \oplus \begin{smallmatrix} 2^- \ 2^+ \\ 3 \ \end{smallmatrix}
		\rightarrow 2^-
		\rightarrow
		\\
		&\croot_6\colon &
		&\begin{smallmatrix} 2^- \\ 1^-  \end{smallmatrix}
		\rightarrow 
		\begin{smallmatrix} 0^- \ 2^- \\ 1^- \end{smallmatrix} 
		\rightarrow 0^- 
		\rightarrow
	\end{align*}
	The middle term in the first sequence reflects the relation $\croot_2\croot_1 = \croot_1 + \croot_3 \in \hachebsr{7}$ and the middle term in the second sequence reflects the relation $\croot_4\croot_1 = \croot_3 + \croot_5 \in \hachebsr{7}$. One also notes from this that the original Auslander-Reiten sequence ends in an object whose $F^{\coxA_7}$-projected dimension vector is of length $1$, and that under the action of $\croot_i$, this maps to an Auslander-Reiten sequence ending in an object whose $F^{\coxA_7}$-projected dimension vector is of length $\hrhom{\coxA_7}(\croot_i)$. Similarly, the Auslander-Reiten sequence/triangle
	\begin{equation*}
		\begin{smallmatrix} \ \ \; 2^+ \ 2^- \\ 1^+ \ 3 \ \ \; \end{smallmatrix}
		\rightarrow 
		\begin{smallmatrix} 2^+ \\ 1^+  \end{smallmatrix} \oplus \begin{smallmatrix} 0^+ \  2^+ \  2^- \\ 1^+ \  3 \ \end{smallmatrix}
		\rightarrow \begin{smallmatrix} 0^+ \ 2^+ \\ 1^+ \end{smallmatrix} 
		\rightarrow
	\end{equation*}
	maps to sequences/triangles that are isomorphic to the following direct sums of Auslander-Reiten sequences/triangles under the following actions.
	\begin{align*} \footnotesize
		&\croot_2\colon  
		\begin{smallmatrix} \ \ \; 2^+ \ 2^- \\ 1^+ \ 3 \ \ \; \end{smallmatrix} \oplus \begin{smallmatrix} 0^+ \ 2^+ \ 2^- \ 0^- \\ 1^+ \ 3 \ \ 1^- \end{smallmatrix}
		\rightarrow 
		\left(\begin{smallmatrix} 2^+ \\ 1^+  \end{smallmatrix} \oplus \begin{smallmatrix} 0^+ \  2^+ \  2^- \\ 1^+ \  3 \ \end{smallmatrix}\right) \oplus \left(\begin{smallmatrix} 0^+ \  2^+ \  2^- \\ 1^+ \  3 \ \end{smallmatrix} \oplus \begin{smallmatrix} 0^- \  2^- \  2^+ \\ 1^- \  3 \ \end{smallmatrix}\right)
		\rightarrow \begin{smallmatrix} 0^+ \ 2^+ \\ 1^+ \end{smallmatrix} \oplus \begin{smallmatrix} 2^+ \ 2^- \\ 3 \ \end{smallmatrix}
		\\
		&\croot_4\colon 
		\begin{smallmatrix} \ \ \; 2^- \ 2^+ \\ 1^- \ 3 \ \ \; \end{smallmatrix} \oplus \begin{smallmatrix} 0^- \ 2^- \ 2^+ \ 0^+ \\ 1^- \ 3 \ \ 1^+ \end{smallmatrix}
		\rightarrow 
		\left(\begin{smallmatrix} 2^- \\ 1^-  \end{smallmatrix} \oplus \begin{smallmatrix} 0^- \  2^- \  2^+ \\ 1^- \  3 \ \end{smallmatrix}\right) \oplus \left(\begin{smallmatrix} 0^- \  2^- \  2^+ \\ 1^- \  3 \ \end{smallmatrix} \oplus \begin{smallmatrix} 0^+ \  2^+ \  2^- \\ 1^+ \  3 \ \end{smallmatrix}\right)
		\rightarrow \begin{smallmatrix} 0^- \ 2^- \\ 1^- \end{smallmatrix} \oplus \begin{smallmatrix} 2^- \ 2^+ \\ 3 \ \end{smallmatrix}
		\\
		&\croot_6\colon 
		\begin{smallmatrix} \ \ \; 2^- \ 2^+ \\ 1^- \ 3 \ \ \; \end{smallmatrix}
		\rightarrow 
		\begin{smallmatrix} 2^- \\ 1^-  \end{smallmatrix} \oplus \begin{smallmatrix} 0^- \  2^+ \  2^+ \\ 1^- \  3 \ \end{smallmatrix}
		\rightarrow \begin{smallmatrix} 0^- \ 2^- \\ 1^- \end{smallmatrix} 
		\rightarrow
	\end{align*}
	This highlights multiple relations at play simultaneously. For example, we have the relations $\croot_2(1 + \croot_2) = 1+ 2\croot_2 + \croot_4$ and $\croot_4(1 + \croot_2) = \croot_2 + 2\croot_4 + \croot_6$ reflecting the action on the objects of the middle terms of the first two sequences. Moreover, we have the relations $\croot_2\croot_1 = \croot_1 + \croot_3$ and $\croot_4\croot_1 = \croot_3 + \croot_5$ reflecting the action on the starting/ending terms of the first two sequences.
	
	The action on the Auslander-Reiten translate of the above sequences is given in the natural way by the fact that the action of $r \in \achebsr{7}$ commutes with $\tau$. A key upshot of this is that every Auslander-Reiten sequence in $\acat$ can be labelled by an element $\croot_i \in \hachebsr{7}$ determined by the starting or ending term (which in turn, can be determined by the row in the Auslander-Reiten quiver that the starting/ending object resides in). If $\croot_j\croot_i = \sum_{k\in V_{ji}} \croot_k$ for some index set  $V_{ji}$, then $\croot_j$ maps an Auslander-Reiten sequence labelled by $\croot_i$ to a sequence that is isomorphic to a direct sum of Auslander-Reiten sequences labelled by the summands $\croot_k$ (whose ending terms all reside in the same column).
	
	To showcase the theory of $\achebsr{7}$-generators on categories, we will focus on $\mod*KQ^{\coxA_7}$. Consider the sets
	\begin{equation*}
		\Gamma_{0^\pm} = \{
		0^\pm,
		\begin{smallmatrix} 2^\pm \\ 1^\pm  \end{smallmatrix},
		\begin{smallmatrix} 2^\mp \\ 3  \end{smallmatrix},
		\begin{smallmatrix} 0^\mp \\  1^\mp \end{smallmatrix}
		\}
		\qquad
		\text{and}
		\qquad
		\Gamma_{1^\pm} = \{
		\begin{smallmatrix} 0^\pm \ 2^\pm \\ 1^\pm \end{smallmatrix},
		\begin{smallmatrix} \ \ \; 2^\pm \ 2^\mp \\ 1^\pm \ 3 \ \ \;  \end{smallmatrix},
		\begin{smallmatrix} \ \ 2^\mp \ 0^\mp \\ 3 \ 1^\mp \ \ \end{smallmatrix},
		1^\mp
		\}
	\end{equation*}
	By Theorem~\ref{thm:RPGenerators}(a), any set $\Gamma_{0^s,1^{s'}}=\Gamma_{0^s} \cup \Gamma_{1^{s'}}$ is a basic, $\tau$-closed, $\integer_2$-minimal set of $\achebsr{7}$-generators of $\mod*KQ^{\coxA_7}$, where $s,s' \in \{+,-\}$. Note that the elements of $\Gamma_{0^s,1^{s'}}$ bijectively correspond to the columns of the Auslander-Reiten quiver, which in turn, bijectively correspond to the positive roots of $\coxI_2(8)$ (this is Corollary~\ref{cor:GenCorrespondence}). The object $\begin{smallmatrix} 2^+ \ 2^- \\ 3 \ \end{smallmatrix}$, for example, is $\achebsr{7}$-generated by the object $\begin{smallmatrix} 0^+ \ 2^+ \\ 1^+ \end{smallmatrix}$ with $\achebr{7}$-index $\croot_2-1$, or equivalently, with the $\achebsr{7}$-generating pair $(\croot_2,1)$ (this is Definition~\ref{defn:RPGeneration}). That is, we have an isomorphism of split exact sequences
	\begin{equation*}
		\xymatrix{
			0
				\ar[r]
			& {\begin{smallmatrix} 0^+ \ 2^+ \\ 1^+ \end{smallmatrix}}
				\ar[r]^-f
				\ar[d]^-\sim
			& {\croot_2\left(\begin{smallmatrix} 0^+ \ 2^+ \\ 1^+ \end{smallmatrix}\right)}
				\ar[r]
				\ar[d]^-\sim
			& \Coker f
				\ar[r]
				\ar[d]^-\sim
			& 0
			\\
			0
				\ar[r]
			& {\begin{smallmatrix} 0^+ \ 2^+ \\ 1^+ \end{smallmatrix}}
				\ar[r]
			&  {\begin{smallmatrix} 0^+ \ 2^+ \\ 1^+ \end{smallmatrix} \oplus \begin{smallmatrix} 2^+ \ 2^- \\ 3 \ \end{smallmatrix}}
 				\ar[r]
			& {\begin{smallmatrix} 2^+ \ 2^- \\ 3 \ \end{smallmatrix}}
				\ar[r]
			& 0
		}
	\end{equation*}
	By Remark~\ref{rem:GrothendieckModule}, the Grothendieck group $K_0(\mod*KQ^{\coxA_7})$ has the structure of a $\achebr{7}$-module.
	
	By Section~\ref{sec:ClusterTilting}, the action of $\achebsr{7}$ naturally extends to the cluster category $\clus_{\coxA_7}$. We will choose our basic, $\integer_2$ set of $\achebsr{7}$-generators of $\clus_{\coxA_7}$ to be
	\begin{equation*}
		\Gamma = \{
		\sus \begin{smallmatrix} 0^- \\  1^- \end{smallmatrix},
		0^+,
		\begin{smallmatrix} 2^+ \\ 1^+  \end{smallmatrix},
		\begin{smallmatrix} 2^- \\ 3  \end{smallmatrix},
		\begin{smallmatrix} 0^- \\  1^- \end{smallmatrix}
		\} \cup \{
		\sus 1^-,
		\begin{smallmatrix} 0^+ \ 2^+ \\ 1^+ \end{smallmatrix},
		\begin{smallmatrix} \ \ \; 2^+ \ 2^- \\ 1^+ \ 3 \ \ \;  \end{smallmatrix},
		\begin{smallmatrix} \ \ 2^- \ 0^- \\ 3 \ 1^- \ \ \end{smallmatrix},
		1^-
		\}
	\end{equation*}
	With respect to $\Gamma$, the object $\begin{smallmatrix} 0^- \\  1^- \end{smallmatrix}$ is almost complete $\achebsr{7}$-tilting with complements $\begin{smallmatrix} \ \ 2^- \ 0^- \\ 3 \ 1^- \ \ \end{smallmatrix}$ and $1^-$. Note that these complements reside in adjacent columns of the Auslander-Reiten quiver to $\begin{smallmatrix} 0^- \\  1^- \end{smallmatrix}$, as per Theorem~\ref{thm:RPTilting}. Thus for example, the object $T= \begin{smallmatrix} 0^- \\  1^- \end{smallmatrix} \oplus 1^-$ is basic $\achebsr{7}$-tilting with respect to $\Gamma$, and the act of changing complement corresponds to mutating the folded quiver $\coxI_2(8)$ at a vertex. Note also that by Theorem~\ref{thm:RPTilting}, $T$ corresponds to a cluster-tilting object
	\begin{equation*}
		\wh{T} = \left(\begin{smallmatrix} 0^- \\  1^- \end{smallmatrix} \oplus \begin{smallmatrix} 2^- \\  1^- \ 3 \end{smallmatrix} \oplus \begin{smallmatrix} 2^+ \\  1^+ \ 3 \end{smallmatrix} \oplus \begin{smallmatrix} 0^+ \\  1^+ \end{smallmatrix}\right) \oplus \left(1^- \oplus 3 \oplus 1^+\right)
	\end{equation*}
	and mutation/changing complement in $T$ corresponds to composite mutation/iteratively changing complement for each object in a bracket in $\wh{T}$.
	
	Now we shall compute the folded $g$-vectors in $\clus_{\coxA_7}$, as per Definition~\ref{defn:CatGVect}. The set $\Gamma$ consists of all objects in rows of weight 1 (up to the group action of $\croot_6$), so we will focus on computing the folded $g$-vectors of these objects --- it is not difficult to check that the folded $g$-vectors $g^X$ and $g^{\croot_6 X}$ are in fact the same. We will compute the $g$-vectors with respect to the projectives $\begin{smallmatrix} 0^- \\  1^- \end{smallmatrix}$ and $1^-$. We thus have the following triangles and corresponding folded $g$-vectors, where for readability purposes, we have denoted the zero object by $\mathbf{0}$ and the identity element $1 \in \achebsr{7}$ by $\croot_0$ (or omitted it completely if the action on an object is trivial).
	\begin{align*}
		\mathbf{0} \rightarrow 1^- \rightarrow 1^- \rightarrow \mathbf{0}& & &\rightsquigarrow &
		g^{1^-} &= (1,0), \\
		\mathbf{0} \rightarrow \begin{smallmatrix} 0^- \\  1^- \end{smallmatrix} \rightarrow \begin{smallmatrix} 0^- \\  1^- \end{smallmatrix} \rightarrow \mathbf{0}& &
		&\rightsquigarrow & 
		g^{\begin{smallmatrix} 0^- \\  1^- \end{smallmatrix}} &= (0,1), \\
		1^- \rightarrow (\croot_0+\croot_2) \begin{smallmatrix} 0^- \\  1^- \end{smallmatrix} \rightarrow \begin{smallmatrix} \ \ 2^- \ 0^- \\ 3 \ 1^- \ \ \end{smallmatrix} \rightarrow \sus 1^-& &
		&\rightsquigarrow & 
		g^{\begin{smallmatrix} \ \ 2^- \ 0^- \\ 3 \ 1^- \ \ \end{smallmatrix}} &= (-1,2+\sqrt{2}), \\
		1^- \rightarrow \croot_2 \begin{smallmatrix} 0^- \\  1^- \end{smallmatrix} \rightarrow \begin{smallmatrix} 2^- \\ 3  \end{smallmatrix} \rightarrow \sus 1^-&
		& &\rightsquigarrow & 
		g^{\begin{smallmatrix} 2^- \\ 3  \end{smallmatrix}} &= (-1,1+\sqrt{2}), \\
		\croot_2 1^- \rightarrow (\croot_2 + \croot_4) \begin{smallmatrix} 0^- \\  1^- \end{smallmatrix} \rightarrow \begin{smallmatrix} \ \ \; 2^+ \ 2^- \\ 1^+ \ 3 \ \ \;  \end{smallmatrix} \rightarrow \sus \croot_2 1^-&
		& &\rightsquigarrow & 
		g^{\begin{smallmatrix} \ \ \; 2^+ \ 2^- \\ 1^+ \ 3 \ \ \;  \end{smallmatrix}} &= (-1-\sqrt{2},2+2\sqrt{2}), \\
		\croot_2 1^- \rightarrow 1^- \oplus \croot_4 \begin{smallmatrix} 0^- \\  1^- \end{smallmatrix} \rightarrow \begin{smallmatrix} 2^+ \\ 1^+  \end{smallmatrix} \rightarrow \sus \croot_2 1^-&
		& &\rightsquigarrow & 
		g^{\begin{smallmatrix} 2^+ \\ 1^+  \end{smallmatrix}} &= (-\sqrt{2},1+\sqrt{2}), \\
		\croot_4 1^- \rightarrow (\croot_4+\croot_6) \begin{smallmatrix} 0^- \\  1^- \end{smallmatrix} \rightarrow \begin{smallmatrix} 0^+ \ 2^+ \\ 1^+ \end{smallmatrix} \rightarrow \sus\croot_4 1^-&
		& &\rightsquigarrow & 
		g^{\begin{smallmatrix} 0^+ \ 2^+ \\ 1^+  \end{smallmatrix}} &= (-1-\sqrt{2},2+\sqrt{2}), \\
		\croot_6 1^- \rightarrow \croot_6 \begin{smallmatrix} 0^- \\  1^- \end{smallmatrix} \rightarrow 0^+ \rightarrow \sus \croot_6 1^-&
		& &\rightsquigarrow & 
		g^{0^+} &= (-1,1), \\
		1^- \rightarrow \mathbf{0} \rightarrow \sus 1^- \rightarrow \sus 1^-&
		& &\rightsquigarrow & 
		g^{\sus 1^-} &= (-1,0), \\
		\begin{smallmatrix} 0^- \\  1^- \end{smallmatrix} \rightarrow \mathbf{0} \rightarrow \sus \begin{smallmatrix} 0^- \\  1^- \end{smallmatrix} \rightarrow \sus \begin{smallmatrix} 0^- \\  1^- \end{smallmatrix}&
		& &\rightsquigarrow & 
		g^{\sus \begin{smallmatrix} 0^- \\  1^- \end{smallmatrix}} &= (0,-1).
	\end{align*}
	In the above, we have used the fact that $\rhom{\coxA_7}(\croot_2)=\rhom{\coxA_7}(\croot_4)=U_2(\cos\frac{\pi}{8})=1 + \sqrt{2}$ and $\rhom{\coxA_7}(\croot_0)=\rhom{\coxA_7}(\croot_6)=1$. As per Corollary~\ref{cor:GVectors}, these are $g$-vectors of the standard root system of type $\coxI_2(8)$, and for each basic $\achebsr{7}$-tilting object $T = X_0 \oplus X_1$, we obtain a $G$-matrix of $\coxI_2(8)$ by $G^{X_1,X_0}=(g^{X_1},g^{X_0})$. For example, the $\achebsr{7}$-tilting object $\begin{smallmatrix} 0^- \\  1^- \end{smallmatrix} \oplus 1^-$ corresponds to the initial $G$-matrix. Using Theorem~\ref{thm:Tesseract}, we can mutate the initial $G$-matrix $G^{1^-,\begin{smallmatrix} 0^- \\  1^- \end{smallmatrix}}$ at index $[0]$ to obtain
	\begin{equation*}
		\xymatrix{ 
		{\begin{pmatrix}
			1 & 0 \\ 0 & 1
		\end{pmatrix}}
		\ar@{<-|}[r]^-{({-}^T)\inv}
		\ar@{|->}[d]^-{\mu_{[0]}} &
		{\begin{pmatrix}
			1 & 0 \\ 0 & 1
		\end{pmatrix}}
		\ar@{=}[r] 
		\ar@{|-->}[d]^-{\mu_{[0]}} &
		G^{1^-,\begin{smallmatrix} 0^- \\  1^- \end{smallmatrix}}
		\ar@{|-->}[d]^-{\mu_{[0]}}
		 \\
		{\begin{pmatrix}
			-1 & 2+ \sqrt{2} \\ 0 & 1
		\end{pmatrix}}
		\ar@{|->}[r]^-{({-}^T)\inv} &
		{\begin{pmatrix}
			-1 & 0 \\ 2+ \sqrt{2} & 1
		\end{pmatrix}}
		\ar@{=}[r] &
		G^{\begin{smallmatrix} \ \ 2^- \ 0^- \\ 3 \ 1^- \ \ \end{smallmatrix}, \begin{smallmatrix} 0^- \\  1^- \end{smallmatrix}}
		}
	\end{equation*}
	which corresponds to changing complement, as required.
	
	\subsection{The folding $F^{\coxD_5}$} \label{sec:DExample}
	Inevitably, combinatorics in the setting of foldings of type $\coxD$ is more complicated than in the setting of type $\coxA$. Since $F^{\coxD_5}$ is also a folding onto $Q^{\coxI_2(8)}$, we will focus on the aspects of the theory that differ from the type $\coxA_7$ setting. To further simplify things, most of this example will concern the category $\acat=\mod*KQ^{\coxD_5}$, as the theory for $\bder(\mod*KQ^{\coxD_5})$ is a straightforward extension of this. The non-crystallographic projection of $\acat$ is illustrated in Figure~\ref{fig:D5I8}. We remind the reader that whilst the positive simple roots are of length $2$ and $4 \cos\frac{\pi}{8}$ respectively (under the standard basis of $\real^2$), the positive simple roots are written as $(1,0)$ and $(0,1)$ with respect to the standard root system of $\coxI_2(8)$.
	
	\begin{figure}[t]
		\centering
		\def\y{0.35cm}
\begin{tikzpicture}[scale=2]
	\foreach \x in {0,...,3} {
		\foreach \z in {0,2} {
			\coordinate (t\x_I\z) at ($(1.75-\x,{\z/2 + 3.25})$);
		}
	}
	\foreach \x in {0,...,3} {
		\foreach \z in {1,3,4} {
			\coordinate (t\x_I\z) at ($(1.25-\x,{((\z - 1)/2 + 3.75)})$);
		}
	}
	
	\draw[<-, shorten <= \y*1.5,shorten >= \y*1.5,gray] (t0_I0) -- (t0_I1);
	\draw[<-, shorten <= \y*1.5,shorten >= \y*1.5,gray] (t0_I2) -- (t0_I1);
	\draw[<-, shorten <= \y*1.5,shorten >= \y*1.5,gray] (t0_I2) -- (t0_I3);
	\draw[<-, shorten <= \y*1.5,shorten >= \y*1.5,gray] (t0_I2) -- (t0_I4);
	
	\draw[->, shorten <= \y*1.5,shorten >= \y*1.5,gray] (t1_I0) -- (t0_I1);
	\draw[->, shorten <= \y*1.5,shorten >= \y*1.5,gray] (t1_I2) -- (t0_I1);
	\draw[->, shorten <= \y*1.5,shorten >= \y*1.5,gray] (t1_I2) -- (t0_I3);
	\draw[->, shorten <= \y*1.5,shorten >= \y*1.5,gray] (t1_I2) -- (t0_I4);
	
	\draw[<-, shorten <= \y*1.5,shorten >= \y*1.5,gray] (t1_I0) -- (t1_I1);
	\draw[<-, shorten <= \y*1.5,shorten >= \y*1.5,gray] (t1_I2) -- (t1_I1);
	\draw[<-, shorten <= \y*1.5,shorten >= \y*1.5,gray] (t1_I2) -- (t1_I3);
	\draw[<-, shorten <= \y*1.5,shorten >= \y*1.5,gray] (t1_I2) -- (t1_I4);
	
	\draw[->, shorten <= \y*1.5,shorten >= \y*1.5,gray] (t2_I0) -- (t1_I1);
	\draw[->, shorten <= \y*1.5,shorten >= \y*1.5,gray] (t2_I2) -- (t1_I1);
	\draw[->, shorten <= \y*1.5,shorten >= \y*1.5,gray] (t2_I2) -- (t1_I3);
	\draw[->, shorten <= \y*1.5,shorten >= \y*1.5,gray] (t2_I2) -- (t1_I4);
	
	\draw[<-, shorten <= \y*1.5,shorten >= \y*1.5,gray] (t2_I0) -- (t2_I1);
	\draw[<-, shorten <= \y*1.5,shorten >= \y*1.5,gray] (t2_I2) -- (t2_I1);
	\draw[<-, shorten <= \y*1.5,shorten >= \y*1.5,gray] (t2_I2) -- (t2_I3);
	\draw[<-, shorten <= \y*1.5,shorten >= \y*1.5,gray] (t2_I2) -- (t2_I4);
	
	\draw[->, shorten <= \y*1.5,shorten >= \y*1.5,gray] (t3_I0) -- (t2_I1);
	\draw[->, shorten <= \y*1.5,shorten >= \y*1.5,gray] (t3_I2) -- (t2_I1);
	\draw[->, shorten <= \y*1.5,shorten >= \y*1.5,gray] (t3_I2) -- (t2_I3);
	\draw[->, shorten <= \y*1.5,shorten >= \y*1.5,gray] (t3_I2) -- (t2_I4);
	
	\draw[<-, shorten <= \y*1.5,shorten >= \y*1.5,gray] (t3_I0) -- (t3_I1);
	\draw[<-, shorten <= \y*1.5,shorten >= \y*1.5,gray] (t3_I2) -- (t3_I1);
	\draw[<-, shorten <= \y*1.5,shorten >= \y*1.5,gray] (t3_I2) -- (t3_I3);
	\draw[<-, shorten <= \y*1.5,shorten >= \y*1.5,gray] (t3_I2) -- (t3_I4);
	
	\draw (t0_I0) node {\footnotesize$0$};
	\draw (t1_I0) node {\footnotesize$\begin{smallmatrix} 2 \\ 1  \end{smallmatrix}$};
	\draw (t2_I0) node {\footnotesize$\begin{smallmatrix} 2 \ \\ 3^+ \ 3^-  \end{smallmatrix}$};
	\draw (t3_I0) node {\footnotesize$\begin{smallmatrix} 0 \\ 1  \end{smallmatrix}$};
	
	\draw (t0_I1) node {\footnotesize$\begin{smallmatrix} 0 \ 2 \\ 1 \end{smallmatrix}$};
	\draw (t1_I1) node {\footnotesize$\begin{smallmatrix} 2 \ \ 2 \ \ \\ 1 \ 3^+ \ 3^- \end{smallmatrix}$};
	\draw (t2_I1) node {\footnotesize$\begin{smallmatrix} 0 \ \; 2 \quad \\ \quad 1 \, 3^+  3^- \end{smallmatrix}$};
	\draw (t3_I1) node {\footnotesize$1$};
	
	\draw (t0_I2) node {\footnotesize$2$};
	\draw (t1_I2) node {\footnotesize$\begin{smallmatrix} 0 \ \; 2 \ \ 2 \quad \\ \ \ 1 \ 3^+ \ 3^-  \end{smallmatrix}$};
	\draw (t2_I2) node {\footnotesize$\begin{smallmatrix} 0 \quad 2 \quad \, 2 \ \; \\ \ \; 1 \, 3^+  3^- \ 1 \end{smallmatrix}$};
	\draw (t3_I2) node {\footnotesize$\begin{smallmatrix} 2 \; \\ \; 1 \, 3^+ 3^- \end{smallmatrix}$};
	
	\draw[] (t0_I3) node {\footnotesize$\begin{smallmatrix} 2 \; \\ \; 3^+ \ \end{smallmatrix}$};
	\draw[] (t1_I3) node {\footnotesize$\begin{smallmatrix} 0 \ 2 \ \ \\ \ \ 1 \ 3^- \end{smallmatrix}$};
	\draw[] (t2_I3) node {\footnotesize$\begin{smallmatrix} 2 \ \\ \ 1 \ 3^+ \end{smallmatrix}$};
	\draw[] (t3_I3) node {\footnotesize$3^-$};
	
	\draw (t0_I4) node {\footnotesize$\begin{smallmatrix} 2 \; \\ \; 3^- \ \end{smallmatrix}$};
	\draw (t1_I4) node {\footnotesize$\begin{smallmatrix} 0 \ 2 \ \ \\ \ \ 1 \ 3^+ \end{smallmatrix}$};
	\draw (t2_I4) node {\footnotesize$\begin{smallmatrix} 2 \ \\ \ 1 \ 3^- \end{smallmatrix}$};
	\draw (t3_I4) node {\footnotesize$3^+$};
	
	\foreach \x in {1,...,8} {
		\coordinate (m\x) at ($({cos(22.5*\x-22.5)}, {sin(22.5*\x-22.5)})$);
		\coordinate (p1m\x) at ($({2*cos(22.5)*cos(22.5*\x-22.5)}, {2*cos(22.5)*sin(22.5*\x-22.5)})$);
		\coordinate (p2m\x) at ($({(4*(cos(22.5))^2-1)*cos(22.5*\x-22.5)}, {(4*(cos(22.5))^2-1)*sin(22.5*\x-22.5)})$);
		\coordinate (p3m\x) at ($({(4*(cos(22.5))^3-2*cos(22.5))*cos(22.5*\x-22.5)}, {(4*(cos(22.5))^3-2*cos(22.5))*sin(22.5*\x-22.5)})$);
		
		\coordinate (n\x) at ($-1*(m\x)$);
		\coordinate (p1n\x) at ($-1*(p1m\x)$);
		\coordinate (p2n\x) at ($-1*(p2m\x)$);
		\coordinate (p3n\x) at ($-1*(p3m\x)$);
	}
	\draw[red,opacity=0.5] (m1) -- (m3) -- (m5) -- (m7) -- ($0.5*(n1)+0.5*(m7)$);
	\draw[red,opacity=0.4] ($0.5*(p1n8)+0.5*(p1m2)$) -- (p1m2) -- (p1m4) -- (p1m6) -- (p1m8);
	\draw[red,opacity=0.3] (p2m1) -- (p2m3) -- (p2m5) -- (p2m7) -- ($0.5*(p2n1)+0.5*(p2m7)$);
	\draw[red,opacity=0.2] ($0.5*(p3n8)+0.5*(p3m2)$) -- (p3m2) -- (p3m4) -- (p3m6) -- (p3m8);

	\draw (m1) node {\footnotesize$0$};
	\draw (m3) node {\footnotesize$\begin{smallmatrix} 2 \\ 1  \end{smallmatrix}$};
	\draw (m5) node {\footnotesize$\begin{smallmatrix} 2 \ \\ 3^+ \ 3^-  \end{smallmatrix}$};
	\draw (m7) node {\footnotesize$\begin{smallmatrix} 0 \\ 1  \end{smallmatrix}$};
	
	\draw (p1m2) node {\footnotesize$\begin{smallmatrix} 0 \ 2 \\ 1 \end{smallmatrix}$};
	\draw (p1m4) node {\footnotesize$\begin{smallmatrix} 2 \ \ 2 \ \ \\ 1 \ 3^+ \ 3^- \end{smallmatrix}$};
	\draw (p1m6) node {\footnotesize$\begin{smallmatrix} 0 \quad 2 \quad \\ \quad 1 \, 3^+  3^- \end{smallmatrix}$};
	\draw (p1m8) node {\footnotesize$1$};
	
	\draw (p2m1) node {\footnotesize$2$};
	\draw (p2m3) node {\footnotesize$\begin{smallmatrix} 0 \ \; 2 \ \ 2 \quad \\ \ \ 1 \ 3^+ \ 3^-  \end{smallmatrix}$};
	\draw (p2m5) node {\footnotesize$\begin{smallmatrix} 0 \quad 2 \quad \, 2 \ \; \\ \ \; 1 \, 3^+  3^- \ 1 \end{smallmatrix}$};
	\draw (p2m7) node {\footnotesize$\begin{smallmatrix} 2 \; \\ \; 1 \, 3^+ 3^- \end{smallmatrix}$};
	
	\draw (p3m2) node {\footnotesize$\begin{smallmatrix} 2 \; \\ \; 3^\pm \ \end{smallmatrix}$};
	\draw (p3m4) node {\footnotesize$\begin{smallmatrix} 0 \ 2 \ \ \\ \ \ 1 \ 3^\mp \end{smallmatrix}$};
	\draw (p3m6) node {\footnotesize$\begin{smallmatrix} 2 \ \\ \ 1 \ 3^\pm \end{smallmatrix}$};
	\draw (p3m8) node {\footnotesize$3^\mp$};
	
	\draw[<-, shorten <= \y,shorten >= \y,gray] (m1) -- (p1m2);
	\draw[<-, shorten <= \y,shorten >= \y,gray] (p2m1) -- (p1m2);
	\draw[<-, shorten <= \y,shorten >= \y,gray] (p2m1) -- (p3m2);
	
	\draw[->, shorten <= \y,shorten >= \y,gray] (m3) -- (p1m2);
	\draw[->, shorten <= \y,shorten >= \y,gray] (p2m3) -- (p1m2);
	\draw[->, shorten <= \y,shorten >= \y,gray] (p2m3) -- (p3m2);
	
	\draw[<-, shorten <= \y,shorten >= \y,gray] (m3) -- (p1m4);
	\draw[<-, shorten <= \y*2,shorten >= \y*2,gray] (p2m3) -- (p1m4);
	\draw[<-, shorten <= \y*1.5,shorten >= \y,gray] (p2m3) -- (p3m4);
	
	\draw[->, shorten <= \y*1,shorten >= \y*1.25,gray] (m5) -- (p1m4);
	\draw[->, shorten <= \y*1.5,shorten >= \y,gray] (p2m5) -- (p1m4);
	\draw[->, shorten <= \y*1.5,shorten >= \y*1,gray] (p2m5) -- (p3m4);
	
	\draw[<-, shorten <= \y*1,shorten >= \y*1.25,gray] (m5) -- (p1m6);
	\draw[<-, shorten <= \y*1.5,shorten >= \y*1.25,gray] (p2m5) -- (p1m6);
	\draw[<-, shorten <= \y*1.5,shorten >= \y*1,gray] (p2m5) -- (p3m6);
	
	\draw[->, shorten <= \y,shorten >= \y*1,gray] (m7) -- (p1m6);
	\draw[->, shorten <= \y*1.5,shorten >= \y*1.5,gray] (p2m7) -- (p1m6);
	\draw[->, shorten <= \y*1.5,shorten >= \y*1,gray] (p2m7) -- (p3m6);
	
	\draw[<-, shorten <= \y*1,shorten >= \y*1,gray] (m7) -- (p1m8);
	\draw[<-, shorten <= \y*1.25,shorten >= \y*1,gray] (p2m7) -- (p1m8);
	\draw[<-, shorten <= \y*1.5,shorten >= \y*1,gray] (p2m7) -- (p3m8);
\end{tikzpicture}
		\caption{The folding $F^{\coxD_5}\colon Q^{\coxD_5} \rightarrow Q^{\coxI_2(8)}$. Top: The Auslander-Reiten quiver of $\mod*KQ^{\coxD_5}$. Bottom: The non-crystallographic projection of the Auslander-Reiten quiver under the map $\dimproj^{\mod*KQ^{\coxD_5}}_{F^{\coxD_5}}$, with irreducible morphisms superimposed. The radius of each half-octagon (from the centre outwards) is $2$, $U_3(\cos\frac{\pi}{8})$, $2U_1(\cos\frac{\pi}{8})$ and $2U_2(\cos\frac{\pi}{8})$.} \label{fig:D5I8}
	\end{figure}
	
	The category $\acat$ has an action of $\dchebsr{5} = \nnint[\croot^-_0,\croot_2^+,\croot_2^-]$, where we will also write $1=\croot_0^+ \in \dchebsr{5}$ for clarity. Here, the element $\croot_0^-$ represents the $\integer_2$-symmetry on the quiver. Similar to the theory for type $\coxA$, we can label the Auslander-Reiten sequences by elements of $\hdchebsr{5}$, and this label is determined by the row that the starting/ending term resides in. In particular, we can assign an Auslander-Reiten sequence a label of $\croot_i^+ + \croot_i^-$ (for $i \in \{0,1,2\}$) if the starting/ending term resides in the row $\mathcal{I}^\acat_i$, and a label of $\croot_3^\pm$ if the starting/ending term resides in the row $\mathcal{I}^\acat_{3^\pm}$. The action of an element $r \in \dchebsr{5}$ on an Auslander-Reiten sequence labelled by $s \in \dchebsr{5} $ is then such that the sequence is mapped to an exact sequence isomorphic to a direct sum of Auslander-Reiten sequences whose labels are determined from the product $rs$. That is, we can write $rs = \sum_{i=0}^2 a_i (\croot_i^++\croot_i^-) + a_{3^+} \croot_3^+ + a_{3^-} \croot_3^-$, and thus the resulting sequence is isomorphic to a direct sum of $a_i$ Auslander-Reiten sequences labelled by $\croot_i^+ +\croot_i^-$ and $a_{3^\pm}$ Auslander-Reiten sequences labelled by $\croot_3^\pm$ (all in the same column). Likewise, we can also label individual indecomposable objects by elements of $\hdchebsr{5}$, which is determined by the row the object resides in (as before). The action of $\dchebsr{5}$ on an object will also always respect this labelling (column-wise).
	
	To demonstrate this, consider the Auslander-Reiten sequences
	\begin{gather} \tag{S0} \label{eq:ars1}
		\mathbf{0} \rightarrow
		\begin{smallmatrix} 2 \\ 1  \end{smallmatrix}
		\rightarrow 
		\begin{smallmatrix} 0 \ 2 \\ 1 \end{smallmatrix} 
		\rightarrow 
		0
		\rightarrow \mathbf{0} 
		\\ \tag{S1} \label{eq:ars2}
		\mathbf{0} \rightarrow
		\begin{smallmatrix} 2 \ \ 2 \ \ \\ 1 \ 3^+ \ 3^- \end{smallmatrix}
		\rightarrow 
		\begin{smallmatrix} 2 \\ 1  \end{smallmatrix} \oplus \begin{smallmatrix} 0 \ \; 2 \ \ 2 \quad \\ \ \ 1 \ 3^+ \ 3^-  \end{smallmatrix}
		\rightarrow
		\begin{smallmatrix} 0 \ 2 \\ 1 \end{smallmatrix} 
		\rightarrow \mathbf{0} 
		\\ \tag{S2} \label{eq:ars3}
		\mathbf{0} \rightarrow
		\begin{smallmatrix} 0 \ \; 2 \ \ 2 \quad \\ \ \ 1 \ 3^+ \ 3^-  \end{smallmatrix}
		\rightarrow
		\begin{smallmatrix} 0 \ 2 \\ 1 \end{smallmatrix}  \oplus \begin{smallmatrix} 2 \; \\ \; 3^+ \ \end{smallmatrix} \oplus \begin{smallmatrix} 2 \; \\ \; 3^- \ \end{smallmatrix}
		\rightarrow
		2
		\rightarrow \mathbf{0}
		\\ \tag{S3p} \label{eq:ars4}
		\mathbf{0} \rightarrow
		\begin{smallmatrix} 0 \ 2 \ \ \\ \ \ 1 \ 3^- \end{smallmatrix}
		\rightarrow 
		\begin{smallmatrix} 0 \ \; 2 \ \ 2 \quad \\ \ \ 1 \ 3^+ \ 3^-  \end{smallmatrix}
		\rightarrow 
		\begin{smallmatrix} 2 \; \\ \; 3^+ \ \end{smallmatrix}
		\rightarrow \mathbf{0}
		\\ \tag{S3m} \label{eq:ars5}
		\mathbf{0} \rightarrow
		\begin{smallmatrix} 0 \ 2 \ \ \\ \ \ 1 \ 3^+ \end{smallmatrix}
		\rightarrow 
		\begin{smallmatrix} 0 \ \; 2 \ \ 2 \quad \\ \ \ 1 \ 3^+ \ 3^-  \end{smallmatrix}
		\rightarrow 
		\begin{smallmatrix} 2 \; \\ \; 3^- \ \end{smallmatrix}
		\rightarrow \mathbf{0}
	\end{gather}
	where for clarity, we have denoted the zero object by $\mathbf{0}$ to distinguish it from the simple module $0$. The above Auslander-Reiten sequences (and the starting and ending terms) may be labelled by $\croot_0^+ + \croot_0^-$, $\croot_1^+ + \croot_1^-$, $\croot_2^+ + \croot_2^-$, $\croot_3^+$ and $\croot_3^-$, respectively. The objects in the middle terms of each sequence therefore have labels $\croot_1^+ + \croot_1^-$, $\croot_0^+ + \croot_0^- + \croot_2 + \croot_2^-$, $\croot_1^+ + \croot_1^- + \croot_3^+ + \croot_3^-$, $\croot_2^+ + \croot_2^-$ and $\croot_2^+ + \croot_2^-$, respectively.
	
	The action of $\croot_2^\pm$ on the Auslander-Reiten sequence (\ref{eq:ars1}) is a map to (a sequence isomorphic to) the Auslander-Reiten sequence (\ref{eq:ars3}). This reflects the product $\croot_2^\pm(\croot_0^+ + \croot_0^-)=\croot_2^+ + \croot_2^-$ for the sequence itself, and simultaneously, the product $\croot_2^\pm(\croot_1^+ + \croot_1^-)=\croot_1^+ + \croot_1^- + \croot_3^+ + \croot_3^-$ for the middle term. The action of $\croot_2^\pm$ on (\ref{eq:ars3}), or equivalently the action of $(\croot_2^\pm)^2 = \croot_0^+ + \croot_2^+ + \croot_2^-$ on (\ref{eq:ars1}), is such that we obtain a sequence isomorphic to (\ref{eq:ars1}) $\oplus$ (\ref{eq:ars3}) $\oplus$ (\ref{eq:ars3}). Both of these examples showcase isomorphism condition (\ref{eq:DIso1}).
	
	Another example is the action of $\croot_0^-$, which representing $\integer_2$-symmetry, maps (\ref{eq:ars4}) to (\ref{eq:ars5}) and vice versa. On the other hand the action of $\croot_0^-$ on any other sequence in the list is trivial (up to isomorphism).
	
	Finally, one needs to be careful of complicated sign changes with some of the products in $\hdchebsr{5}$. The effect of this is highlighted in isomorphism condition (\ref{eq:DIso2}). For example, $\croot_2^+\croot_3^+ = \croot_1^+ + \croot_1^- +\croot_3^-$. Thus, the action of $\croot_2^+$ on (\ref{eq:ars4}) is such that we obtain (\ref{eq:ars2}) $\oplus$ (\ref{eq:ars5}).
	
	Now we define the following three sets.
	\begin{equation*}
		\Gamma_0 = \{
		0,
		\begin{smallmatrix} 2 \\ 1  \end{smallmatrix},
		\begin{smallmatrix} 2 \ \\ 3^+ \ 3^-  \end{smallmatrix},
		\begin{smallmatrix} 0 \\ 1  \end{smallmatrix}
		\}
		\qquad \text{and} \qquad
		\Gamma_\pm = \{
		\begin{smallmatrix} 2 \; \\ \; 3^\mp \ \end{smallmatrix},
		\begin{smallmatrix} 0 \ 2 \ \ \\ \ \ 1 \ 3^\pm \end{smallmatrix},
		\begin{smallmatrix} 2 \ \\ \ 1 \ 3^\mp \end{smallmatrix},
		3^\pm
		\}
	\end{equation*}
	Theorem~\ref{thm:RPGenerators}(c) then states that the sets $\Gamma_0\cup\Gamma_\pm$ are pairwise equivalent, basic, $\tau$-closed, $\integer_2$-minimal sets of $\dchebsr{5}$-generators of $\acat$. Extending this to the cluster category $\clus_{\coxD_5}$, we obtain similar sets of basic, $\integer_2$-minimal $\dchebsr{5}$-generators
	\begin{equation*}
		\Gamma^\pm = \{
		\sus \begin{smallmatrix} 0 \\ 1  \end{smallmatrix},
		0,
		\begin{smallmatrix} 2 \\ 1  \end{smallmatrix},
		\begin{smallmatrix} 2 \ \\ 3^+ \ 3^-  \end{smallmatrix},
		\begin{smallmatrix} 0 \\ 1  \end{smallmatrix}
		\}
		\cup \{
		\sus 3^\pm,
		\begin{smallmatrix} 2 \; \\ \; 3^\mp \ \end{smallmatrix},
		\begin{smallmatrix} 0 \ 2 \ \ \\ \ \ 1 \ 3^\pm \end{smallmatrix},
		\begin{smallmatrix} 2 \ \\ \ 1 \ 3^\mp \end{smallmatrix},
		3^\pm
		\}.
	\end{equation*}
	An example of a basic $\dchebsr{5}$-tilting object with respect to $\Gamma^+$ is therefore $\begin{smallmatrix} 2 \\ 1  \end{smallmatrix} \oplus \begin{smallmatrix} 0 \ 2 \ \ \ \\ \ 1 \ 3^+ \end{smallmatrix}$ and the corresponding basic $\dchebsr{5}$-tilting object with respect to $\Gamma^-$ is $\begin{smallmatrix} 2 \\ 1  \end{smallmatrix} \oplus \begin{smallmatrix} 0 \ 2 \ \ \ \\ \ 1 \ 3^- \end{smallmatrix}$. The theory does not depend on which set we choose, as long as we are consistent with our choice.
	
	Similar to the situation in type $\coxA$, we can compute the $g$-vectors of $\coxI_2(8)$ from the category $\clus_{\coxD_5}$. However this time, our choice of objects in the resolution do not always belong to a set of $\dchebsr{5}$-generators of $\clus_{\coxD_5}$. For example, the row containing $3^\pm$ does not have the correct weight, so we must instead choose the object $1$. We therefore have the following triangles and corresponding folded $g$-vectors.
	\begin{align*}
		\mathbf{0} \rightarrow 1 \rightarrow 1 \rightarrow \mathbf{0}& & &\rightsquigarrow &
		g^{1} &= (1,0), \\
		\mathbf{0} \rightarrow \begin{smallmatrix} 0 \\  1 \end{smallmatrix} \rightarrow \begin{smallmatrix} 0 \\  1 \end{smallmatrix} \rightarrow \mathbf{0}& &
		&\rightsquigarrow & 
		g^{\begin{smallmatrix} 0 \\  1 \end{smallmatrix}} &= (0,1), \\
		1 \rightarrow (\croot_0^++\croot_2^+) \begin{smallmatrix} 0 \\  1 \end{smallmatrix} \rightarrow \begin{smallmatrix} 0 \ \; 2 \quad \\ \ \ 1 \, 3^+  3^- \end{smallmatrix} \rightarrow \sus 1& &
		&\rightsquigarrow & 
		g^{\begin{smallmatrix} 0 \ \; 2 \quad \\ \ \ 1 \, 3^+  3^- \end{smallmatrix}} &= (-1,2+\sqrt{2}), \\
		1 \rightarrow \croot_2^+ \begin{smallmatrix} 0 \\  1 \end{smallmatrix} \rightarrow \begin{smallmatrix} 2 \ \\ 3^+ \ 3^-  \end{smallmatrix} \rightarrow \sus 1&
		& &\rightsquigarrow & 
		g^{\begin{smallmatrix} 2 \ \\ 3^+ \ 3^-  \end{smallmatrix}} &= (-1,1+\sqrt{2}), \\
		\croot_2 1 \rightarrow (\croot_2^+ + \croot_2^-) \begin{smallmatrix} 0 \\  1 \end{smallmatrix} \rightarrow \begin{smallmatrix} 2 \ \ 2 \ \ \\ 1 \ 3^+ \ 3^- \end{smallmatrix} \rightarrow \sus \croot_2^+ 1&
		& &\rightsquigarrow & 
		g^{\begin{smallmatrix} 2 \ \ 2 \ \ \\ 1 \ 3^+ \ 3^- \end{smallmatrix}} &= (-1-\sqrt{2},2+2\sqrt{2}), \\
		\croot_2^+ 1 \rightarrow 1 \oplus \croot_2^+ \begin{smallmatrix} 0 \\  1 \end{smallmatrix} \rightarrow \begin{smallmatrix} 2 \\ 1  \end{smallmatrix} \rightarrow \sus \croot_2^+ 1&
		& &\rightsquigarrow & 
		g^{\begin{smallmatrix} 2 \\ 1  \end{smallmatrix}} &= (-\sqrt{2},1+\sqrt{2}), \\
		\croot_2^+ 1 \rightarrow (\croot_0^++\croot_2^+) \begin{smallmatrix} 0 \\  1 \end{smallmatrix} \rightarrow \begin{smallmatrix} 0 \ 2 \\ 1 \end{smallmatrix} \rightarrow \sus\croot_2^+ 1&
		& &\rightsquigarrow & 
		g^{\begin{smallmatrix} 0 \ 2 \\ 1  \end{smallmatrix}} &= (-1-\sqrt{2},2+\sqrt{2}), \\
		1 \rightarrow \begin{smallmatrix} 0 \\  1 \end{smallmatrix} \rightarrow 0 \rightarrow \sus  1&
		& &\rightsquigarrow & 
		g^{0} &= (-1,1), \\
		1 \rightarrow \mathbf{0} \rightarrow \sus 1 \rightarrow \sus 1&
		& &\rightsquigarrow & 
		g^{\sus 1} &= (-1,0), \\
		\begin{smallmatrix} 0 \\  1 \end{smallmatrix} \rightarrow \mathbf{0} \rightarrow \sus \begin{smallmatrix} 0 \\  1 \end{smallmatrix} \rightarrow \sus \begin{smallmatrix} 0 \\  1 \end{smallmatrix}&
		& &\rightsquigarrow & 
		g^{\sus \begin{smallmatrix} 0 \\  1 \end{smallmatrix}} &= (0,-1).
	\end{align*}
	As we can see, these are identical to the folded $g$-vectors computed before.
	
	\bibliography{Bibliography}{}

\begin{thebibliography}{10}

\bibitem{Bernstein}
I.~N. Bernstein, I.~M. Gel'fand, and V.~A. Ponomarev.
\newblock Coxeter functors and gabriel's theorem.
\newblock {\em Russian Mathematical Surveys}, 28(2):17, apr 1973.

\bibitem{BrennerButler}
S.~Brenner and M.~C.~R. Butler.
\newblock The equivalence of certain functors occurring in the representation
  theory of artin algebras and species.
\newblock {\em Journal of the London Mathematical Society}, s2-14(1):183--187,
  1976,
  https://londmathsoc.onlinelibrary.wiley.com/doi/pdf/10.1112/jlms/s2-14.1.183.

\bibitem{ClusterTiltingSurvey}
A.~B. Buan and B.~R. Marsh.
\newblock Cluster-tilting theory.
\newblock In {\em Trends in representation theory of algebras and related
  topics}, volume 406 of {\em Contemp. Math.}, pages 1--30. Amer. Math. Soc.,
  Providence, RI, 2006.

\bibitem{BMRRT}
A.~B. Buan, B.~R. Marsh, M.~Reineke, I.~Reiten, and G.~Todorov.
\newblock Tilting theory and cluster combinatorics.
\newblock {\em Adv. Math.}, 204(2):572--618, 2006.

\bibitem{DehyKeller}
R.~Dehy and B.~Keller.
\newblock On the combinatorics of rigid objects in 2-{C}alabi-{Y}au categories.
\newblock {\em Int. Math. Res. Not. IMRN}, (11):Art. ID rnn029, 17, 2008.

\bibitem{Deodhar}
V.~V. Deodhar.
\newblock On the root system of a {C}oxeter group.
\newblock {\em Comm. Algebra}, 10(6):611--630, 1982.

\bibitem{DWZ2}
H.~Derksen, J.~Weyman, and A.~Zelevinsky.
\newblock Quivers with potentials and their representations {II}: applications
  to cluster algebras.
\newblock {\em J. Amer. Math. Soc.}, 23(3):749--790, 2010.

\bibitem{DlabRingelFinite}
V.~Dlab and C.~M. Ringel.
\newblock {\em On algebras of finite representation type}.
\newblock Department of Mathematics, Carleton Univ., Ottawa, Ont., 1973.
\newblock Carleton Mathematical Lecture Notes, No. 2.

\bibitem{DTOdd}
D.~D. Duffield and P.~Tumarkin.
\newblock Categorification of non-integer quivers: Types {$H_4$}, {$H_3$} and
  {$I_2(2n+1)$}.
\newblock {\em Represent. Theor.}, 28:275--327, 2024.

\bibitem{EGNO}
P.~Etingof, S.~Gelaki, D.~Nikshych, and V.~Ostrik.
\newblock {\em Tensor categories}, volume 205 of {\em Mathematical Surveys and
  Monographs}.
\newblock American Mathematical Society (AMS), 2015.

\bibitem{FST3}
A.~Felikson, M.~Shapiro, and P.~Tumarkin.
\newblock Cluster algebras and triangulated orbifolds.
\newblock {\em Adv. Math.}, 231(5):2953--3002, 2012.

\bibitem{FST2}
A.~Felikson, M.~Shapiro, and P.~Tumarkin.
\newblock Cluster algebras of finite mutation type via unfoldings.
\newblock {\em Int. Math. Res. Not. IMRN}, (8):1768--1804, 2012.

\bibitem{FominZelevinskyIV}
S.~Fomin and A.~Zelevinsky.
\newblock Cluster algebras. {IV}. {C}oefficients.
\newblock {\em Compos. Math.}, 143(1):112--164, 2007.

\bibitem{Gabriel}
P.~Gabriel.
\newblock Unzerlegbare {D}arstellungen. {I}.
\newblock {\em Manuscripta Math.}, 6:71--103; correction, ibid. 6 (1972), 309,
  1972.

\bibitem{Heng}
E.~Heng.
\newblock Coxeter quiver representations in fusion categories and {G}abriel’s
  theorem.
\newblock {\em Sel. Math. New Ser.}, 30(4):67, 2024.

\bibitem{KellerOrbitCats}
B.~Keller.
\newblock On triangulated orbit categories.
\newblock {\em Doc. Math.}, 10:551--581, 2005.

\bibitem{KellerReiten}
B.~Keller and I.~Reiten.
\newblock Cluster-tilted algebras are {G}orenstein and stably {C}alabi-{Y}au.
\newblock {\em Adv. Math.}, 211(1):123--151, 2007.

\bibitem{Kirillov}
A.~Kirillov and V.~Ostrik.
\newblock On a $q$-analogue of the {M}c{K}ay correspondence and the {$ADE$}
  classification of $\mathfrak{sl}_2$ conformal field theories.
\newblock {\em Advances in Mathematics}, 171(2):183--227, 2002.

\bibitem{Lusztig}
G.~Lusztig.
\newblock Some examples of square integrable representations of semisimple
  {$p$}-adic groups.
\newblock {\em Trans. Amer. Math. Soc.}, 277(2):623--653, 1983.

\bibitem{MoodyPatera}
R.~V. Moody and J.~Patera.
\newblock Quasicrystals and icosians.
\newblock {\em J. Phys. A}, 26(12):2829--2853, 1993.

\bibitem{Muhlherr}
B.~M\"{u}hlherr.
\newblock Coxeter groups in {C}oxeter groups.
\newblock In {\em Finite geometry and combinatorics ({D}einze, 1992)}, volume
  191 of {\em London Math. Soc. Lecture Note Ser.}, pages 277--287. Cambridge
  Univ. Press, Cambridge, 1993.

\bibitem{Nagao}
K.~Nagao.
\newblock Donaldson-{T}homas theory and cluster algebras.
\newblock {\em Duke Math. J.}, 162(7):1313--1367, 2013.

\bibitem{NZ}
T.~Nakanishi and A.~Zelevinsky.
\newblock On tropical dualities in cluster algebras.
\newblock 565:217--226, 2012.

\bibitem{Plamondon}
P.-G. Plamondon.
\newblock Cluster algebras via cluster categories with infinite-dimensional
  morphism spaces.
\newblock {\em Compos. Math.}, 147(6):1921--1954, 2011.

\bibitem{Reading}
N.~Reading.
\newblock Universal geometric cluster algebras.
\newblock {\em Math. Z.}, 277(1-2):499--547, 2014.

\bibitem{WaveFrontsReflGroups}
O.~P. Shcherbak.
\newblock Wave fronts and reflection groups.
\newblock {\em Uspekhi Mat. Nauk}, 43(3(261)):125--160, 271, 272, 1988.

\bibitem{SteinbergRefl}
R.~Steinberg.
\newblock Finite reflection groups.
\newblock {\em Transactions of the American Mathematical Society},
  91(3):493--504, 1959.

\bibitem{WatkinsZeitlin}
W.~Watkins and J.~Zeitlin.
\newblock The minimal polynomial of $\cos (2\pi/n)$.
\newblock {\em The American Mathematical Monthly}, 100(5):471--474, 1993.

\end{thebibliography}
	\bibliographystyle{habbrv}
\end{document}